\documentclass[11pt]{amsart}

\usepackage{amsmath, amssymb, amscd}

\usepackage{amsthm}
\usepackage{bm}

\usepackage{color}

\usepackage{graphicx}

\usepackage{indentfirst}

\makeatletter
\def\widebreve#1{\mathop{\vbox{\m@th\ialign{##\crcr\noalign{\kern\p@}%
  \brevefill\crcr\noalign{\kern0.1\p@\nointerlineskip}%
  $\hfil\displaystyle{#1}\hfil$\crcr}}}\limits}
\makeatother

\makeatletter
\def\brevefill{$\m@th \setbox\z@\hbox{}%
 \hfill\scalebox{0.7}{\rotatebox[origin=c]{90}{(}} \kern4pt $}
\makeatother

 \makeatletter
    
    \@addtoreset{equation}{section}
  \makeatother

\theoremstyle{plain}
\newtheorem{thm}{Theorem}[subsection]

\newtheorem{lem}[thm]{Lemma}
\newtheorem{prop}[thm]{Proposition}
\newtheorem{cor}[thm]{Corollary}

\theoremstyle{remark}
\newtheorem{rem}[thm]{Remark}

\theoremstyle{definition}
\newtheorem{dfn}[thm]{Definition}

\theoremstyle{plain}
\newtheorem{theor}{Theorem}%[subsection]

\newtheorem{nclaim}{Claim}

\newcommand{\QLS}{{\rm QLS}}
\newcommand{\B}{{\rm B}}
\newcommand{\QB}{{\rm QB}}
\newcommand{\EQB}{{\rm EQB}}
\newcommand{\QBG}{\mathrm{QBG}}
\newcommand{\qwt}{{\rm qwt}}
\newcommand{\ed}{{\rm end}}
\newcommand{\id}{e}
\newcommand{\gch}{\mathop{\rm gch}\nolimits}
\newcommand{\wt}{\mathrm{wt}}
\newcommand{\dg}{\mathrm{deg}}

\newcommand{\dir}{\mathrm{dir}}

\newcommand{\OS}{{\rm OS}}
\newcommand{\aff}{\mathrm{aff}}

\newcommand{\ext}{{\rm ext}}
\newcommand{\eqdef}{:=}
\newcommand{\ardef}{\overset{\rm def}{\Leftrightarrow}}
\newcommand{\bqed}{\quad \hbox{\rule[-0.5pt]{6pt}{6pt}}  \vspace{3mm}}
\newcommand{\lon}{w_\circ}
\newcommand{\lons}{w_\circ (S)}

\newcommand{\rr}{\Delta_{\aff}}
\newcommand{\prr}{\Delta_{\aff}^{+}}
\newcommand{\trr}{\widetilde{\Delta}_{\aff}}
\newcommand{\ptrr}{\widetilde{\Delta}_{\aff}^{+}}
\newcommand{\ntrr}{\widetilde{\Delta}_{\aff}^{-}}

\newenvironment{enu}{%
 \begin{enumerate}%
}{\end{enumerate}}

\newcommand{\Fg}{\mathfrak{g}}
\newcommand{\Fh}{\mathfrak{h}}

\newcommand{\BZ}{\mathbb{Z}}

\newcommand{\BC}{\mathbb{C}}

\newcommand{\CB}{\mathcal{B}}
\newcommand{\CL}{\mathcal{L}}

\newcommand{\vpi}{\varpi}

\newcommand{\bp}{\mathbf{p}}
\newcommand{\bq}{\mathbf{q}}
\newcommand{\brho}{\bm{\rho}}
\newcommand{\bvrho}{\bm{\varrho}}

\newcommand{\Hom}{\mathrm{Hom}}
\newcommand{\GL}{\mathrm{GL}}
\newcommand{\Par}{\mathrm{Par}}
\newcommand{\Conn}{\mathrm{Conn}}
 
\newcommand{\qad}{U_{q_s}(\Fg_\aff)}
\newcommand{\qa}{U'_{q_s}(\Fg_\aff)}

\newcommand{\norm}{\mathrm{norm}}
\newcommand{\cl}{\mathrm{cl}}

\newcommand{\Img}{\mathop{\rm Img}\nolimits}
\newcommand{\Span}{\mathop{\rm Span}\nolimits}
\newcommand{\ch}{\mathop{\rm ch}\nolimits}

\newcommand{\PJ}{\Pi^{S}}
\newcommand{\Qad}{Q^{\vee,\,\text{\rm $S$-ad}}}

\newcommand{\SBG}{\mathrm{SiBG}}
\newcommand{\SBa}{\mathrm{SiBG}(\lambda\,;\,\sigma)}
\newcommand{\SBx}[1]{\mathrm{SiBG}(\lambda\,;\,#1)}

\newcommand{\SLS}{\mathbb{B}^{\frac{\infty}{2}}}

\newcommand{\sell}{\ell^{\frac{\infty}{2}}}
\newcommand{\sil}{\prec}
\newcommand{\sile}{\preceq}
\newcommand{\sig}{\succ}
\newcommand{\sige}{\succeq}

\newcommand{\edge}[1]{ \xrightarrow{\hspace{3pt}#1\hspace{3pt}} }
\newcommand{\mcr}[1]{\lfloor #1 \rfloor}

\newcommand{\pair}[2]{\langle #1,\,#2 \rangle}
\newcommand{\Bpair}[2]{\left\langle #1,\,#2 \right\rangle}

\newcommand{\ol}[1]{\overline{#1}}
\newcommand{\ti}[1]{\widetilde{#1}}
\newcommand{\sss}{s}

\title[Formula for specialized nonsymmetric Macdonald polynomials]{
Specialization of nonsymmetric Macdonald polynomials at $t=\infty$
and Demazure submodules of level-zero extremal weight modules$^\ast$}
\author[S.~Naito]{Satoshi Naito}
\address[Satoshi Naito]
 {Department of Mathematics, Tokyo Institute of Technology,
  2-12-1 Oh-Okayama, Meguro-ku, Tokyo 152-8551, Japan}
\email{naito@math.titech.ac.jp}

\author[F.~Nomoto]{Fumihiko Nomoto}
\address[Fumihiko Nomoto]
 {Department of Mathematics, Tokyo Institute of Technology,
  2-12-1 Oh-Okayama, Meguro-ku, Tokyo 152-8551, Japan}
\email{nomoto.f.aa@m.titech.ac.jp}

\author[D.~Sagaki]{Daisuke Sagaki}
\address[Daisuke Sagaki]
 {Institute of Mathematics, University of Tsukuba, 
  Tsukuba, Ibaraki 305-8571, Japan}
\email{sagaki@math.tsukuba.ac.jp}

\begin{document}
\maketitle

\begin{abstract}
In this paper, we give a representation-theoretic interpretation
of the specialization $E_{\lon \lambda} (q,\infty)$ of the nonsymmetric Macdonald polynomial $E_{\lon \lambda}(q,t)$
at $t=\infty$ in terms of the Demazure submodule $V_{\lon}^- (\lambda)$ of the level-zero extremal weight module $V(\lambda)$
over a quantum affine algebra of an arbitrary untwisted type;
here, $\lambda$ is a dominant integral weight,
and $\lon$ denotes the longest element in the finite Weyl group $W$.
Also, for each $x \in W$, we obtain a combinatorial formula for the specialization $E_{x \lambda} (q, \infty)$ at $t=\infty$ of the nonsymmetric Macdonald polynomial $E_{x \lambda} (q,t)$,
and also one for the graded character $\gch V_{x}^- (\lambda)$ 
 of the Demazure submodule
$V_{x}^- (\lambda)$ of $V(\lambda)$; both of these formulas are described in terms of quantum Lakshmibai-Seshadri paths
of shape $\lambda$.

\end{abstract}
%%%%%%%%%%%
%%%%%%%%%%%
%MSC 
%%%%%%%%%%%
%%%%%%%%%%%

\section{Introduction}
\renewcommand{\thefootnote}{\fnsymbol{footnote}}
\footnote[0]{%%%%%%%%%%%%11
$\ast$ This paper is a revised version of our earlier one 
``An explicit formula for the specialization of nonsymmetric Macdonald polynomials at $t=\infty$ (arXiv:1511.07005)'',
in which Theorem \ref{main} of this paper was the main result.}
%\footnote[0]{Mathematics Subject Classification: Primary 05E05; Secondary 33D52, 17B37, 20G42.}
%
%\footnote{In the Introduction, comment further on the PBW filtration conjecture of \cite{CO}, in particular how the methods of this paper give a different (but more tractable) representation-theoretic interpretation.}
Symmetric Macdonald polynomials %(in affine type A) 
with two parameters $q$ and $t$ were introduced 
by Macdonald
 \cite{M1} as a family of orthogonal symmetric polynomials,
which include as special or limiting cases almost all the classical families of orthogonal symmetric polynomials.
This family of polynomials are characterized in terms of the double affine Hecke algebra (DAHA) introduced by Cherednik (\cite{C1}, \cite{C3}).
In fact, there exists another family of orthogonal polynomials, called nonsymmetric Macdonald polynomials, which are simultaneous eigenfunctions of $Y$-operators 
acting on the polynomial representation of the DAHA; 
by 
``symmetrizing''
nonsymmetric Macdonald polynomials, we obtain symmetric Macdonald polynomials (see \cite{M}).

%%%%%%%%%%%%1
Based on the characterization above of nonsymmetric Macdonald polynomials,
Ram-Yip \cite{RY} obtained a combinatorial formula expressing symmetric or nonsymmetric Macdonald polynomials associated to an arbitrary untwisted affine
root system; this formula is described in terms of alcove walks, which are certain strictly combinatorial objects.
%%%%%%%%%%%%1
In addition, Orr-Shimozono \cite{OS} refined the Ram-Yip formula above,
and generalized it to an arbitrary affine root system
(including the twisted case);
also, they specialized their formula at $t=0$, $t=\infty$, $q=0$, and $q=\infty$.

As for representation-theoretic interpretations of the specialization of symmetric or nonsymmetric Macdonald polynomials at $t=0$,
%
%
%
%
%
%
%
%Concerning the second question, 
we know the following.
Ion \cite{I} proved that for a dominant integral weight $\lambda$ and an element $x$ of the finite Weyl group $W$, 
%%%%%%%%%%%%2
the specialization $E_{x \lambda} (q, 0)$ of the nonsymmetric Macdonald polynomial $E_{x\lambda} (q,t)$ at $t=0$
is equal to
%%%%%%%%%%%%2
the graded character of a certain Demazure submodule of an irreducible highest weight module over
an affine Lie algebra of a dual untwisted type. %equals the specialization $E_{x\lambda}(q, 0)$ of the nonsymmetric Macdonald polynomial $E_{x\lambda}(q, t)$ at $t=0$.
Afterward,
Lenart-Naito-Sagaki-Schilling-Shimozono \cite{LNSSS2} proved that for a dominant integral weight $\lambda$, the set 
$\QLS(\lambda)$
of all quantum Lakshmibai-Seshadri (QLS) paths of shape $\lambda$ provides a realization of the crystal basis of a special quantum Weyl module %$W_q (\lambda)$ 
over a quantum affine algebra $\qa$ (without degree operator) of an arbitrary untwisted type, 
and that its graded character equals the specialization $E_{\lon \lambda}(q,0)$ at $t=0$, where $\lon$ denotes the longest element of $W$.
Here a QLS path is obtained from an affine level-zero Lakshmibai-Seshadri path
through the projection $\mathbb{R}\otimes_\mathbb{Z} P_{\aff} \rightarrow \mathbb{R}\otimes_\mathbb{Z} P$,
which factors the null root
$\delta$ of an affine Lie algebra $\mathfrak{g}_\aff$,
and is described in terms of (the parabolic version of) the quantum Bruhat graph, introduced by Brenti-Fomin-Postnikov \cite{BFP};
the set of QLS paths 
is endowed with an affine crystal structure in a way
similar to
the one for the set of ordinary LS paths introduced by Littelmann \cite{L1}.
Moreover,
Lenart-Naito-Sagaki-Schilling-Shimozono \cite{LNSSS3}
 obtained a formula for the specialization $E_{x\lambda}(q, 0)$,
$x \in W$, at $t=0$ in an arbitrary untwisted affine type, which is described in terms of $\QLS$ paths of shape $\lambda$, and proved that the specialization $E_{x\lambda}(q, 0)$ is just the 
graded character of a certain Demazure-type submodule of the special quantum Weyl module.
%%%%%%%%%%%%3
The crucial ingredient in the proof of this result is a graded character formula obtained in \cite{NS-D}
for the Demazure submodule $V_e^- (\lambda)$ of the level-zero extremal weight module $V(\lambda)$
of extremal weight $\lambda$ over a quantum affine algebra $\qad$,
% associated to $\mathfrak{g}_\aff$.
where  $e$ is the identity element of $W$.
More precisely, Naito and Sagaki proved that the graded character $\gch V_e^- (\lambda)$ of $V_e^- (\lambda) \subset V (\lambda)$
is equal to 
$
\left(
\prod_{i \in I}
\prod_{r=1}^{m_i}(1-q^{-r})
\right)^{-1}
E_{\lon \lambda}(q^{-1},0)
$,
where
$\lambda = \sum_{i \in I} m_i \varpi_i $ is a dominant integral weight,
with $\varpi_i$, $i \in I$, the fundamental weights; the graded character $\gch V_e^- (\lambda)$ is obtained from the ordinary character of $V_e^- (\lambda)$ by replacing $e^\delta$ by $q$, with $\delta$ the null root of the affine Lie albegra $\mathfrak{g}_\aff$.
%%%%%%%%%%%%3

%%%%%%%%%%%%4
The purpose of this paper is to give a representation-theoretic interpretation of the specialization
$E_{\lon \lambda} (q, \infty)$ of the nonsymmetric Macdonald polynomial
$E_{\lon \lambda} (q, t)$ at $t= \infty$ in terms of the Demazure submodule $V_{\lon}^- (\lambda)$ of $V(\lambda)$;
here we remark that $V_{\lon}^- (\lambda) \subset V_{e}^- (\lambda) $.
More precisely, we prove the following theorem.
\vspace{4mm}
\begin{theor}[{$=$ Theorem \ref{thm:demazure_character}}]\label{TheoremA}
Let $\lambda = \sum_{i \in I} m_i \varpi_i $ be a dominant integral weight.
Then, the graded character $\gch V_{\lon}^- (\lambda)$ of the Demazure submodule $V_{\lon}^- (\lambda)$ of $V(\lambda)$ is equal to 
%%%%%%%%%%%%
\begin{equation*}
\left(
\prod_{i \in I}
\prod_{r=1}^{m_i}(1-q^{-r})
\right)^{-1}
E_{\lon \lambda}(q,\infty)
.
\end{equation*}
\end{theor}
\vspace{4mm}
%%%%%%%%%%%%4
%%%%%%%%%%%%5
In order to prove Theorem~\ref{TheoremA}, we first rewrite the Orr-Shimozono formula for the specialization $E_{x \lambda} (q, \infty)$ for $x \in W$
(originally described in terms of quantum alcove walks)
in terms of $\QLS$ paths, by use of an explicit bijection sending quantum alcove walks to $\QLS$ paths
that preserves weights and degrees;
in some ways, this bijection generalizes a similar one in \cite{LNSSS2}.
In particular, for $x = \lon$, the Orr-Shimozono formula rewritten in terms of $\QLS$ paths states that
\begin{equation}\label{*}
E_{\lon \lambda}(q, \infty)
=
\sum_{\psi \in \QLS(\lambda)} e^{\wt(\psi)} q^{\dg_{\lon \lambda} (\psi)},
 \tag{$\ast$}
\end{equation}
where $\QLS(\lambda)$ is the set of all $\QLS$ paths of shape $\lambda$,
and for $\psi \in \QLS(\lambda)$, $\dg_{\lon \lambda} (\psi)$ is a certain nonpositive integer,
which is explicitly described in terms of the quantum Bruhat graph;
see \S 3.2 for details.
%%%%%%%%%%%%5

%%%%%%%%%%%%6
Next, using the explicit realization, obtained in \cite{INS},
of the crystal basis $\mathcal{B}(\lambda)$ of $V(\lambda)$ by semi-infinite LS paths of shape $\lambda$,
we compute the graded character $\gch V_x^- (\lambda)$ of the Demazure submodule $V_x^- (\lambda)$ for $x \in W$,
and prove  the following theorem.
\vspace{4mm}
\begin{theor}[{$=$ Theorem \ref{thm:gch}}]\label{TheoremB}
Let $\lambda = \sum_{i\in I}m_i \varpi_i$ be a dominant integral weight, and $x$ an element of the finite Weyl group $W$.
Then, the graded character $\gch V_{x}^- (\lambda)$ of $V_{x}^- (\lambda)$ is equal to 
\begin{equation*}
\left(
\prod_{i \in I}
\prod_{r=1}^{m_i}(1-q^{-r})
\right)^{-1}
\sum_{\psi \in \QLS(\lambda)} e^{\wt(\psi)} q^{\dg_{x \lambda} (\psi)}.
\end{equation*}
\end{theor}
\vspace{4mm}

The proof of Theorem~\ref{TheoremB} is based on the fact 
%Here we remark that
by factoring the null root $\delta$ of $\mathfrak{g}_\aff$, we obtain a surjective strict morphism of crystals from the set of all semi-infinite LS paths
%%%%%%%%%%%%6
%%%%%%%%%%%%7
of shape $\lambda$ onto $\QLS(\lambda)$.
By combining the special case $x = \lon$ of Theorem \ref{TheoremB} with equation (\ref{*}) above,
we obtain Theorem \ref{TheoremA}.

Finally, for $x \in W$, we define a certain (finite-dimensional) quotient module $V_{x}^- (\lambda) / X_{x}^- (\lambda)$
of $V_{x}^- (\lambda) \subset V(\lambda)$,
and prove that its graded character $\gch \left( V_{x}^- (\lambda) / X_{x}^- (\lambda) \right)$ is equal to 
$\sum_{\psi \in \QLS(\lambda)} e^{\wt(\psi)} q^{\dg_{x \lambda} (\psi)}$.
Hence it follows that under the specialization $e^\delta = q= 1$, all the modules $V_{x}^- (\lambda) / X_{x}^- (\lambda)$, $x \in W$,
have the same character;
in particular, they have the same dimension.
%%%%%%%%%%%%7
%%%%%%%%%%%%8
Also, in the case $x = \lon$, we have 
$\gch \left( V_{\lon}^- (\lambda) / X_{\lon}^- (\lambda) \right) = E_{\lon \lambda}(q, \infty)$;
note that in the case $x=e$, the quotient module $ V_{e}^- (\lambda) / X_{e}^- (\lambda) $ is just the one in \cite[\S 7.2]{NS-D},
and hence we have $\gch \left( V_{e}^- (\lambda) / X_{e}^- (\lambda) \right) = E_{\lon \lambda}(q^{-1}, 0)$
(see \cite[\S 3]{LNSSS3} and \cite[\S 6.4]{NS-D}).
Based on these results
together with \cite[Theorem 5.1]{Kat} for the classical limit,
we can think of
the quotient modules $V_{x}^- (\lambda) / X_{x}^- (\lambda)$, $x \in W$, as a quantum analog of
``generalized Weyl modules'' introduced in \cite{FM}.
%%%%%%%%%%%%8

%%%%%%%%%%%%9
This paper is organized as follows.
In Section 2, we fix our notation, and recall some basic facts about the (parabolic) quantum Bruhat graph.
Also, we briefly review the Orr-Shimozono formula for the specialization $E_{x \lambda} (q, \infty)$ at $t=\infty$ for $x \in W$.
In Section 3, we prove equation (\ref{*}) above, or more generally Theorem \ref{main};
this theorem gives the description of the specialization $E_{x \lambda}(q, \infty)$ at $t=\infty$ for $x \in W$ in terms of $\QLS$ paths of shape $\lambda$.
In Section 4, we compute the graded character $\gch V^-_{x} (\lambda)$ for an arbitrary $x \in W$, and prove Theorem \ref{TheoremB};
by combining the special case $x = \lon$ of Theorem \ref{TheoremB} with equation (\ref{*}), we obtain Theorem \ref{TheoremA}.
Finally, for $x \in W$, we define a certain (finite-dimensional) quotient module $V_{x}^- (\lambda) / X_{x}^- (\lambda)$ of $V_{x}^- (\lambda)$,
and compute its graded character;
%%%%%%%%%%%%9
%%%%%%%%%%%%10
in the special case $x = \lon$, we obtain the equality $\gch \left(V_{\lon}^- (\lambda) / X_{\lon}^- (\lambda) \right) = E_{\lon \lambda}(q, \infty)$.
%%%%%%%%%%%%10

%%%%%%%%%%%%%11
%This paper is a revised version of our earlier one 
%``An explicit formula for the specialization of nonsymmetric Macdonald polynomials at $t=\infty$ (arXiv:1511.07005)'',
%in which Theorem \ref{main} of this paper is the main result.

\vspace{6mm}
{\bf Acknowledgments}: The authors would
like to thank Syu Kato for sending us his preprint \cite{Kat}.

\section{(Parabolic) Quantum Bruhat Graph and Orr-Shimozono Formula}
\subsection{(Parabolic) quantum Bruhat graph}
Let $\mathfrak{g}$ be a finite-dimensional simple Lie algebra over $\mathbb{C}$,
$I$ the vertex set for the Dynkin diagram of  $\mathfrak{g}$,
$\{ \alpha_i \}_{i \in I }$
(resp., $\{ {\alpha}^{\lor}_i \}_{i \in I }$)
 the set of all simple roots (resp., coroots) of  $\mathfrak{g}$,
$\mathfrak{h} = \bigoplus_{i \in I}\mathbb{C}\alpha^\lor_i$ a Cartan subalgebra of  $\mathfrak{g}$,
$\mathfrak{h}^* = \bigoplus_{i \in I}\mathbb{C}\alpha_i$ the dual space of $\mathfrak{h}$,
and $\mathfrak{h}_\mathbb{R}^* = \bigoplus_{i \in I}\mathbb{R}\alpha_i$ the real form of $\mathfrak{h}^*$;
the duality pairing between  $\mathfrak{h}$ and  $\mathfrak{h}^*$ is denoted by
$\langle \cdot, \cdot \rangle : \mathfrak{h}^* \times \mathfrak{h} \rightarrow \mathbb{C}$.
Let $Q = \sum_{i \in I}\mathbb{Z}\alpha_i  
\subset \mathfrak{h}_\mathbb{R}^*$ denote 
the root lattice of $\mathfrak{g}$,
$Q^\lor = \sum_{i \in I}\mathbb{Z}\alpha_i^\lor  
\subset \mathfrak{h}_\mathbb{R}$ 
the coroot lattice of $\mathfrak{g}$,
and 
$P = \sum_{i \in I}\mathbb{Z}\varpi_i \subset \mathfrak{h}_\mathbb{R}^*$  the weight lattice of $\mathfrak{g}$,
where the $\varpi_i$, $i \in I$, are the fundamental weights for $\mathfrak{g}$,
i.e., 
$\langle \varpi_i ,\alpha_j^\lor \rangle = \delta_{i j}$
for $j \in I$;
we set $P^+ \eqdef \sum_{i \in I} \BZ_{\geq 0} \varpi_i$, and call an elements $\lambda$ of $P^+$ a dominant weight.
Let us denote by $\Delta$ the set of all roots,
and by $\Delta^{+}$ (resp., $\Delta^{-}$) the set of all positive (resp., negative) roots.
Also, let $W \eqdef \langle s_i \ | \ i \in I \rangle$
be the Weyl group of $\mathfrak{g}$,
where
$s_i $, $i \in I$, are the simple reflections acting on $\mathfrak{h}^*$ and on $\mathfrak{h}$:
\begin{align*}
s_i \nu = \nu - \langle  \nu , \alpha^\lor_i  \rangle \alpha_i, & \ \ \
\nu \in \mathfrak{h}^*,\\
s_i h = h - \langle \alpha_i , h \rangle \alpha^\lor_i,
& \ \ \
h \in \mathfrak{h};
\end{align*}
we denote the identity element and the longest element of $W$ by $e$ and $\lon$, respectively.
If $\alpha \in \Delta$ is written as $\alpha = w \alpha_i$ for 
$w\in W$ and $i \in I$, we define $\alpha^\lor$ to be $w \alpha^\lor_i$;
we often identify $s_\alpha$ with $s_{\alpha^\lor}$.
For  $u \in W$,
the length of $u$ is denoted by $\ell(u)$,
 which equals
$\# (\Delta^+ \cap u^{-1}\Delta^-)$.

\begin{dfn}[\normalfont{\cite[Definition 6.1]{BFP}}]\label{QBG}
The quantum Bruhat graph, denoted by $\QBG$, is a directed graph with vertex set $W$ and directed edges 
 labeled by positive roots;
for $u,v \in W$, and $\beta \in \Delta^+$, 
an arrow $u \xrightarrow{\beta} v $ is an edge of $\QBG$
if the following hold{\rm :}

\begin{enu}
\item
$v=u s_\beta$, and

\item
either
(2a):
$\ell(v)=\ell(u)+1$ or
(2b):
$\ell(v)=\ell(u) - 2\langle \rho, \beta^\lor \rangle +1$, 
\end{enu}
where $\rho \eqdef \frac{1}{2} \sum_{\alpha \in \Delta^+}{\alpha}$.
An edge satisfying (2a) 
(resp., (2b))
is called a Bruhat (resp., quantum) edge.
\end{dfn}

%%%%%%%%%%%%%%%%%%%%%%%
%%%%%%%%%%%%%%%%%%%%%%%
%%1 ~rem:left_qbg
%%%%%%%%%%%%%%%%%%%%%%%
%%%%%%%%%%%%%%%%%%%%%%%
\begin{rem}
The quantum Bruhat graph defined above is a ``right-handed'' version, while 
the one defined in \cite{BFP} is a ``left-handed'' version.
We remark that the results of \cite{BFP} used in this paper (such as Proposition \ref{shellability}) are unaffected by this difference
(cf. \cite{Po}).
\end{rem}
%%%%%%%%%%%%%%%%%%%%%%%
%%%%%%%%%%%%%%%%%%%%%%%
%%end~rem:left_qbg
%%%%%%%%%%%%%%%%%%%%%%%
%%%%%%%%%%%%%%%%%%%%%%%

%%%%%%%%%%%%%%%%%%%%%%%
%%%%%%%%%%%%%%%%%%%%%%%
%%1 start
%%%%%%%%%%%%%%%%%%%%%%%
%%%%%%%%%%%%%%%%%%%%%%%
For an edge $u \xrightarrow{\beta} v $ of $\QBG$,
we set
\begin{equation*}
\wt (u \rightarrow v)
\eqdef
\left\{
\begin{array}{ll}
      0 & \mbox{if} \ u \xrightarrow{\beta} v \mbox{ is a Bruhat edge}, \\
      \beta^{\lor} &  \mbox{if} \ u \xrightarrow{\beta} v \mbox{ is a quantum edge}.
\end{array}
\right.
\end{equation*}
Also, for $u,v \in W$, we take a shortest directed path
$u=x_0 \xrightarrow{\gamma_1} x_1 \xrightarrow{\gamma_2}\cdots \xrightarrow{\gamma_r} x_r=v$ in $\QBG$,
and set
\begin{equation*}
\wt (u\Rightarrow v) \eqdef \wt (x_0 \rightarrow x_1)+\cdots +\wt (x_{r-1} \rightarrow x_r) \in Q^\lor;
\end{equation*}
%\begin{equation*}
%\wt_\lambda (u\Rightarrow v) \eqdef \wt_\lambda (x_0 \rightarrow x_1)+\cdots +\wt_\lambda (x_{r-1} \rightarrow x_r);
%\end{equation*}
we know from \cite[Lemma 1 (2), (3)]{Po} that
this definition does not depend on the choice of a shortest directed path from $u$ to $v$ in $\QBG$.
For a dominant weight $\lambda \in P^+$, we set $\wt_\lambda (u \Rightarrow v ) \eqdef  \pair{\lambda}{\wt (u\Rightarrow v)}$,
and call it the $\lambda$-weight of a directed path from $u$ to $v$ in $\QBG$.

%%%%%%%%%%%%%%%%%%%%%%%
%%%%%%%%%%%%%%%%%%%%%%%
%%1 end
%%%%%%%%%%%%%%%%%%%%%%%
%%%%%%%%%%%%%%%%%%%%%%%

\begin{lem}\label{involution}
If $x \xrightarrow{\beta} y$ is a Bruhat $($resp., quantum$)$ edge of $\QBG$, then $y\lon \xrightarrow{- \lon \beta} x \lon$ is also a Bruhat $($resp., quantum$)$ edge of $\QBG$.
\end{lem}

\begin{proof}
This follows easily from equalities
$\ell(y)-\ell(x)=\ell(x \lon )-\ell(y \lon )$ and $\langle \rho, - \lon \beta^\lor \rangle =\langle \rho, \beta^\lor \rangle $.
\end{proof}

Let $w \in W$. We take (and fix) reduced expressions $w = s_{i_{1}}\cdots s_{i_{p}}$ and $\lon w^{-1} =s_{i_{-q}}\cdots s_{i_{0}} $;
note that
	\begin{equation*}
		\lon =s_{i_{-q}}\cdots s_{i_{0}}s_{i_{1}}\cdots s_{i{p}}
	\end{equation*}
is also a reduced expression for the longest element $\lon$.
Now we set
	\begin{equation}\label{inversion_root}
		\beta_{k} \eqdef
	s_{i_{p}}\cdots s_{i_{k+1}} \alpha_{i_{k}},
\ \ \ 
-q \leq k \leq p;
	\end{equation}
we have $\{\beta_{-q}, \ldots , \beta_{0}, \ldots , \beta_{p} \}=\Delta^+$.
Then we define a total order
$\prec$ on $\Delta^+$ by
	\begin{equation}\label{reflectionorder}
		\beta_{-q} \prec \beta_{-q+1} \prec \cdots \prec \beta_{p};
	\end{equation}
note that this total order is a weak reflection order in the sense of Definition \ref{311}.

\begin{prop}[{\cite[Theorem 6.4]{BFP}}] \label{shellability}
	Let $u$ and $v$ be elements in $W$.

\begin{enu}
\item
There exists a unique directed path from $u$ to $v$ in $\QBG$ 
for which the edge labels are strictly increasing {\rm(}resp., strictly decreasing{\rm)} in the total order $\prec$ above.

\item
The unique label-increasing {\rm(}resp., label-decreasing{\rm)} path
\begin{equation*}
u = u_0
\xrightarrow{\gamma_1}
u_1
\xrightarrow{\gamma_2}
\cdots
\xrightarrow{\gamma_r}
u_r=v
\end{equation*}
 from $u$ to $v$ in $\QBG$
is a shortest directed path from $u$ to $v$.
Moreover, it is lexicographically minimal {\rm(}resp., lexicographically maximal{\rm)} among all shortest directed paths from $u$ to $v$;
that is, for an arbitrary shortest directed path
\begin{equation*}
u = u'_0
\xrightarrow{\gamma'_1}
u'_1
\xrightarrow{\gamma'_2}
\cdots
\xrightarrow{\gamma'_r}
u'_r=v
\end{equation*}
from $u$ to $v$ in $\QBG$, there exists $1 \leq j \leq r$
such that $\gamma_j \prec \gamma'_j$ {\rm(}resp., $\gamma_j \succ \gamma'_j${\rm)},
and $\gamma_k = \gamma'_k$ for $1 \leq k \leq j-1$.
\end{enu}
	\end{prop}

For a subset $S \subset I$,
we set
$W_S \eqdef \langle s_i \ | \ i \in S \rangle$;
notice that $S$ may be an empty set $\emptyset$.
We denote the longest element of $W_S$ by $\lons$.
Also, we set $\Delta_S \eqdef Q_S \cap \Delta$,
where $Q_S \eqdef \sum_{i \in S} \mathbb{Z}\alpha_i$, 
and then $\Delta_S^+ \eqdef \Delta_S \cap \Delta^+ $, 
$\Delta_S^- \eqdef \Delta_S \cap \Delta^- $.
Let $W^S$ denote the set of all minimal-length coset representatives for the cosets in $W / W_S$.
For $w\in W$, we denote  the minimal-length coset representative of the coset $w W_S$ by 
$\lfloor w \rfloor$, and
for a subset $U \subset W$, we set $\lfloor U \rfloor \eqdef \{ \lfloor w \rfloor \ | \ w \in U \} \subset W^S$.
%Also,
%for $\mu \in W \lambda$, we denote by $v(\mu) \in W^S$ the minimal-length coset representative for the coset $\{ w \in W  \ | \ w\lambda=\mu\}$
%
\begin{dfn}[{\cite[\S 4.3]{LNSSS1}}]\label{QBGS}
The parabolic quantum Bruhat graph, denoted by $\QBG^S$, is a directed graph with vertex set $W^S$, and directed edges
labeled by 
positive roots in
$ \Delta^+ \setminus \Delta^+_S$;
for $u,v \in W^S$, 
and $\beta \in \Delta^+ \setminus \Delta^+_S$,
an arrow $u \xrightarrow{\beta} v $ is an edge of $\QBG^S$
if the following hold{\rm :}
\begin{enu}
\item
$v=\lfloor u s_\beta \rfloor$, and 
\item
either
{\rm (2a):} 
$\ell(v)=\ell(u)+1$ or
{\rm (2b):} $\ell(v)=\ell(u) - 2\langle \rho - \rho_S , \beta^\lor \rangle +1$, 
\end{enu}
where $\rho_S =\frac{1}{2}\sum_{\alpha \in \Delta^+_S}\alpha$.
An edge satisfying {\rm (2a)} (resp., {\rm (2b)}) is called a Bruhat (resp., quantum) edge.
\end{dfn}

%%%%%%%%%%%%%%
%%%%%%%%%%%%%%
%begin 1
%%%%%%%%%%%%%%
%%%%%%%%%%%%%%
For an edge $u \xrightarrow{\beta} v $ in $\QBG^S$, we set
\begin{equation*}
\wt^S (u \rightarrow v)
\eqdef
\left\{
\begin{array}{ll}
      0 & \mbox{if} \ u \xrightarrow{\beta} v \mbox{ is a Bruhat edge}, \\
      \beta^{\lor} &  \mbox{if} \ u \xrightarrow{\beta} v \mbox{ is a quantum edge}.
\end{array}
\right.
\end{equation*}
Also, for $u, v \in W^S$, we take a shortest directed path
$\bp:u=x_0 \xrightarrow{\gamma_1} x_1 \xrightarrow{\gamma_2}\cdots \xrightarrow{\gamma_r} x_r=v$ in $\QBG^S$
(such a path always exists by \cite[Lemma 6.12]{LNSSS1}),
and set
\begin{equation*}
\wt^S (\bp) \eqdef \wt^S (x_0 \rightarrow x_1)+\cdots +\wt^S (x_{r-1} \rightarrow x_r) \in Q^\lor;
\end{equation*}
we know from \cite[Proppsition 8.1]{LNSSS1} that if $\bq$ is another shortest directed path from $u$ to $v$ in $\QBG^S$,
then $\wt^S(\bp) - \wt^S(\bq) \in Q_S^\lor \eqdef \sum_{i \in S} \BZ_{\geq 0} \alpha^\lor_i$. 
%%%%%%%%%%%%%%
%%%%%%%%%%%%%%
%end 1
%%%%%%%%%%%%%%
%%%%%%%%%%%%%%

%\subsection{Relation between $\QBG$ and  $\QBG^S$}\footnote{タイトル不適切?}
Now, we take and fix an arbitrary dominant weight
$\lambda \in P^+$,
%$\lambda \in P$, i.e., 
%$\langle \lambda , \alpha^{\lor}_i \rangle \geq 0$
%for all $i \in I$.
and set 
\begin{equation*}
S = S_\lambda \eqdef \{ i \in I \ | \  \langle \lambda , \alpha^{\lor}_i \rangle =0 \}
% \subset I
.
\end{equation*}
%%%%%%%%%%%%%%
%%%%%%%%%%%%%%
%begin 2
%%%%%%%%%%%%%%
%%%%%%%%%%%%%%
By the remark just above,
for $u,v \in W^S$,
the value $\pair{\lambda}{\wt^S(\bp)}$ does not depend on the choice of a shortest directed path $\bp$
from $u$ to $v$ in $\QBG^S$;
this value is called the $\lambda$-weight of a directed path from $u$ to $v$ in $\QBG^S$.
Moreover,
we know from \cite[Lemma 7.2]{LNSSS2}
that the value $\pair{\lambda}{\wt^S(\bp)}$ is equal to the value 
$\wt_\lambda (x \Rightarrow y) =  \pair{\lambda}{\wt (x \Rightarrow y) }$ for all $x \in uW_S$ and $y \in vW_S$.
In view of this, for $u, v \in W^S$, we write $\wt_\lambda (u \Rightarrow v)$ also for the value $\pair{\lambda}{\wt^S(\bp)}$
by abuse of notation.
%%%%%%%%%%%%%%
%%%%%%%%%%%%%%
%end 2
%%%%%%%%%%%%%%
%%%%%%%%%%%%%%

\begin{dfn}[{\cite[\S 3.2]{LNSSS2}}]
Let $\lambda \in P^+$ be a dominant weight and 
$\sigma \in \mathbb{Q}\cap [0,1]$,
and set $S=S_\lambda$.
We denote by
$\QBG_{\sigma \lambda}$ (resp., $\QBG^S_{\sigma \lambda}$ ) 
 the subgraph of $\QBG$ (resp., $\QBG^S$)
with the same vertex set but having only the edges:
$u \xrightarrow{\beta} v$ with $\sigma \langle \lambda, \beta^{\lor}  \rangle \in \mathbb{Z}$.
\end{dfn}

\begin{lem}[{\cite[Lemma 6.1]{LNSSS2}}]\label{8.1}
%\mbox{}
%%
%\begin{enu}
%\item
Let $\sigma \in \mathbb{Q}\cap [0,1]$; notice that $\sigma$ may be $1$.
If $u \xrightarrow{\beta} v$ is an edge of $\QBG_{\sigma \lambda}$,
then
there exists a  directed path from $\lfloor u \rfloor$ to $\lfloor v \rfloor$ in $\QBG^S_{\sigma \lambda}$.
%
%\item
%Let $u, v \in W^S$. %, and $x \in u W_S$, $y \in v W_S$.
%Then, the element
%\begin{equation*}
%\wt_\lambda(u \Rightarrow v) =
%\wt_\lambda(u_0 \rightarrow u_1)+ \cdots + \wt_\lambda(u_{r-1} \rightarrow u_r)
%\end{equation*}
%does not depend on 
%the choice of a shortest directed path
%$
%u = u_0 \rightarrow \cdots \rightarrow  u_r =v
%$  from $u$ to $v$ in $\QBG^S$;
%this element is called the $\lambda$-weight of a directed path in $\QBG^S$
%from $u$ to $v$.
%Moreover, the $\lambda$-weight $\wt_\lambda(u \Rightarrow v)$
%is equal to 
%the value  $\pair{\lambda}{\wt(x \Rightarrow y)}$ for all $x \in u W_S$ and $y \in v W_S$.
%\end{enu}
\end{lem}
%Therefore, by abuse of notation, we write  $\wt_\lambda (x \Rightarrow y)$ in stead of $\pair{\lambda}{\wt(x \Rightarrow y)}$
%also for arbitrary $x, y \in W$.

%Therefore, in what follows,
%we write $\wt_\lambda(u \rightarrow v)$, and $\wt_\lambda(u \Rightarrow v)$ instead of $\wt_\lambda^T(u \rightarrow v)$, and $\wt_\lambda^T(u \Rightarrow v)$,
%respectively.

Also, for $u,v\in W$, 
let $\ell(u\Rightarrow v)$ denote
the length of a shortest directed path in $\QBG$ from $u$ to $v$.
For $w\in W$,
as in \cite{BFP},
we define the $w$-tilted Bruhat order $\leq_{w}$ on $W$ as follows:
for $u,v \in W$,
	\begin{equation*}
		u \leq_{w} v
		\ardef
		\ell(w\Rightarrow v) =\ell(w\Rightarrow u)+\ell(u\Rightarrow v).
	\end{equation*}
We remark that the $w$-tilted Bruhat order on $W$ is a partial order
with the unique minimal element $w$.

\begin{lem}[{\cite[Theorem 7.1]{LNSSS1}}, {\cite[Lemma 6.5]{LNSSS2}}]\label{8.5}
Let $u, v\in W^S$, and $w \in W_S$.

\begin{enu}
\item
There exists a unique minimal element in the coset $v W_S$
in the $u w$-tilted Bruhat order $<_{u w}$.
We denote it by $\min(v W_S, <_{u w})$.

\item
There exists a unique  directed path from $u w$ to some $x \in v W_S$ 
in $\QBG$ whose edge labels are increasing 
in the 
total order $\prec$ on $\Delta^+$,
defined in $(\ref{reflectionorder})$,
and lie in $\Delta^+ \setminus \Delta^+_S$.
This path ends with $\min(v W_S, <_{u w})$.

\item
Let
$\sigma \in\mathbb{Q}\cap [0,1]$, 
and
$\lambda \in P$  a dominant weight.
If there exists a directed path from $u$ to $v$ in $\QBG_{\sigma \lambda}^S$,
then the directed path in $(2)$ is in  $\QBG_{\sigma \lambda}$.
\end{enu}
\end{lem}

\subsection{Orr-Shimozono formula}
In this subsection,
we review a formula \cite[Proposition 5.4]{OS} for the specialization of nonsymmetric Macdonald polynomials at $t= \infty$.

Let
$\widetilde{ \mathfrak{g}}$ denote the  finite-dimensional simple Lie algebra
whose root datum is dual to that of $\mathfrak{g}$; the set of simple roots is
$\{ {\alpha}^{\lor}_i \}_{i \in I } \subset \mathfrak{h}$,
and the set of simple coroots is
$\{ \alpha_i \}_{i \in I } \subset \mathfrak{h}^*$;
%the Cartan subalgebra of $\widetilde{ \mathfrak{g}}$
%is $\mathfrak{h}^*$, and its dual space is $\mathfrak{h}$.
We denote the set of all roots of $\widetilde{ \mathfrak{g}}$ by $\widetilde{\Delta} = \{ \alpha^\lor \ | \ \alpha \in \Delta \}$,
and 
the set of all positive (resp., negative) roots of $\widetilde{ \mathfrak{g}}$ by
$\widetilde{\Delta}^{+}$ (resp., $\widetilde{\Delta}^{-}$).
Also, for 
a subset $S \subset I$,
%the fixed dominant weight $\lambda \in P$,
we set
$\widetilde{Q}_S \eqdef \sum_{i \in S}\mathbb{Z}\alpha^\lor_i$,
$\widetilde{\Delta}_S \eqdef \widetilde{\Delta} \cap \widetilde{Q}_S$,
$\widetilde{\Delta}^+_S = \widetilde{\Delta}_S \cap \widetilde{\Delta}^{+}$,
and
$\widetilde{\Delta}^-_S = \widetilde{\Delta}_S \cap \widetilde{\Delta}^{-}$.

We consider the untwisted affinization of the root datum of $\widetilde{ \mathfrak{g}}$.
Let us denote
by $\trr$ 
the set of all real roots,
and 
by $\ptrr$
(resp., $\ntrr$)
the set of all positive (resp., negative) real roots.
Then we have
$\trr=
 \{ \alpha^\lor +a \widetilde{\delta} \ | \ \alpha \in \Delta , a \in \mathbb{Z} \}$,
with  $\widetilde{\delta}$ the null root.
We set
$\alpha^{\lor}_0 \eqdef \widetilde{\delta} - \varphi^{\lor} $,
where $\varphi \in \Delta$ denotes the highest short root, 
and set
$ I_{\aff}\eqdef I\sqcup \{0 \}$.
Then, $\{ \alpha^{\lor}_i \}_{i\in  I_{\aff}}$ is the set of all simple roots.
Also, for 
$\beta \in \mathfrak{h}\oplus \mathbb{C}\widetilde{\delta}$,
we define ${\dg}(\beta) \in \mathbb{C} $ and 
$\overline{\beta} \in \mathfrak{h}$
by
%%%%%%%%%%
%%%%%%%%%%
%eq:dfn_deg
%%%%%%%%%%
%%%%%%%%%%
\begin{equation}\label{eq:dfn_deg}
\beta = \overline{\beta} + {\dg}(\beta) \widetilde{\delta}.
\end{equation}

We denote the Weyl group of $\widetilde{ \mathfrak{g}}$ by $\widetilde{W}$;
%since $\widetilde{W} \cong W $,
we identify $\widetilde{W}$ and $W$
through the identification of the simple reflections of the same index $I$.
%
%.\footnote{We identify these two groups via index set $I_\aff$.}
For $\nu \in \mathfrak{h}^*$,
let $t(\nu)$ denote the translation in $\mathfrak{h}^*$: $ t(\nu) \gamma = \gamma + \nu$ for $\gamma \in \mathfrak{h}^*$.
The corresponding affine Weyl group and the extended affine Weyl group
are defined by
$\widetilde{W}_{\aff}\eqdef t(Q) \rtimes W $ and
$\widetilde{W}_{\ext} \eqdef t(P) \rtimes W $,
respectively.
Also, we define $s_0 : \mathfrak{h}^* \rightarrow \mathfrak{h}^*$ by $\nu \mapsto \nu -( \langle \nu, \varphi^\lor  \rangle -1)\varphi $.
Then, $\widetilde{W}_{\aff}=\langle s_i \ | \ i \in I_{\aff}\rangle$;
note that $s_0 = t(\varphi) s_\varphi $.
The extended affine Weyl group $\widetilde{W}_{\ext}$ acts on
$\mathfrak{h}\oplus \mathbb{C}\widetilde{\delta}$
as linear transformations,
and on 
$\mathfrak{h}^*$ 
as affine transformations:
for $v\in W$, $t(\nu) \in t(P)$,
\begin{eqnarray*}
v t(\nu)( \overline{\beta}+r\widetilde{\delta} )=v\overline{\beta}+(r-\langle \nu, \overline{\beta} \rangle ) \widetilde{\delta},
&
\overline{\beta} \in  \mathfrak{h} , r \in \mathbb{C},
\\
v t(\nu) \gamma = v \nu +v \gamma,
&
\gamma \in \mathfrak{h}^* .
\end{eqnarray*}

An element $u \in \widetilde{W}_{\ext} $ can be written as 
%%%%%%%%%%
%%%%%%%%%%
%eq:dfn_wt
%%%%%%%%%%
%%%%%%%%%%
\begin{equation}\label{eq:dfn_wt}
u=t({{\wt}(u)}) \dir (u),
\end{equation}
where ${\wt}(u) \in P$ and $ {\dir}(u) \in W$,
according to the decomposition
$\widetilde{W}_{\ext} = t(P) \rtimes W $.
For $w \in \widetilde{W}_{\ext}$,
we denote the length of $w$ by 
$\ell (w) $,
which equals
$\#
\left(\ptrr
\cap
w^{-1}\ntrr \right)
$.
Also, we set 
$\Omega 
\eqdef
\{ w \in \widetilde{W}_{\ext}  \ | \ \ell(w)=0 \}$.

For $\mu \in P$, 
we denote the shortest element in the coset	$t(\mu)W$ by $m_{\mu} \in \widetilde{W}_{\ext}$.
In the following, we fix $\mu \in P$,
and take a reduced expression $m_{\mu} = u s_{\ell_{1}}\cdots s_{\ell_{L}}
\in \widetilde{W}_{\ext} = \Omega \ltimes \widetilde{W}_{\aff}$,
where $u \in \Omega$ and $ \ell_1 , \ldots , \ell_L \in  I_{\aff}$.

For each $J = \{ j_{1} < j_{2} < j_{3} < \cdots < j_{r} \} \subset \{1,\ldots,L\}$,
we define an alcove path $p^{\OS}_{J} =
			\left( m_{\mu} = z^{\OS}_0, z^{\OS}_{1} , \ldots , z^{\OS}_{r} ; \beta^{\OS}_{j_1} , \ldots , \beta^{\OS}_{j_r} \right)$ as follows: 
we set
$\beta^{\OS}_{k} \eqdef s_{\ell_{L}}\cdots s_{\ell_{k+1}} \alpha^{\lor}_{\ell_{k}} \in \ptrr$ 
for $1 \leq k \leq L$, and set
	\begin{eqnarray*}
		z^{\OS}_{0}&=&m_{\mu} ,\\
		z^{\OS}_{1}&=&m_{\mu}s_{\beta^{\OS}_{j_{1}}},\\
		z^{\OS}_{2}&=&m_{\mu}s_{\beta^{\OS}_{j_{1}}}s_{\beta^{\OS}_{j_2}},\\
				&\vdots& \\
		z^{\OS}_{r}&=&m_{\mu}s_{\beta^{\OS}_{j_{1}}}\cdots s_{\beta^{\OS}_{j_r}}.
	\end{eqnarray*}
Also, following \cite[\S 3.3]{OS}, 
we set
	$\B ({\id};m_{\mu})
	\eqdef
	\left\{ p^{\OS}_{J} \ \left| \ J \subset \{ 1,\ldots ,L \}  \right. \right\}$
and
	$\ed (p^{\OS}_{J}) \eqdef z^{\OS}_{r}\in \widetilde{W}_{\ext}$.
Then we define
	${\overleftarrow{\QB}}({\id}; m_{\mu})$
to be the following subset of $\B ({\id};m_{\mu})$:
	\begin{equation*}
		\left\{
			 p^{\OS}_{J} \in \B ({\id} ; m_{\mu} ) \ 
			\left|  \ 
				\dir(z^{\OS}_{i}) 
				\xleftarrow{-\left(\overline{ {\beta}^{\OS}_{j_{i+1}}  } \right)^{\lor}}
				\dir (z^{\OS}_{i+1}) \ 
				\text{ is an edge of }\QBG, \ 0\leq  i \leq r-1
			\right. 
		\right\}.
	\end{equation*}

	\begin{rem}[\normalfont{\cite[(2.4.7)]{M}}]
		If $j \in \{ 1, \ldots, L \}$,
		then $-\left(\overline{ {\beta}^{\OS}_{j}  } \right)^{\lor} \in {\Delta}^{+}$.
	\end{rem}

	For $p^{\OS}_{J} \in {\overleftarrow{\QB}}({\id} ; m_{\mu})$, we define ${\qwt}^{*}(p^{\OS}_{J})$ as follows.
Let $J^+ \subset J$ denote the set of all indices $j_i \in J$  such that 
	$\dir(z_{i-1}^{\OS}) 
	\xleftarrow{-\left(\overline{ {\beta}^{\OS}_{j_i}  } \right)^{\lor}} 
	\dir (z_{i}^{\OS})$ is a quantum edge.
Then we set
	\begin{equation*}
		{\qwt}^{*}(p^{\OS}_{J}) 
		\eqdef
		\sum_{j \in J^{+}} \beta^{\OS}_{j}.
	\end{equation*}

For $\mu \in P$, let $E_{\mu}(q,t)$ denote the nonsymmetric Macdonald polynomial, 
and set
%
%which is of the form $E_{\mu}(q,t) = e^{\mu} + \sum_{\nu < \mu}f_\nu e^{\nu} $,
%$f_\nu \in \mathbb{Q}(q,t)$;
%here the partial order on $P$ is the one in \cite[(2.7.5)]{M}.
%We set 
$E_{\mu}(q,\infty) \eqdef \lim_{t \to \infty}E_{\mu}(q,t)$;
this specialization is studied in \cite{CO} in the dual untwisted cases.
%, which is well-defined (\cite[Proposition 3.5 (i), (iii)]{CO}\footnote{\cite{CO} は dual untwisted の場合なのだが, 引用はこのままで良い?
%(結局 Orr-Shimozono formula があるので well-definedness はそれほど気にする必要はないはず.)
%}).

We know the following formula for the specialization $E_\mu(q, \infty)$ at $t=\infty$.

\begin{prop}[{\cite[Proposition 5.4]{OS}}]\label{os}
Let $\mu \in P$. Then,
	\begin{equation*}
		E_{\mu}(q,\infty)=
		\sum_{p^{\OS}_{J} \in {\overleftarrow{\QB}}({\id} ; m_{\mu}) } 
		q^{-{\dg}({\qwt}^{*}(p^{\OS}_{J}))}e^{\wt(\ed (p^{\OS}_{J}))}.
	\end{equation*}
\end{prop}

\section{Orr-Shimozono formula in terms of $\QLS$ paths}

\subsection{Weak reflection orders}
Let $\lambda \in P^+$ be a dominant weight, $\mu \in W\lambda$,
and set $S \eqdef S_\lambda = \{ i \in I \ | \ \langle \lambda , \alpha^\lor_i \rangle =0 \}$.
We denote by $v(\mu) \in W^S$ the minimal-length coset representative for the coset $\{ w \in W  \ | \ w\lambda=\mu\}$ in $W / W_S$.
We have $\ell(v(\mu)w)=\ell(v(\mu))+ \ell(w)$ for all $w \in W_S$.
In particular, we have $\ell(v(\mu)\lons )=\ell(v(\mu))+ \ell(\lons)$.
When $\mu=\lambda_- \eqdef \lon \lambda$, it is clear that
$\lon \in \{ w \in W \ | \ w\lambda = \lambda_- \}$.
Since $\lon$ is the longest element of $W$, we have 
%%%%%%%%%%
%%%%%%%%%%
%eq:red_longest
%%%%%%%%%%
%%%%%%%%%%
\begin{equation}\label{eq:red_longest}
\lon = v(\lambda_-) \lons,
\end{equation}
and $\ell(v(\lambda_-)\lons )=\ell(v(\lambda_-))+ \ell(\lons)$;
note that $v(\lambda_-) = \lon \lons = \lfloor \lon \rfloor$.
The following lemma follows from \cite[Chap. 2]{M}.
\begin{lem}\label{vmu}
\mbox{}
\begin{enu}
\item
$\dir(m_\mu) = v(\mu)v(\lambda_-)^{-1} $
and
$\ell(\dir(m_\mu)) + \ell( v(\mu) ) =\ell (v(\lambda_-))${\rm;}
hence
\begin{equation}\label{A}
m_\mu 
=
t(\mu)
v(\mu)
v(\lambda_-)^{-1}.
\end{equation}

\item
$v(\mu)v(\lambda_-)^{-1}\lon= v(\mu)\lons$.

\item
$\left( v(\lambda_-)v(\mu)^{-1} \right) m_\mu =m_{\lambda_-} $,
and $\ell(  v(\lambda_-)v(\mu)^{-1} )+ \ell ( m_\mu ) = \ell ( m_{\lambda_-} ) $.

\item
$\ell(v(\lambda_-)v(\mu)^{-1}) + \ell(v(\mu)) = \ell(v(\lambda_-))$.
\end{enu}
\end{lem}

%Let $\lambda \in P$ be a dominant weight,
%and set $S =\{ i \in I \ | \ \langle \lambda , \alpha^\lor_i \rangle =0 \}$.
%%We set $\lambda_- = \lon \lambda$. 
%%Since $\lambda_- $ is an antidominant weight
%%(i.e.,
%%$\langle \lambda_- , \alpha^{\lor}_i \rangle \leq 0$
%%for all $i \in I$),
%We have $m_{\lambda_-} = t(\lambda_-)$ by (\ref{A}).
%as remarked in the proof of Lemma \ref{vmu} (3).
In this subsection, we give a particular reduced expression for $m_{\lambda_-}$ ($= t(\lambda_-)$ by \eqref{A}),
and then study some of its properties.

First of all, we recall the notion of a weak reflection order on $\Delta^+$.
%\begin{dfn}[{\cite[Definition on page 661]{papi}}\label{311}
\begin{dfn}\label{311}
%\footnote{The definition and Theorem of \cite{papi} do not match up as cited with Definition 3.1.1, Proposition 3.1.2, Definition 3.1.3, and Propositions 3.1.4-5.
%It would be cleaner to say that these results collectively follow from the Definition and Theorem of \cite{papi}
%taken together.}]
A total order $\prec$ on $\Delta^+$ is called a (dual) weak reflection order  on $\Delta^+$ if it satisfies
the following condition:
if $\alpha ,\beta ,\gamma \in \Delta^+$ with $\gamma^\lor = \alpha^\lor + \beta^\lor$,
then $\alpha \prec \gamma \prec \beta$ or $\beta \prec \gamma \prec \alpha$.
\end{dfn}

The following result is well-known; see \cite[Theorem on page 662]{papi} for example.
\begin{prop}\label{reford}
For a total order $\prec$ on $\Delta^+$, the following are equivalent{\rm;}
\begin{enu}
\item
the order $\prec$ is a weak reflection order{\rm;}

\item
there exists a {\rm(}unique{\rm)} reduced expression $\lon =s_{i_1}\cdots s_{i_N}$
%such that if $ 1 \leq k<j \leq N$, then $s_{i_N}\cdots s_{i_{k+1}} (\alpha_{i_k}) \prec s_{i_N}\cdots s_{i_{j+1}} (\alpha_{i_j})$.
such that $s_{i_N}\cdots s_{i_{k+1}} \alpha_{i_k} \prec s_{i_N}\cdots s_{i_{j+1}} \alpha_{i_j}$ for $ 1 \leq k<j \leq N$.
\end{enu}
\end{prop}

%Similarly, we recall the notion of a weak reflection order on finite subsets of $\ptrr$.
Next, we recall from \cite[pages 661--662]{papi} the notion and some properties of a weak reflection order on a finite
subset of $\ptrr$;
we remark that arguments in \cite{papi} work in the general setting of Kac-Moody algebras.

\begin{dfn}%[{\cite[Definition on page 661]{papi}}]
Let $T$ be a finite subset of $\ptrr$,
and $\prec'$ a total order on $T$.
We say that the order $\prec'$ is  a weak reflection order on $T$ 
if it satisfies the following conditions:

\begin{enu}
\item
if $\theta_1, \theta_2 \in T$ satisfy $\theta_1 \prec' \theta_2$ and $\theta_1 +\theta_2 \in \ptrr$,
then $\theta_1 +\theta_2 \in T$ and $\theta_1 \prec' \theta_1 +\theta_2 \prec' \theta_2 $;

\item
if  $\theta_1 , \theta_2 \in \ptrr$ satisfy $\theta_1 +\theta_2 \in T$, 
then $\theta_1 \in T$ and $\theta_1 +\theta_2 \prec' \theta_1$,
or
$\theta_2 \in T$ and $\theta_1 +\theta_2 \prec' \theta_2$.
\end{enu}
\end{dfn}

We remark that
there does not necessarily exist a weak reflection order on an arbitrary finite subset of $\ptrr$.

%\begin{prop}[{\cite[Theorem on page 662]{papi}
%\footnote{Need to remark that arguments of Papi hold in Kac-Moody generality.}}]\label{affreford}
\begin{prop}\label{affreford}
Let $T$ be a finite subset of $\ptrr$,
and $\prec'$ a weak reflection order on $T$.
We write $ T$ as $\{ \gamma_1 \prec' \gamma_2 \prec' \cdots \prec' \gamma_p \}$.
Then there exists $w \in \widetilde{W}_{\aff}$ such that $\ptrr \cap w^{-1} \ntrr = T$.
Moreover, there exists a {\rm(}unique{\rm)} reduced expression $w =s_{\ell_1}\cdots s_{\ell_p}$ for $w$
such that $s_{\ell_p}\cdots s_{\ell_{j+1}} \alpha^\lor_{\ell_j}=\gamma_j$
for $1\leq j \leq p$.
\end{prop}

The converse of Proposition \ref{affreford}  also holds.
%\begin{prop}[\normalfont{\cite[Theorem on page 662]{papi}}]\label{affrefordinv}
\begin{prop}\label{affrefordinv}
Let $w \in  \widetilde{W}_{\aff}$,
and let  $w =s_{\ell_1}\cdots s_{\ell_p}$ be a reduced expression.
We define a $\gamma_j \eqdef s_{\ell_p}\cdots s_{\ell_{j+1}} \alpha^\lor_{\ell_j}$
for $1 \leq j \leq p$.
We consider the total order $\prec'$ on $\ptrr \cap w^{-1} \ntrr$ as follows{\rm:} for $1\leq j,k \leq p$, $\gamma_j  \prec' \gamma_k
\ardef
j<k$.
Then, the total order $\prec'$ is a weak reflection order on  $\ptrr \cap w^{-1} \ntrr$.
\end{prop}

\begin{rem}\label{restriction}
Let
\begin{align*}
v(\lambda_-)&= s_{i_1} \cdots s_{i_M},\\
\lons&= s_{i_{M+1}} \cdots s_{i_N},\\
\lon&= s_{i_1} \cdots s_{i_M} s_{i_{M+1}} \cdots s_{i_N}
\end{align*}
be reduced expressions for $v(\lambda_-)$, $\lons$, and $\lon = v(\lambda_-) \lons$,
respectively,
where  $S = S_\lambda = \{ i \in I \ | \ \langle \lambda , \alpha^\lor_i \rangle =0 \}$;
recall that 
$\lons$ is the longest element of $W_S$.
We set
$\beta_j \eqdef s_{i_N}\cdots s_{i_{j+1}}\alpha_{i_j}$, $1 \leq j \leq N$.
By Proposition \ref{reford}, we have
$\Delta^+\setminus \Delta^+_S = 
\{ \beta_1\prec \beta_2 \prec \cdots \prec \beta_M \}$ and
$\Delta^+_S = \{ \beta_{M+1} \prec \beta_{M+2} \prec \cdots \prec \beta_N \}$,
where $\prec$ is the weak reflection order on $\Delta^+$ determined by the reduced expression above for $\lon$.
In particular, we have
\begin{equation}\label{bunkai}
\theta_1 \prec \theta_2 \ \mbox{for} \ \theta_1 \in \Delta^+\setminus \Delta^+_S \ \mbox{and}  \ \theta_2 \in \Delta^+_S.
\end{equation}
Conversely, 
if a weak reflection order on  $\Delta^+$ satisfies
(\ref{bunkai}), 
then the reduced expression $\lon = s_{\ell_1}\cdots s_{\ell_N}$ for $\lon$ corresponding to
this weak reflection order
is given by
concatenating a reduced expression for $v(\lambda_-)$ with a reduced expression for $\lons$.
Moreover,
if we alter a reduced expression for $\lons$
with a reduced expression for $v(\lambda_-)$ unchanged,
then
the restriction to $\Delta^+\setminus \Delta^+_S$
of the weak reflection order on $\Delta^+$ does not change.
Thus, the restriction to $\Delta^+\setminus \Delta^+_S$
of the weak reflection order on $\Delta^+$ satisfying (\ref{bunkai})
depends only on
a reduced expression for $v(\lambda_-)$.
\end{rem}

First let us take a reduced expression $v (\lambda_-) = s_{i_1}\cdots s_{i_M}$ and 
a weak reflection order $\prec$ on $\Delta^+$
such that
the restriction to $\Delta^+\setminus \Delta^+_S$
of this weak reflection order $\prec$ is
determined by the reduced expression $v (\lambda_-) = s_{i_1}\cdots s_{i_M}$
as in Remark \ref{restriction}.
Also, we define an injective map $\Phi$ by:
\begin{eqnarray*}
\Phi : \ptrr \cap m^{-1}_{\lambda_-} \ntrr
&\rightarrow&
\mathbb{Q}_{\geq 0} \times (\Delta^+\setminus \Delta^+_S ) ,\\
\beta = \overline{\beta} + \dg(\beta) \widetilde{\delta}
&\mapsto &
\left(\frac{\langle {\lambda_-} ,  \overline{\beta} \rangle -  \dg(\beta)}{\langle {\lambda_-} ,  \overline{\beta} \rangle }  
,  \lon \left( \overline{\beta} \right)^\lor  
\right);
\end{eqnarray*}
note that $\langle {\lambda_-} ,  \overline{\beta} \rangle >0$, $\langle {\lambda_-} ,  \overline{\beta} \rangle -  \dg(\beta) \geq 0$,
and 
$\lon \left( \overline{\beta} \right)^\lor \in \Delta^+\setminus \Delta^+_S $
since $\langle {\lambda_-} ,  \overline{\beta} \rangle = \langle {\lambda} ,  \lon \overline{\beta} \rangle >0$,
since  we know from \cite[(2.4.7) (i)]{M}
that
\begin{equation}\label{B}
\ptrr \cap m^{-1}_{\lambda_-} \ntrr =
\{ \alpha^{\lor} + a \tilde{\delta} 
\ | \
\alpha \in \Delta^-,
0 < a \leq  \langle \lambda_- , \alpha^{\lor}
\rangle 
\}.
\end{equation}
We now consider the lexicographic order $<$ on $\mathbb{Q}_{\geq 0} \times (\Delta^+\setminus \Delta^+_S)$
induced by the usual total order on $\mathbb{Q}_{\geq 0}$
and
the restriction to $\Delta^+\setminus \Delta^+_S$
of the weak reflection order $\prec$ on $\Delta^+$;
that is, for $(a, \alpha), (b, \beta) \in \mathbb{Q}_{\geq 0} \times (\Delta^+\setminus \Delta^+_S)$,
\begin{equation*}
(a, \alpha)<(b, \beta) \mbox{ if and only if } 
a<b, \mbox{ or }
a=b \mbox{ and } \alpha \prec \beta.
\end{equation*}
Then we denote by $\prec'$ 
the total order on $\ptrr \cap m^{-1}_{\lambda_-} \ntrr$
induced by  the lexicographic order on $\mathbb{Q}_{\geq 0} \times (\Delta^+\setminus \Delta^+_S)$
through the map $\Phi$,
and write $\ptrr \cap m^{-1}_{\lambda_-} \ntrr$
as 
$\left\{\gamma_1 \prec' \cdots \prec' \gamma_L \right\}$.

\begin{prop}\label{goodreducedexpression}
Keep the notation and setting above.
Then, there exists a unique reduced expression $m_{{\lambda_-}}=u s_{\ell_1}\cdots s_{\ell_L}$ for $m_{\lambda_-}$, $u \in \Omega$, $\{\ell_1, \ldots , \ell_L \} \subset I_{\aff}$,
such that
$\beta^\OS_j 
\left( =
s_{\ell_L}\cdots s_{\ell_{j+1}}\alpha^\lor_{\ell_j} \right) =\gamma_j$
for $1 \leq j \leq L$.
\end{prop}

\begin{proof}
We will show that the total order $\prec'$ 
 is a weak reflection order on $\ptrr \cap m^{-1}_{\lambda_-} \ntrr$.

We check condition (1). 
Assume that
$\theta_1, \theta_2 \in \ptrr \cap m^{-1}_{\lambda_-} \ntrr$ satisfy 
$\theta_1 \prec' \theta_2$ and
$\theta_1+\theta_2 \in \ptrr$.
Then it is clear that
$\theta_1+\theta_2 \in \ptrr \cap m^{-1}_{\lambda_-} \ntrr $.

Consider the case that the first component of $\Phi(\theta_1)$ is less than that of $\Phi(\theta_2)$
(i.e., $\frac{\langle {\lambda_-} ,  \overline{\theta_1}\rangle - \dg(\theta_1) }{\langle {\lambda_-} ,  \overline{\theta_1}\rangle}  <
\frac{\langle {\lambda_-} ,  \overline{\theta_2}\rangle - \dg(\theta_2)}{\langle {\lambda_-} ,  \overline{\theta_2}\rangle}$).
In this case, the first component of $\Phi(\theta_1+\theta_2)$
is equal to
$\frac{\langle {\lambda_-} ,  \overline{\theta_1+\theta_2}\rangle - \dg(\theta_1+\theta_2) }{\langle {\lambda_-},  \overline{\theta_1+\theta_2}\rangle}$, 
which lies between the first components of 
$\Phi(\theta_1)$ and $\Phi(\theta_2)$.
Hence we have
$\Phi(\theta_1)<\Phi(\theta_1+\theta_2)<\Phi(\theta_2)$.

Consider the case that the first component of $\Phi(\theta_1)$ is equal to that of $\Phi(\theta_2)$.
In this case, 
we have $\lon \left( \overline{\theta_1 } \right)^\lor \prec \lon \left( \overline{\theta_2} \right)^\lor$,
where $\prec$ is the restriction to $\Delta^+\setminus \Delta^+_S$ of the weak reflection order on $\Delta^+$.
Note that the first component of $\Phi(\theta_1+\theta_2)$ is 
equal to 
$\frac{\langle {\lambda_-} ,  \overline{\theta_1+\theta_2}\rangle - \dg(\theta_1+\theta_2) }{\langle {\lambda_-} ,  \overline{\theta_1+\theta_2}\rangle}$,
which is equal to both of the first components of 
$\Phi(\theta_1)$ and $\Phi(\theta_2)$.
Moreover,
since
$\theta_1+\theta_2 \in \ptrr \cap m^{-1}_{\lambda_-} \ntrr$,
we have
$\lon \left( \overline{\theta_1 +\theta_2} \right)^\lor \in \Delta^+\setminus \Delta^+_S$.
It follows from the definition of the weak reflection order $\prec$ on $\Delta^+$ that
$\lon \left( \overline{\theta_1 } \right)^\lor 
\prec 
\lon \left( \overline{\theta_1 +\theta_2} \right)^\lor 
\prec
\lon \left( \overline{\theta_2} \right)^\lor$.
Hence we have
$\Phi(\theta_1)<\Phi(\theta_1+\theta_2)<\Phi(\theta_2)$.
Thus, the total order $\prec'$ satisfies condition (1).

We check condition (2).
If $\theta_1, \theta_2 \in \ptrr \setminus m^{-1}_{\lambda_-} \ntrr$
and 
$\theta_1 + \theta_2 \in \ptrr$,
then it is clear that $\theta_1 +\theta_2 \in \ptrr \setminus m^{-1}_{\lambda_-} \ntrr$.
Hence we may assume that 
$\theta_1 \in \ptrr \cap m^{-1}_{\lambda_-} \ntrr$
 and
$\theta_2 \in \ptrr \setminus m^{-1}_{\lambda_-} \ntrr$;
if $\theta_1 ,\theta_2 \in \ptrr \cap m^{-1}_{\lambda_-} \ntrr$, 
the assertion is obvious by condition (1).
Since
$\ptrr \cap m^{-1}_{\lambda_-} \ntrr=
\{ \alpha^{\lor} + a \tilde{\delta} 
\ | \
\alpha \in \Delta^-,
0 < a \leq  \langle {\lambda_-} , \alpha^{\lor}
\rangle 
\}$,
we have 
$0 < \dg(\theta_1) \leq  \langle {\lambda_-} , \overline{\theta_1} \rangle $ and
$0 < \dg(\theta_1+\theta_2) \leq  \langle {\lambda_-} , \overline{\theta_1+\theta_2} \rangle $.
Also, since $\theta_2 \in \ptrr \setminus m^{-1}_{\lambda_-} \ntrr$, 
we find that
$\langle {\lambda_-} , \overline{\theta_2} \rangle <0 \leq \dg(\theta_2)$, 
$\dg(\theta_2) > \langle {\lambda_-} ,
  \overline{\theta_2} \rangle \geq0$,
  or
  $\langle {\lambda_-} , \overline{\theta_2} \rangle =  \dg(\theta_2) =0$;
if $0 > \dg(\theta_2)$, then we have $\theta_2 \in \ntrr$, a contradiction.

In the case that $\langle {\lambda_-} , \overline{\theta_2} \rangle <0 \leq \dg(\theta_2)$,
the first component of $\Phi(\theta_1+\theta_2)$, which is $
\frac{\langle {\lambda_-} ,  \overline{\theta_1+\theta_2}\rangle - \dg(\theta_1+\theta_2) }{\langle {\lambda_-} ,  \overline{\theta_1+\theta_2}\rangle} $,
satisfies the inequality
\begin{eqnarray*}
\frac{\langle {\lambda_-} ,  \overline{\theta_1+\theta_2}\rangle - \dg(\theta_1+\theta_2) }{\langle {\lambda_-} ,  \overline{\theta_1+\theta_2}\rangle}
&\leq&
\frac{\langle {\lambda_-} ,  \overline{\theta_1+\theta_2}\rangle - \dg(\theta_1) }{\langle {\lambda_-} ,  \overline{\theta_1+\theta_2}\rangle} \\
&=&
1-\frac{\dg(\theta_1) }{\langle {\lambda_-} ,  \overline{\theta_1+\theta_2}\rangle} 
<
1-\frac{\dg(\theta_1) }{\langle {\lambda_-} ,  \overline{\theta_1}\rangle} 
=
\frac{\langle {\lambda_-} ,  \overline{\theta_1}\rangle - \dg(\theta_1) }{\langle {\lambda_-} ,  \overline{\theta_1}\rangle} .
\end{eqnarray*}
Therefore,
we deduce that
the first component of $\Phi(\theta_1+\theta_2)$ is less than that of $\Phi(\theta_1)$,
and hence $\Phi(\theta_1+\theta_2)<\Phi(\theta_1)$.

In the case that $\dg(\theta_2) > \langle {\lambda_-} , \overline{\theta_2} \rangle \geq 0$,
the first component of $\Phi(\theta_1+\theta_2)$
satisfies the inequality
\begin{eqnarray*}
\frac{\langle {\lambda_-} ,  \overline{\theta_1+\theta_2}\rangle - \dg(\theta_1+\theta_2) }{\langle {\lambda_-} ,  \overline{\theta_1+\theta_2}\rangle}
&=&
\frac{\left( \langle {\lambda_-} ,  \overline{\theta_1}\rangle - \dg(\theta_1)\right) +\left( \langle {\lambda_-} ,  \overline{\theta_2}\rangle - \dg(\theta_2)\right)  }{\langle {\lambda_-} ,  \overline{\theta_1+\theta_2}\rangle} \\
&<&
\frac{\left( \langle {\lambda_-} ,  \overline{\theta_1}\rangle - \dg(\theta_1)\right) }{\langle {\lambda_-} ,  \overline{\theta_1+\theta_2}\rangle}
\leq
\frac{\langle {\lambda_-} ,  \overline{\theta_1}\rangle - \dg(\theta_1) }{\langle {\lambda_-} ,  \overline{\theta_1}\rangle} .
\end{eqnarray*}
Therefore,
we deduce that
the first component of $\Phi(\theta_1+\theta_2)$ is less than that of $\Phi(\theta_1)$,
and hence that $\Phi(\theta_1+\theta_2)<\Phi(\theta_1)$.

In the case that 
$\langle {\lambda_-} , \overline{\theta_2} \rangle =  \dg(\theta_2) =0$,
the first component of $\Phi(\theta_1+\theta_2)$
is equal to that of $\Phi(\theta_1)$.
Moreover, since $\langle {\lambda_-} , \overline{\theta_2} \rangle = \langle \lambda , \lon \overline{\theta_2} \rangle = 0$,
we have  $\lon (\overline{\theta_2})^\lor \in \Delta^+_S$.
Therefore, by (\ref{bunkai}), 
we see that 
$\lon (\overline{\theta_1+\theta_2})^\lor \prec \lon (\overline{\theta_2})^\lor$.
It follows from the definition of the weak reflection order
on $\Delta^+$ that $
\lon (\overline{\theta_1})^\lor \prec
\lon (\overline{\theta_1+\theta_2})^\lor \prec \lon (\overline{\theta_2})^\lor$,
and hence that 
$\Phi(\theta_1+\theta_2) < \Phi(\theta_1)$.

Thus, we conclude that $\prec'$ satisfies condition (2), and the total order $\prec'$ is a weak reflection order on
$\ptrr \cap m^{-1}_{\lambda_-} \ntrr$.

Now, by Proposition \ref{affreford},
there exists $w\in \widetilde{W}_\aff$
such that $\ptrr \cap m^{-1}_{\lambda_-} \ntrr
=
\ptrr \cap w^{-1} \ntrr$,
and 
there exists
a reduced expression
$w=s_{\ell_1}\cdots s_{\ell_L}$,
$\{\ell_1, \ldots , \ell_L \} \subset I_{\aff}$,
for $w$ such that
$\gamma_j = s_{\ell_L}\cdots s_{\ell_{j+1}}\alpha^\lor_{\ell_j}$ for $1\leq j \leq L$.
Since $\ptrr \cap m^{-1}_{\lambda_-} \ntrr
=
\ptrr \cap w^{-1} \ntrr$,
it follows from \cite[(2.2.6)]{M} that there exists $u \in \Omega$ such that $u w=m_{\lambda_-}$.
Thus, we obtain a reduced expression $m_{\lambda_-}=us_{\ell_1}\cdots s_{\ell_L}$ for $m_{\lambda_-}$,
with $\gamma_j = s_{\ell_L}\cdots s_{\ell_{j+1}}\alpha^\lor_{\ell_j}=\beta^\OS_j$ for $1\leq j \leq L$.
This completes the proof of Proposition \ref{goodreducedexpression}. 
\end{proof}

By Remark \ref{restriction},
the restriction to $\Delta^+ \setminus \Delta_S^+$ of a weak reflection order on $\Delta^+$ satisfying (\ref{bunkai})
corresponds bijectively to a reduced expression $v({\lambda_-}) = s_{i_1}\cdots s_{i_M}$ for $v({\lambda_-})$.
Hence
by Proposition \ref{goodreducedexpression},
we can take a reduced expression $m_{{\lambda_-}}=u s_{\ell_1}\cdots s_{\ell_L}$ for $m_{\lambda_-}$
corresponding to each reduced expression $v({\lambda_-})= s_{i_1}\cdots s_{i_M}$ for $v({\lambda_-})$.
Conversely,
as seen in Lemma \ref{lengthadditive}, 
from the reduced expression $m_{{\lambda_-}}=u s_{\ell_1}\cdots s_{\ell_L}$ for $m_{\lambda_-}$,
we can obtain a reduced expression for $v({\lambda_-})$,
which is identical to the original reduced expression
$v({\lambda_-})= s_{i_1}\cdots s_{i_M}$
(see Lemma \ref{lengthadditive} below).

In the remainder of this subsection, 
we fix reduced expressions
$v(\lambda_-) = s_{i_1}\cdots s_{i_M} $
and
$\lons = s_{i_{M+1}}\cdots s_{i_N}$,
and use the weak reflection order $\prec$ on $\Delta^+$
(which satisfies 
(\ref{bunkai}))
determined by these reduced expressions for $v(\lambda_-)$ and $\lons$.
Also, we use the total order $\prec'$ on $\ptrr \cap m_{\lambda_-}^{-1} \ntrr$ 
defined just before Proposition \ref{goodreducedexpression},
and take a reduced expression 
$m_{\lambda_-} = u s_{\ell_1}\cdots s_{\ell_L}$
for $m_{\lambda_-}$ given by Proposition \ref{goodreducedexpression}.

Recall that $\beta^{\OS}_k=s_{\ell_L}\cdots s_{\ell_{k+1}}\alpha^\lor_{\ell_k}$
for $1 \leq k \leq L$.
We set $a_k \eqdef \dg(\beta^{\OS}_k) \in \mathbb{Z}_{> 0} $;
since 
$\ptrr \cap m_{\lambda_-}^{-1} \ntrr
=
\{
\beta^\OS_1, \ldots , \beta^\OS_L
\}
$,
we see by (\ref{B}) that
$0<a_k \leq \langle {\lambda_-}, \overline{\beta^\OS_k}\rangle$.
Also, for $1 \leq j \leq L$,
we set $\beta^{\rm L}_k \eqdef u s_{\ell_1}\cdots s_{\ell_{k-1}} \alpha^{\lor}_{\ell_k}$
and $b_k \eqdef \dg(\beta^{\rm L}_k) \in \mathbb{Z}_{\geq 0} $.
Then we have 
$\{ \beta^{\rm L}_k \ | \ 1 \leq k \leq L \}= 
\ptrr \cap m_{\lambda_-} \ntrr 
=
\{ \alpha^{\lor} + a \tilde{\delta} 
\ | \
\alpha \in \Delta^+,
0 \leq a < -\langle {\lambda_-} , \alpha^\lor
\rangle 
\}$
(see \cite[(2.4.7) (ii)]{M}).

\begin{rem}\label{akbk}
For $1\leq k \leq L$, we have
\begin{align*}
-t({\lambda_-})\beta^{\OS}_k
&=
-(u s_{\ell_1}\cdots s_{\ell_L}) ( s_{\ell_L} \cdots s_{\ell_{k+1}} \alpha^\lor_{\ell_k})
=
-u s_{\ell_1}\cdots s_{\ell_{k-1}}s_{\ell_k} \alpha^\lor_{\ell_k} \\
&=
-u s_{\ell_1}\cdots s_{\ell_{k-1}}( - \alpha^\lor_{\ell_k})
=
u s_{\ell_1}\cdots s_{\ell_{k-1}} \alpha^\lor_{\ell_k}
=
\beta^{\rm L}_k
=
\overline{\beta^{\rm L}_k}+b_k \widetilde{\delta}.
\end{align*}
From this, together with
$
-t({\lambda_-})\beta^{\OS}_k =
 -\overline{ \beta^{\OS}_k } 
 -(a_k - \langle {\lambda_-}, \overline{\beta^{\OS}_k} \rangle  )\tilde{\delta}
$,
we obtain
$\overline{\beta^{\rm L}_k }=-\overline{\beta^{\OS}_k}$
and
$\langle {\lambda_-} ,\overline{\beta^{\OS}_k} \rangle - a_k = b_k$.

\end{rem}

\begin{lem}\label{lengthadditive}
Keep the notation and setting above.
Since $us_{\ell_k} =s_{i'_k}u$ for some $i'_k \in I_{\aff}$,
$1 \leq k \leq M$,
we can rewrite the reduced expression  $u s_{\ell_1}\cdots s_{\ell_L}$ for $m_{{\lambda_-}}$
as $s_{i'_1} \cdots s_{i'_M} u s_{\ell_{M+1}} \cdots s_{\ell_L}$.
Then, $s_{i'_1} \cdots s_{i'_M}$ is a reduced expression for $v({\lambda_-})$,
and $u s_{\ell_{M+1}} \cdots s_{\ell_L}$ is a reduced expression for $m_\lambda$.
Moreover, $i_k = i'_k$ for $ 1 \leq k \leq M$.
\end{lem}

\begin{proof}
First we show that $\{\beta^{\rm L}_k \ | \ 1\leq k \leq M \}= -\lon \left( \widetilde{\Delta}^+ \setminus \widetilde{\Delta}^+_S \right)$.
Since $\{
\beta^{\OS}_j \ | \ 1 \leq j \leq L
\}
=
\{ \alpha^{\lor} + a \tilde{\delta} 
\ | \
\alpha \in \Delta^-,
0 < a \leq \langle {\lambda_-} , \alpha^\lor
\rangle 
\}$,
we see that the minimum value of 
the first components of $\Phi(\beta^{\OS}_k)$, 
i.e., $\frac{\langle {\lambda_-} ,  \overline{\beta^{\OS}_k }\rangle - a_k}{\langle {\lambda_-} ,  \overline{\beta^{\OS}_k }\rangle}$ for
$1 \leq k \leq L$,
is equal to 0.
Since 
$\Phi(\beta^\OS_1) < \Phi(\beta^\OS_2)< \cdots < \Phi(\beta^\OS_L)$,
where $<$ denotes
the lexicographic order on $\mathbb{Q}_{\geq 0} \times (\Delta^+\setminus \Delta^+_S)$, 
there exists a positive integer $M'$ such that 
the first component of $\Phi(\beta^{\OS}_k)$  is equal to 0 for $1\leq k \leq M'$,
and greater than 0 for $M'+1 \leq k \leq L$.
Since
$\beta^{\rm L}_k
=\overline{\beta^{\rm L}_k}+b_k\widetilde{\delta}$
and $\langle {\lambda_-} ,\overline{\beta^{\OS}_k} \rangle - a_k = b_k$ by Remark \ref{akbk},
we deduce that the first component of $\Phi(\beta^{\OS}_k)$ is equal to 0 if and only if $\beta^{\rm L}_k = \overline{\beta^{\rm L}_k} \in \widetilde{\Delta}^+ $.
In this case,
we have
$\langle \lambda , - \lon \beta^{\rm L}_k \rangle= \langle \lambda_- , -\beta^{\rm L}_k \rangle \overset{\rm Remark \ \ref{akbk}}= \langle \lambda_- , \overline{\beta^{\OS}_k } \rangle  >0$,
and hence
$\beta^{\rm L}_k \in - \lon (\widetilde{\Delta}^+ \setminus \widetilde{\Delta}^+_S)$.
Therefore, we obtain
$\{\beta^{\rm L}_k \ | \ 1 \leq k \leq L \}\cap -\lon (\widetilde{\Delta}^+ \setminus \widetilde{\Delta}^+_S)=\{\beta^{\rm L}_k \ | \ 1\leq k \leq M' \} \subset - \lon (\widetilde{\Delta}^+ \setminus \widetilde{\Delta}^+_S)$.
Also, because
$\{\beta^{\rm L}_k \ | \ 1 \leq k \leq L \}=\ptrr \cap m_{{\lambda_-}}\ntrr=
\{ \alpha^{\lor} + a \tilde{\delta} 
\ | \
\alpha \in \Delta^+,
0\leq a < - \langle {\lambda_-} , \alpha^{\lor}
\rangle 
\}
\supset - \lon (\widetilde{\Delta}^+ \setminus \widetilde{\Delta}^+_S)
$,
we deduce that 
$\{\beta^{\rm L}_k \ | \ 1\leq k \leq M' \} = - \lon (\widetilde{\Delta}^+ \setminus \widetilde{\Delta}^+_S)$.
Since $\#( \widetilde{\Delta}^+ \setminus \widetilde{\Delta}^+_S)= M$, it follows that $M=M'$,
and hence $\{\beta^{\rm L}_k \ | \ 1\leq k \leq M \} = - \lon (\widetilde{\Delta}^+ \setminus \widetilde{\Delta}^+_S)$.

We show that $i'_k \in I$ for $1 \leq k \leq M$.
We set $\zeta^{\lor}_k  \eqdef s_{i'_1}\cdots s_{i'_{k-1}}\alpha^{\lor}_{i'_k}$
for $1\leq k \leq M$.
Since $u \alpha^{\lor}_{\ell_k} =\alpha^{\lor}_{i'_k}$, 
we have
\begin{equation*}
\beta^{\rm L}_k 
=us_{\ell_1} \cdots s_{\ell_{k-1}}\alpha^{\lor}_{\ell_k}
=s_{i'_1} \cdots s_{i'_{k-1}}u\alpha^{\lor}_{\ell_k}
=s_{i'_1} \cdots s_{i'_{k-1}} \alpha^{\lor}_{i'_k}
=\zeta^{\lor}_k.
\end{equation*}
Hence it follows that
$\{\zeta^{\lor}_k \ | \ 1\leq k \leq M \}=\{\beta^{\rm L}_k \ | \ 1\leq k \leq M \}= - \lon (\widetilde{\Delta}^+ \setminus \widetilde{\Delta}^+_S)$.
If there exists $k \in \{ 1 , \ldots , M \}$ such that $i'_k = 0$,
then, by choosing the minimum of such $k$'s,
we obtain
$\zeta^{\lor}_k = s_{i'_1}\cdots s_{i'_{k-1}}\alpha^{\lor}_{i'_k} \notin \widetilde{\Delta}^+$,
contrary to the equality $\{\zeta^{\lor}_k \ | \ 1\leq k \leq M \}= -\lon (\widetilde{\Delta}^+ \setminus \widetilde{\Delta}^+_S)$.
Therefore, we have $i'_k \in I$ for $1 \leq k \leq M$.

Next, we show that $s_{i'_1} \cdots s_{i'_M}$ is a reduced expression for $v({\lambda_-})$
and 
$u s_{\ell_{M+1}} \cdots s_{\ell_L}$ is a reduced expression for $m_\lambda$.
Since $s_{\ell_1}\cdots s_{\ell_{M}}$ is a reduced expression, so is $s_{i'_1}\cdots s_{i'_{M}}$.
Therefore, there exist $i'_{M+1}, \ldots , i'_{N} \in I$ such that
$ \lon = s_{i'_1}\cdots s_{i'_{M}}s_{i'_{M+1}} \cdots  s_{i'_{N}}$ is a reduced expression for $\lon$.
Because $s_{i'_{N}} \cdots s_{i'_{M+1}} s_{i'_{M}} \cdots s_{i'_{k+1}}\alpha^\lor_{i'_k}
=
-\lon \beta^{\rm L}_k$,
$1 \leq k \leq M$,
by using the reduced expression above for $\lon$,
we obtain
\begin{equation*}
\widetilde{\Delta}^+ =\{ -\lon \beta^{\rm L}_1, \ldots , -\lon\beta^{\rm L}_M,
s_{i'_{N}} \cdots  s_{i'_{M+2}}\alpha^\lor_{i'_{M+1}}, \ldots , \alpha^\lor_{i'_N} \}.
\end{equation*}
Here, $\{\beta^{\rm L}_k \ | \ 1\leq k \leq M \} = -\lon(\widetilde{\Delta}^+ \setminus \widetilde{\Delta}^+_S)$
implies $\{ s_{i'_{N}} \cdots  s_{i'_{M+2}}\alpha^\lor_{i'_{M+1}}, \ldots , \alpha^\lor_{i'_N} \} = \widetilde{\Delta}^+_S$.
From this by descending induction on $M+1 \leq k \leq N$, we deduce that $i'_{M+1}, \ldots , i'_{N} \in S$, and $s_{i'_{M+1}} \cdots  s_{i'_{N}} $ is an element of  $W_S$;
note that the length of this element is equal to $N-M$, which is the cardinality of $ \widetilde{\Delta}^+_S$.
Therefore,  $s_{i'_{M+1}} \cdots  s_{i'_{N}} $ is the longest element $\lons$ of $W_S$,
and hence $s_{i'_1}\cdots s_{i'_{M}} = \lon \lons = v(\lambda_-)$, which is a reduced expression for $v(\lambda_-)$.
Moreover, because $m_{{\lambda_-}}=v(\lambda_-) m_{\lambda}$ with
$\ell(m_{{\lambda_-}})=\ell (v(\lambda_-))+\ell (m_{\lambda})$ by Lemma \ref{vmu} (3) for the case $\mu = \lambda$,
$m_\lambda = v(\lambda_-)^{-1} m_{{\lambda_-}} = u s_{\ell_{M+1}} \cdots s_{\ell_L}$ is a reduced expression for $m_\lambda$.

Finally, we show that $i_k = i'_k$ for $ 1 \leq k \leq M$.
Since $M=M'$ as shown above, 
\begin{equation*}
\Phi(\beta_k^\OS)=\left(\frac{\langle {\lambda_-} ,  \overline{\beta^{\OS}_k }\rangle - a_k }{\langle {\lambda_-} ,  \overline{\beta^{\OS}_k }\rangle}  
,  \lon \left( \overline{\beta^{\OS}_k } \right)^\lor \right) 
=
\left( 0,  \lon \left( \overline{\beta^{\OS}_k } \right)^\lor \right) 
\end{equation*}
for $1\leq k \leq M$
by the definition of $\Phi$,
and 
\begin{align*}
\lon \left( \overline{\beta^{\OS}_k } \right)^\lor
=
-\lon \left( \overline{\beta^{\rm L}_k } \right)^\lor
&=
-\lon \zeta_k
=
-s_{i'_N}\cdots s_{i'_{M+1}}s_{i'_M}\cdots s_{i'_1}s_{i'_1}\cdots s_{i'_{k-1}}\alpha_{i'_{k}} \\
&=
s_{i'_N}\cdots s_{i'_{M+1}}s_{i'_M}\cdots s_{i'_{k+1}}\alpha_{i'_{k}}
\end{align*}
by Remark \ref{akbk}.
Thus, for $1\leq k<j \leq M $,
we have $s_{i'_N}\cdots s_{i'_{M+1}}s_{i'_M}\cdots s_{i'_{k+1}}\alpha_{i'_{k}}
\prec
s_{i'_N}\cdots s_{i'_{M+1}}s_{i'_M}\cdots s_{i'_{j+1}}\alpha_{i'_{j}}$,
where the order $\prec$ is the fixed weak reflection order on $\Delta^+$ defined just before Proposition \ref{goodreducedexpression}.
Here we recall from Remark \ref{restriction} that $\beta_k = s_{i_N}\cdots s_{i_{k+1}}\alpha_{i_k}$, $1\leq k \leq N$.
Because
\begin{equation*}
\{ \beta_k \ | \ 1\leq k \leq M \} =
\{ s_{i'_N}\cdots s_{i'_{M+1}}s_{i'_M}\cdots s_{i'_{k+1}}\alpha_{i'_{k}} \ | \ 1\leq k \leq M \}
= \Delta^+\setminus \Delta^+_S ,
\end{equation*}
it follows from the definition of the weak reflection order
$\prec$ on $\Delta^+$ together with (\ref{bunkai})
that
\begin{equation*}
\left\{ \beta_1 \prec \cdots \prec \beta_M \right\}=
\left\{
s_{i'_N}\cdots s_{i'_{M+1}}s_{i'_M}\cdots s_{i'_{2}}\alpha_{i'_{1}}\prec \cdots \prec 
s_{i'_N}\cdots s_{i'_{M+1}}\alpha_{i'_M}
\right\}
=
\Delta^+\setminus \Delta^+_S.
\end{equation*}
Therefore, noting that $\beta_k =s_{i_N}\cdots s_{i_{k+1}}\alpha_{i_{k}} $ for $ 1 \leq k \leq N$,
we obtain 
\begin{equation}\label{sreduced}
s_{i_N}\cdots s_{i_{k+1}}\alpha_{i_k} =s_{i'_N}\cdots s_{i'_{M+1}} s_{i'_M}\cdots s_{i'_{k+1}}\alpha_{i'_k}, \ \ \mbox{for }  1 \leq k \leq M.
\end{equation}
By substituting the equality $s_{i_{M+1}}\cdots s_{i_{N}}= \lons =s_{i'_{M+1}}\cdots s_{i'_{N}}$ into (\ref{sreduced}),
we have
$s_{i_M}\cdots s_{i_{k+1}}\alpha_{i_k} = s_{i'_M}\cdots s_{i'_{k+1}}\alpha_{i'_k}$ for $1\leq k \leq M$.
In particular, when $k=M$, we have
$\alpha_{i_M}=\alpha_{i'_M}$, which implies that $i_M = i'_M$.
If $i_j = i'_j$ for $ k+1 \leq j \leq M$,
then it follows from  $s_{i_M}\cdots s_{i_{k+1}}\alpha_{i_{k}}= s_{i'_M}\cdots s_{i'_{k+1}}\alpha_{i'_{k}}$
that $\alpha_{i_{k}}=\alpha_{i'_{k}}$, and
hence $i_{k}$= $i'_{k}$.
Thus, by descending induction on $k$,
we deduce that $i_k = i'_k$ for $ 1 \leq k \leq M $. 
\end{proof}

\begin{rem}[\normalfont{\cite[Theorem 8.3]{LNSSS2}}]\label{2.15}
For $1 \leq k \leq L$, we set
	\begin{equation*}
		d_k \eqdef
			 \frac{\langle \lambda_- ,  \overline{\beta^{\OS}_k }\rangle - a_k}{\langle \lambda_- ,  \overline{\beta^{\OS}_k }\rangle}  
			= \frac{b_k}{\langle - \lambda_-  , \overline{ \beta^{\rm L}_k } \rangle};
	\end{equation*}
the second equality follows from Remark \ref{akbk}.
Here $d_k$ is just the first component of $\Phi(\beta^{\OS}_k) \in \mathbb{Q}_{\geq 0} \times (\Delta^+ \setminus \Delta^+_S)$.
For  $1 \leq k,j \leq L$,
$\Phi(\beta^\OS_k)<\Phi(\beta^\OS_j)$ if and only if $k < j$,
and hence we have
	\begin{equation}\label{C}
		0\leq d_1 \leq \cdots \leq d_L \lneqq 1.
	\end{equation}
\end{rem}

\begin{lem}\label{remark2.11}
If $1\leq k<j \leq L$ and $d_k =d_j$,
then	$\lon \left( \overline{ \beta^{\OS}_k } \right)^{\lor} \prec \lon  \left( \overline{ \beta^{\OS}_j } \right)^{\lor}$.
\end{lem}

\begin{proof}
By the definitions, we obtain
$\Phi(\beta^{\OS}_k )=
\left(d_k, \lon \left( \overline{ \beta^{\OS}_k } \right)^{\lor} \right)$ and 
$\Phi(\beta^{\OS}_j )=
\left(d_j, \lon \left( \overline{ \beta^{\OS}_j } \right)^{\lor} \right)$.
Since $d_k = d_j$ and $\Phi(\beta^{\OS}_k) < \Phi(\beta^{\OS}_j)$,
we have
$
\lon \left( \overline{ \beta^{\OS}_k } \right)^{\lor} \prec \lon \left( \overline{ \beta^{\OS}_j } \right)^{\lor} .$
\end{proof}

\subsection{Orr-Shimozono formula in terms of $\QLS$ paths}
Let $\lambda \in P^+$ be a dominant  weight,
and set $S = S_\lambda = \{ i \in I \ | \ \langle \lambda , \alpha^\lor_i \rangle =0 \}$.
\begin{dfn}[{\cite[Definition 3.1]{LNSSS2}}]\label{def_qls}
A pair $\psi = (w_1, w_2 ,\ldots ,w_s ; \sigma_0, \sigma_1 , \ldots , \sigma_s ) $
of a sequence
$w_1,\ldots , w_s$ of  elements in $W^S$ 
such that $w_k \neq w_{k+1}$ for $1  \leq k \leq s-1$
   and an increasing sequence  
$0=\sigma_0< \cdots < \sigma_s=1$ of rational numbers
is called a quantum Lakshmibai-Seshadri ($\QLS$) path of shape $\lambda$
if 

(C)
for every $1\leq i \leq s-1$, there exists a directed  path from $w_{i+1}$ to $w_{i}$ in $\QBG^S_{\sigma_i \lambda}$.

%for every $1\leq i \leq s-1$, there exists a shortest directed  path in $\QBG^S$ from $w_{i+1}$ to $w_{i}$
%which is also a directed path in $\QBG^S_{\sigma_i \lambda}$.
\noindent
Let $\QLS(\lambda)$ denote the set of all $\QLS$ paths of shape $\lambda$.
\end{dfn}

\begin{rem}\label{def_qls_rem}
We know from \cite[Definition 3.2.2 and Theorem 4.1.1]{LNSSS4}
that condition (C) can be replaced by

(C)'
for every $1\leq i \leq s-1$, there exists a directed path from $w_{i+1}$ to $w_{i}$ in $\QBG^S_{\sigma_i \lambda}$
that is also a shortest directed path from $w_{i+1}$ to $w_{i}$ in $\QBG^S$.
\end{rem}
%Note that
%in Definition \ref{def_qls},
%``shortest'' can be omitted (see {\cite[Theorem 4.1.1]{LNSSS3}}).

For $\psi=(w_1, w_2 ,\ldots ,w_s ; \sigma_0, \sigma_1 , \ldots , \sigma_s ) \in \QLS(\lambda)$,
we set 
\begin{equation*}
\wt(\psi) 
\eqdef
 \sum^{s-1}_{i=0}(\sigma_{i+1}-\sigma_i) w_{i+1} \lambda,
\end{equation*}
and
we define 
$\kappa : \QLS(\lambda) \rightarrow W^S$ by $\kappa (\psi) \eqdef w_s$.
Also, for $\mu \in W\lambda$,
we define the degree of $\psi$ at $\mu$ by
	\begin{equation*}
		\dg_{\mu}(\psi) 
		\eqdef - \sum_{i=1}^{s} \sigma_{i} \wt_{\lambda}(w_{i+1} \Rightarrow w_{i});
	\end{equation*}
here we set $w_{s+1} \eqdef v(\mu)$. 
Note that by Remark \ref{def_qls_rem}, 
$\sigma_{i} \wt_{\lambda}(w_{i+1} \Rightarrow w_{i})\in \mathbb{Z}_{\geq 0}$ for $1\leq i \leq s-1$.
Also,
$\sigma_s=1$ for $i=s$
from the definition of a $\QLS$ path.
Hence it follows that ${\dg}_{\mu}(\psi) \in \mathbb{Z}_{\leq 0}$.

Now, we define a subset $\EQB(w)$ of $W$ for each $w \in W$.
Let
$
		w = s_{i_{1}}\cdots s_{i_{p}}
$
be a reduced expression for $w$.
For each 
$
		J = \{ j_{1} < j_{2} < j_{3} < \dots < j_{r} \} \subset \{1,\ldots,p\},
$
we define
\begin{equation*}
p_{J} \eqdef
\left( w=z_0 , \ldots , z_r ;
 \beta_{j_1}, \ldots , \beta_{j_r} \right)
\end{equation*}
 as follows: 
we set $\beta_{k} \eqdef s_{i_{p}}\cdots s_{i_{k+1}}(\alpha_{i_{k}}) \in \Delta^+$ for $1\leq k \leq p$, and set
	\begin{align*}
		z_{0}&=w=s_{i_{1}}\cdots s_{i_{p}},\\
		z_{1}&=w s_{\beta_{j_{1}}} =
s_{i_{1}}\cdots s_{i_{j_1-1}} s_{i_{j_1+1}} \cdots s_{i_{p}}
= s_{i_{1}}\cdots \widebreve{s_{i_{j_{1}}}}  \cdots s_{i_{p}},\\
		z_{2}&=w s_{\beta_{j_{1}}}s_{\beta_{j_2}} =
s_{i_{1}}\cdots s_{i_{j_1-1}} s_{i_{j_1+1}} \cdots 
s_{i_{j_2-1}} s_{i_{j_2+1}} \cdots s_{i_{p}}=
 s_{i_{1}}\cdots 
		\widebreve{ s_{i_{j_{1}}} } \cdots \widebreve{ s_{i_{j_{2}}} } \cdots 	s_{i_{p}} ,\\
		&\vdots \\
		z_{r}&=w s_{\beta_{j_{1}}}\cdots s_{\beta_{j_r}} = s_{i_{1}}\cdots 
		\widebreve{ s_{i_{j_{1}}} } \cdots \widebreve{ s_{i_{j_{r}}}}  \cdots s_{i_{p}}, 
\end{align*}
where the symbol $\widebreve{\cdot \ }$
indicates a term to be omitted; also, we set $\ed(p_J)\eqdef z_r$.
Then we define
${\B}(w)
\eqdef
\{p_{J} \ |  \ J \subset \{ 1,\ldots ,p\} \}$,
and
\begin{equation*}
\QB (w) 
\eqdef
\{ p_{J} \in {\B}(w) \ | \ z_{i} \xrightarrow{\beta_{j_{i+1} }} z_{{i+1}}\,\mbox{is an  edge of }\QBG \mbox{ for all }  0 \leq  i \leq r-1 \} .
\end{equation*}
We remark that $J$ may be an empty set $\emptyset$; in this case, 
$\ed (p_\emptyset) = w$.

\begin{rem}
We identify elements in $\QB(w)$ with directed paths in $\QBG$.
More precisely, for
$
p_J = \left( w=z_0 , \ldots , z_r ;
 \beta_{j_1}, \ldots \beta_{j_r} \right)
\in \QB(w)
$,
we write
\begin{equation*}
p_J = \left( w=z_0 , \ldots , z_r ;
 \beta_{j_1}, \ldots \beta_{j_r} \right) 
=
\left( w = z_0 \xrightarrow{\beta_{j_1}} \cdots 
\xrightarrow{\beta_{j_r}} z_r \right).
\end{equation*}
\end{rem}

\begin{rem}\label{usingshell}
Let $ w=z_0 \xrightarrow{\beta_{j_1}}  z_{1} \xrightarrow{\beta_{j_2}}   \cdots \xrightarrow{\beta_{j_r}} z_{r} = z  $
be a  directed path in $\QBG$.
Then
we see that
	\begin{equation*}
			1 \leq j_1 <j_2 <\cdots <j_r \leq p \Leftrightarrow
			\left( w=z_0 \xrightarrow{\beta_{j_1}}  z_{1} \xrightarrow{\beta_{j_2}}   
			\cdots \xrightarrow{\beta_{j_r}} z_{r} = z  \right) 
			\in {\QB}(w).
	\end{equation*}
%since for $J = \{ j_1 , \ldots , j_r \}$,
%$p_J$ is identical to the path above.
%The converse is also true by the definition of  $\QB(w)$.
%Moreover, 
Also, it follows from Proposition \normalfont{{\ref{shellability}}} (1) that
the map $\ed : \QB(w) \rightarrow W$ is injective.
\end{rem}

By using the map $\ed :{\B}(w) \rightarrow W$ defined above,
we set
$\EQB (w) \eqdef \ed (\QB (w))$.

\begin{prop}\label{independent_of_reduced_expression}
	 The set $\EQB (w)$ is independent of the choice of a reduced expression for $w$.
\end{prop}

\begin{proof}
Let us take two reduced expressions for $w$:
	\begin{equation*}
		{\bf I} : w=s_{i_{1}}\cdots s_{i_{p}} 
		\mbox{ and }
		{\bf K} : w=s_{k_{1}}\cdots s_{k_{p}}.
	\end{equation*}
In this proof, let ${\EQB(w)}_{\bf I}$ 
(resp., ${\EQB(w)}_{\bf K}$)
denote
the set $\EQB (w)$  associated to ${\bf I}$ 
(resp., ${\bf K}$).

It suffices to show that ${\EQB(w)}_{\bf I} \subset {\EQB}(w)_{\bf K}$.
From the two reduced expressions for $w$, we obtain the following two reduced expressions for $\lon$:
	\begin{eqnarray}
		\lon &=&s_{i_{-q}}\cdots s_{i_{0}}s_{i_{1}}\cdots s_{i_{p}} 
		\label{reduced_expression_i}, \\
		&=&s_{i_{-q}}\cdots s_{i_{0}} s_{k_{1}}\cdots s_{k_{p}} 
		\label{reduced_expression_k}.
	\end{eqnarray}
Using the reduced expression (\ref{reduced_expression_i}) (resp., (\ref{reduced_expression_k})), 
we define $\beta_m$ (resp., $\gamma_m$), $-q \leq m\leq p$, as in (\ref{inversion_root}).
Then we have
	\begin{eqnarray}
			\{ \beta_{-q}, \ldots, \beta_p \} 
		&=& \{ \gamma_{-q}, \ldots , \gamma_{p} \} 
			= \Delta^+ \label{positiveroots},  \\
			\{ \beta_{1}, \ldots, \beta_p \} 
		&=& \{ \gamma_{1}, \ldots , \gamma_{p} \} 
			=\Delta^+ \cap w^{-1} \Delta^-  \label{inversionset}.
	\end{eqnarray}
Let $z \in \EQB(w)_{\bf I}$, and 
	\begin{equation}\label{p_J}
		p_{J} =
		\left( w=z_0 \xrightarrow{\beta_{j_1}}  z_{1} \xrightarrow{\beta_{j_2}}   \cdots \xrightarrow{\beta_{j_r}} z_{r} = z  \right) 
		\in \QB(w)_{\bf I};
	\end{equation}
recall from Remark \ref{usingshell} that $1 \leq j_1 \leq \cdots \leq j_r \leq p$.
It follows from Proposition \ref{shellability} (1) that there exists a unique shortest directed path in $\QBG$  
	\begin{equation}\label{10}
		w=y_0\xrightarrow{\gamma_{n_1}} y_{1} \xrightarrow{\gamma_{n_2}}
		\cdots \xrightarrow{\gamma_{n_r}} y_{r}=z,
	\end{equation}
	with $-q \leq n_1 < n_2 < \cdots < n_r \leq p$;
this is a label-increasing directed path in the weak reflection order defined by
$\gamma_{-q} \prec \cdots \prec \gamma_{p}$.
To prove that $z \in {\EQB(w)}_{\bf K}$,
it suffices to show that $1 \leq n_1$.
It follows from (\ref{positiveroots})
that for $ 1 \leq u \leq r $, there exists $-q \leq t_u \leq p$ such that $\beta_{t_u} = \gamma_{n_u} $.
Therefore, by (\ref{10}),
	\begin{equation*}
		w=y_0 \xrightarrow{\beta_{t_1} } y_{1} \xrightarrow{ \beta_{t_2} } \cdots \xrightarrow{ \beta_{t_r} }  y_{r}=z 
	\end{equation*}
is a  directed path in $\QBG$.
We see from Proposition \ref{shellability} (2)
that this path is 
greater than or equal to the path \eqref{p_J}
in the
lexicographic order with respect to the edge labels.
In particular, we have $t_1 \geq  j_1 \geq 1$.
Since $\gamma_{n_1} = \beta_{t_1} \in \Delta^+ \cap w^{-1} \Delta^-$,
we deduce that $n_1 \geq 1$ by (\ref{inversionset}).
This implies that ${\EQB(w)}_{\bf I} \subset {\EQB}(w)_{\bf K}$.
\end{proof}

Let $\mu \in W\lambda$.
Recall that $v(\mu) \in W^S$ is the minimal-length coset representative for the coset $\{ w \in W  \ | \ w\lambda=\mu\}$.
We set 
\begin{equation*}
{\QLS}^{\mu ,\infty}(\lambda) \eqdef \{\psi \in \QLS(\lambda)\ |\ \kappa(\psi) \in \lfloor \EQB (v(\mu) \lons ) \rfloor  \}.
\end{equation*}

\begin{rem}\label{qls_qlsw0}
If $w=\lon$, then we have $\EQB (\lon)=W$  
	by Proposition \ref{shellability} (1),
since in this case, we can use all the positive roots as an edge label.
If $\mu = \lambda_- = \lon \lambda$, then $v(\mu) \lons = \lon$ by \eqref{eq:red_longest},
and hence $\lfloor \EQB(v(\mu) \lons) \rfloor = W^S$.
Therefore, we have $\QLS^{\lon \lambda, \infty}(\lambda) = \QLS(\lambda)$.
%Hence 
%${\QLS}^{\lon,\infty}(\lambda)  = {\QLS}(\lambda) $.
\end{rem}

With the notation above, we set
	\begin{equation*}
		{\gch}_{\mu}{\QLS}^{\mu,\infty}(\lambda) \eqdef \sum_{\psi \in {\QLS}^{\mu,\infty}(\lambda)} e^{\wt(\psi)} q^{\dg_{\mu}(\psi)} .
	\end{equation*}
%We are now ready to state the main result of this paper.
The following is the main result of this section.
\begin{thm}\label{main}
Let $\lambda \in P^+$ be a dominant weight, and $\mu\in W\lambda$.
Then,
	\begin{equation*}
		E_{\mu}(q, \infty)= {\gch}_{\mu}{\QLS}^{\mu,\infty}(\lambda) .
	\end{equation*}
\end{thm}

\subsection{Proof of Theorem~\ref{main}}
Let $\lambda \in P^+$ be a dominant weight,  $\mu \in W\lambda$,
and set $S=S_\lambda = \{ i \in I \ | \ \langle \lambda , \alpha^\lor_i \rangle =0 \}$.
%;
%recall that 
%$S = \{ i \in I \ | \ \langle \lambda , \alpha^\lor_i \rangle =0 \}$, and
%$\lambda_- = \lon \lambda$.
In this subsection, we give a bijection
	\begin{equation*}
		\Xi : \overleftarrow{{\QB}}(\id; m_{\mu}) \rightarrow {\QLS}^{\mu, \infty}(\lambda)
	\end{equation*}
that preserves  weights and degrees in order to prove Theorem~\ref{main}.

We fix reduced expressions
%%%%%%%%%%
%%%%%%%%%%
%eq:red_exp1
%eq:red_exp2
%%%%%%%%%%
%%%%%%%%%%
\begin{align}
v(\lambda_-) v(\mu)^{-1}&=  s_{i_1}\cdots s_{i_{K}} , \nonumber \\ 
v(\mu) &= s_{i_{K+1}}\cdots s_{i_{M}}, \label{eq:red_exp1} \\
\lons &= s_{i_{M+1}}\cdots s_{i_N} \label{eq:red_exp2}
\end{align}
for 
$v(\lambda_-)v(\mu)^{-1}$, $v(\mu)$, and $\lons$, respectively; 
recall that $\lambda_- = \lon \lambda$.
Then, by Lemma \ref{vmu} (4),
$
v(\lambda_-) =  s_{i_1}\cdots s_{i_{M}}
$
is a reduced expression for $v(\lambda_-)$.
As in \S 3.1,
we use the weak reflection order $\prec$ on $\Delta^+$
introduced in Remark \ref{restriction}
 (which satisfies
(\ref{bunkai}))
determined by the reduced expressions above for $v(\lambda_-)$ and $\lons$.
Also, we use the total order $\prec'$ on $\ptrr \cap m_{\lambda_-}^{-1} \ntrr$
defined just before Proposition~\ref{goodreducedexpression},
and take the reduced expression 
$m_{\lambda_-} = u s_{\ell_1}\cdots s_{\ell_L}$
for $m_{\lambda_-}$ given by Proposition \ref{goodreducedexpression};
recall that
$u s_{\ell_k} = s_{i_k} u$ for $1 \leq k \leq M$.
It follows from Lemma \ref{vmu} (3)
that $\left(v(\mu) v(\lambda_-)^{-1} \right)m_{\lambda_-} = m_{\mu}$
and
$- \ell(v(\mu) v(\lambda_-)^{-1} ) + \ell(m_{\lambda_-})  
=  \ell(m_{\mu})$.
Moreover, we see that
\begin{align*}
 \left( v(\mu) v(\lambda_-)^{-1} \right)   m_{\lambda_-}
& = 
\left( s_{i_{K}}\cdots s_{i_1} \right)  us_{\ell_{1}} \cdots s_{\ell_L} \\
& \overset{\rm Lemma  \ \ref{lengthadditive}}{=} 
u s_{\ell_{K}}\cdots s_{\ell_1} s_{\ell_1} \cdots s_{\ell_L}
=
us_{\ell_{K+1}}\cdots s_{\ell_L},
\end{align*}
and hence
$m_{\mu}=us_{\ell_{K+1}}\cdots s_{\ell_L}$
is a reduced expression for $m_{\mu}$.
In particular, when $\mu = \lambda$
(note that $v(\lambda)=e$), $m_{\lambda}=us_{\ell_{M+1}}\cdots s_{\ell_L}$
is a reduced expression for $m_{\lambda}$.

Also,
recall from Remark \ref{restriction} and the beginning of \S 3.1 that
$\beta_k = s_{i_N} \cdots s_{i_{k+1}}\alpha_{i_k}$,
$1 \leq k \leq N$,
and
$\beta^\OS_k = s_{\ell_L} \cdots s_{\ell_{k+1}}\alpha^\lor_{\ell_k}$,
$1 \leq k \leq L$.

\begin{rem}
Keep the notation above.
We have
\begin{align*}
\ptrr \cap m^{-1}_{\lambda_-} \ntrr &= \{ \beta^\OS_1 , \ldots , \beta^\OS_L \}, \\
\ptrr \cap m^{-1}_{\mu} \ntrr &= \{ \beta^\OS_{K+1} , \ldots , \beta^\OS_L \}, \\
\ptrr \cap m^{-1}_{\lambda} \ntrr &= \{ \beta^\OS_{M+1} , \ldots , \beta^\OS_L \}. 
\end{align*}
In particular, we have 
$\ptrr \cap m^{-1}_{\lambda} \ntrr
\subset
\ptrr \cap m^{-1}_{\mu} \ntrr
\subset
\ptrr \cap m^{-1}_{\lambda_-} \ntrr$.
\end{rem}

\begin{lem}[\normalfont{\cite[(2.4.7) (i)]{M}}]\label{shortest}
If we denote by $\varsigma$ the characteristic function of $\Delta^-$, i.e.,
\begin{align*}
\varsigma(\gamma) \eqdef
\left\{
			\begin{array}{ll}
				0  & \ \ \mbox{if} \ \gamma \in \Delta^+  ,  \\
				1  & \ \ \mbox{if} \ \gamma \in \Delta^-  ,
			\end{array}
\right.
\end{align*}
then
\begin{equation*}
\ptrr \cap m^{-1}_{\mu} \ntrr
=
\{
\alpha^\lor + a\widetilde{\delta}
\ | \
\alpha \in \Delta^-,
0 < a < \varsigma(v(\mu) v(\lambda_-)^{-1} \alpha) +\langle  \lambda , \lon \alpha^\lor \rangle
\}.
\end{equation*}
\end{lem}

\begin{rem}\label{affinereflectionorder}
Let
$\gamma_1, \gamma_2 ,\ldots, \gamma_r \in \ptrr \cap m^{-1}_\mu \ntrr$,
and 
consider the sequence 
%%extra space
\\
$\left( y_0 , y_1 , \ldots , y_r ; \gamma_1 , \gamma_2, \ldots , \gamma_r \right)$
defined by
$y_0=m_\mu$, and $y_i =y_{i-1}s_{\gamma_i}$ for $1 \leq i \leq r$.
Then, the sequence 
$\left( y_0 , y_1 , \ldots , y_r ; \gamma_1 , \gamma_2, \ldots , \gamma_r \right)$
is an element of 
$\overleftarrow{{\QB}}({\id}; m_{\mu})$
if and only if the following conditions hold{\rm :}
\begin{enu}
\item
$\gamma_1\prec' \gamma_2 \prec' \cdots \prec' \gamma_r$,
where the order $\prec'$ is 
the weak reflection order on $ \ptrr \cap m^{-1}_\mu \ntrr$
introduced at the beginning of \S 3.3;

\item
$\dir(y_{i-1}) \xleftarrow{-\left( \overline{\gamma_i} \right)^\lor} \dir(y_{i})$
is
an edge of $\QBG$ for $1\leq i \leq r$.
\end{enu}
\end{rem}

In the following,
we define a map
$\Xi : \overleftarrow{{\QB}}(\id; m_{\mu}) \rightarrow {\QLS}^{\mu , \infty}(\lambda)$.
Let $p^{\OS}_{J}$ be an arbitrary element of  $\overleftarrow{{\QB}}(\id; m_{\mu})$ of the form
	\begin{equation*}
		p^{\OS}_{J} = \left( m_{\mu} = z^{\OS}_0 ,  z^{\OS}_{1} , \ldots , z^{\OS}_{r}
; \beta^{\OS}_{j_1} , \beta^{\OS}_{j_2}, \ldots , \beta^{\OS}_{j_r} \right) \in \overleftarrow{{\QB}}(\id; m_{\mu}),
	\end{equation*}
with $J = \{ j_1 < \cdots < j_r\} \subset \{ K+1 , \ldots , L \}$.
	We set $x_k \eqdef {\dir}(z^{\OS}_k)$,
	$0 \leq k \leq r$. Then, by the definition of $\overleftarrow{\QB}({\id} ; m_{\mu})$, 
	\begin{equation}\label{2.15.5}
		v(\mu)v(\lambda_-)^{-1} \overset{\rm Lemma \ \ref{vmu}}{=}  x_0 \xleftarrow{- \left( \overline{ \beta^{\OS}_{j_1} } \right)^{\lor}  }   x_1
		\xleftarrow{- \left( \overline {\beta^{\OS}_{j_2} } \right)^{\lor} } \cdots  \xleftarrow{ - \left( \overline{ \beta^{\OS}_{j_r} } \right)^{\lor} }  x_r
	\end{equation}
	is a directed path in $\QBG$.
	We take $0 = u_0 \leq u_1 < \cdots < u_{s-1} < u_s=r$ and $0 = \sigma_0 \leq \sigma_1 < \cdots <\sigma_{s-1} < 1 = \sigma_{s}$ in such a way that
(see (\ref{C}))
	\begin{equation}\label{2.16}
		 \underbrace{0 = d_{j_1} = \cdots = d_{j_{u_1}} }_{=\sigma_0}
		< \underbrace{d_{j_{u_1 +1}} = \cdots =d_{j_{u_2}}}_{=\sigma_1} < \cdots <
		\underbrace{ d_{j_{u_{s-1}+1}} = \cdots =d_{j_r} }_{=\sigma_{s-1}} <1 = \sigma_{s};
	\end{equation}
	note that $d_{j_1}>0$ if and only if $u_1=0$.
	We set
	$w'_p \eqdef x_{u_p}$ for $0 \leq p \leq s-1$, and
	$w'_s \eqdef x_r$. 
	Then, by taking a subsequence of (\ref{2.15.5}), we obtain the following  directed path in $\QBG$
	for each $0 \leq p \leq s-1$:
	\begin{equation*}
		w'_p = x_{{u_p}} \xleftarrow{- \left( \overline{ \beta^{\OS}_{j_{u_p +1}} } \right)^{\lor} }   x_{{u_p +1}} 
		\xleftarrow{- \left( \overline {  \beta^{\OS}_{j_{u_p +2}} } \right)^{\lor} } \cdots  
		\xleftarrow{ -\left( \overline{  \beta^{\OS}_{j_{u_{p+1}}} } \right)^{\lor} }  x_{{u_{p+1}}} = w'_{p+1}.
	\end{equation*}
	Multiplying this directed path on the right by $\lon$, 
	we obtain the following directed path in $\QBG$
	for each $0 \leq p \leq s-1$
	(see Lemma \ref{involution}):
	\begin{equation}\label{w'_p}
		w_p \eqdef
		w'_p \lon =
		x_{{u_p}} \lon \xrightarrow{ \lon \left( \overline{ \beta^{\OS}_{j_{u_p +1}} } \right)^{\lor} }  \cdots  
		\xrightarrow{ \lon \left( \overline{ \beta^{\OS}_{j_{u_{p+1}}} } \right)^{\lor} }  x_{{u_{p+1}}} \lon
		= w'_{p+1} \lon \eqdef w_{p+1}.
	\end{equation}
	Note that the edge labels of this directed path are increasing
in the weak reflection order $\prec$ on $\Delta^+$ introduced at the beginning of \S 3.3
(see Lemma \ref{remark2.11})
	  and lie in $\Delta^+ \setminus \Delta^+_S$;
	this property will be used to give the inverse to $\Xi$.
	Because
	\begin{equation*}
		(1-\sigma_p) \langle \lambda , \lon \overline{ {\beta^{\OS}_{j_u}} }   \rangle
		=
		(1-d_{j_u}) \langle \lambda , \lon \overline{ {\beta^{\OS}_{j_u}} }   \rangle
		=
		-\frac{a_{j_u}}{\langle \lambda_- ,  -\overline{ \beta^{\OS}_{j_u} }   \rangle}
		\langle \lambda_- , \overline{ \beta^{\OS}_{j_u} }   \rangle
		=
		 a_{j_u} \in \mathbb{Z}
	\end{equation*}
	for $u_p +1 \leq u \leq u_{p+1}$, $0 \leq p \leq s-1$,
	we find that (\ref{w'_p}) is a  directed  path in $\QBG_{(1-\sigma_p) \lambda}$ for $0 \leq p \leq s-1$.
	Therefore, by Lemma \ref{8.1}, there exists a  directed path in $\QBG^S_{(1-\sigma_p) \lambda}$ from $\lfloor w_p \rfloor$ to $\lfloor w_{p+1} \rfloor$,
where
$S = \{ i \in I \ | \ \langle \lambda , \alpha^\lor_i \rangle =0 \}$.
Also, we claim that $\lfloor w_p \rfloor \neq \lfloor w_{p+1} \rfloor$ for $1 \leq p \leq s-1$. Suppose,
for a contradiction, that $\lfloor w_p \rfloor = \lfloor w_{p+1} \rfloor$
for some $p$.
Then, $w_p W_S = w_{p+1} W_S$,
and hence $\min(w_{p+1}W_S, w_p) = \min(w_{p}W_S, w_p) =w_p$.
Recall that 
the directed path (\ref{w'_p}) is a path in QBG from $w_{p}$ to $w_{p+1}$ whose labels are increasing and lie in $\Delta^+ \setminus \Delta^+_S$.
%Therefore, the directed path (\ref{w'_p}) is a shortest path from $w_{p}$ to $w_{p+1}$ by Proposition \ref{shellability} (1).
By Lemma \ref{8.5} (1), (2), the directed path (\ref{w'_p}) is a shortest path in QBG from $w_p$ to $\min(w_{p+1}W_S, w_p) = \min(w_{p}W_S, w_p) =w_p$,
which implies that the length of the directed path (\ref{w'_p}) is equal to $0$.
Therefore, $\{ j_{u_p +1}, \ldots , j_{u_{p+1}} \} = \emptyset$, and hence $u_p = u_{p+1}$,
which contradicts the fact that $u_p < u_{p+1}$.

	Thus we obtain 
	\begin{equation}\label{D}
		\psi \eqdef 
		(\lfloor w_{s} \rfloor , \lfloor w_{s-1} \rfloor ,\ldots \ , \lfloor w_{1} \rfloor ; 1-{\sigma}_{s},\ldots, 1-{\sigma}_{0})  \in \QLS(\lambda).
	\end{equation}
We now define $\Xi (p^{\OS}_{J}) \eqdef \psi$.

\begin{lem}\label{final_direction}
Keep the notation and setting  above{\rm;}
let $s_{i_{K+1}} \cdots s_{i_M} s_{i_{M+1}} \cdots s_{i_N}$ be a reduced expression for 
$v(\mu)\lons$ obtained
by combining \eqref{eq:red_exp1} and \eqref{eq:red_exp2}. 
Then, $\lfloor w_1 \rfloor \in \lfloor \EQB(v(\mu)\lons) \rfloor$.
Hence we obtain a map
$
\Xi:
\overleftarrow{{\QB}}(\id; m_{\mu})
\rightarrow
\QLS^{\mu , \infty}(\lambda)
$.
\end{lem}

\begin{proof}
Since it is clear that 
$v(\mu) \in \lfloor \EQB(v(\mu) \lons ) \rfloor$,
we may assume that $\lfloor w_1 \rfloor \neq v(\mu)$.

Since $z^\OS_0 = m_{\mu}$, we have $w'_0 = x_0 = \dir (z^\OS_0) = v(\mu)v(\lambda_-)^{-1}$.
It follows that $w_{0} = w'_0 \lon = \left( v(\mu)v(\lambda_-)^{-1}\right) \lon 
\overset{\rm Lemma \ \ref{vmu}\ (2)}{=}
 v(\mu) \lons $.
If $u_1 = 0$, then we obtain $w_1 =w_0 = v(\mu)\lons$ , contrary to the assumption that $\lfloor w_1 \rfloor \neq v(\mu)$.
Hence it follows that $u_1 \geq 1$.
This implies that $j_{u_1} \leq M$ by  the definition of $u_1$ in (\ref{2.16})  and the proof of Lemma \ref{lengthadditive}.
Thus, we obtain $K+1 \leq j_1 < j_2 < \cdots < j_{u_1} \leq M$.

Now, consider the directed  path (\ref{w'_p}) in the case $p=0$.
This is a (nontrivial) directed path in $\QBG$ from $w_0=v(\mu)\lons$ to $w_1$ whose edge labels are increasing
in the weak reflection order $\prec$ on $\Delta^+$
introduced at the beginning of \S 3.3.
Because these edge labels are $\lon \left( \overline{ \beta^{\OS}_{j_{k}} }
 \right)^{\lor} = \beta_{j_{k}}= s_{i_N}\cdots s_{i_{j_k+1}}\alpha_{i_{j_k}}$ for $ 1 \leq k \leq u_1 $
(the first equality follows from the proof of Lemma \ref{lengthadditive}),
it follows from the fact that $K+1 \leq j_1 < j_2 < \cdots < j_{u_1} \leq M$ and Remark \ref{usingshell} 
(recall that we take a reduced expression for $\lon$ given by concatenating the reduced expressions for $v(\lambda_-)v(\mu)^{-1} $ and $v(\mu)\lons$)
that $w_1 \in \EQB(v(\mu) \lons )$.
Hence $\lfloor w_1 \rfloor \in \lfloor \EQB(v(\mu)\lons ) \rfloor$.
\end{proof}

\begin{prop}\label{bijective}
		The map $\Xi:\overleftarrow{{\QB}}({\id}; m_{\mu}) \rightarrow \QLS^{\mu ,\infty}(\lambda)$
is bijective.
\end{prop}

\begin{proof}
Let us give the inverse to $\Xi$.
Take an arbitrary
$
		\psi =(y_{1} , \ldots \ , y_{s} ; {\tau}_{0},\ldots, {\tau}_{s})  \in \QLS^{\mu, \infty}(\lambda)
$.
By convention, we set $y_{s+1}=v(\mu) \in W^S$.
We define the elements $v_{p}$, $1\leq p \leq s+1$, by:
$v_{s+1} =v(\mu) \lons $,
and
$v_{p} = \min(y_{p}W_S , \leq_{v_{p+1}})$  for $1\leq p \leq s$.

Because there exists a  directed path in $\QBG^S_{\tau_p \lambda}$ from $y_{p+1}$ to $y_{p}$ for $1\leq p \leq s-1$,
we see from Lemma \ref{8.5} (2), (3) that
there exists a unique  directed path 
\begin{equation}\label{path3}
v_p \xleftarrow{-\lon \gamma_{p,1}} \cdots \xleftarrow{-\lon \gamma_{p,t_{p}}} v_{p+1}
\end{equation}
in $\QBG_{\tau_p \lambda}$ from $v_{p+1}$ to $v_p$ whose edge labels $-\lon\gamma_{p, t_p}, \ldots, -\lon \gamma_{p, 1}$ are increasing 
in the weak reflection order $\prec$
and lie in $\Delta^+ \setminus \Delta^+_S$
for $1\leq p \leq s-1$;
we remark that this is also true for $p=s$, since
$\tau_s=1$.
Multiplying this directed path on the right by $\lon$, 
we obtain by Lemma \ref{involution} the following directed  paths
	\begin{equation*}
		v_{p,  0} \eqdef v_{p}\lon  
		\xrightarrow{ \gamma_{p, 1} }  v_{p,  1} 
		\xrightarrow{ \gamma_{p, 2} } 
		\cdots \xrightarrow{ \gamma_{p, t_p} } v_{p+1}\lon 
		\eqdef v_{p,  t_p},
\ \
1 \leq p \leq s.
	\end{equation*}
Concatenating these paths for $1\leq p\leq s$, we obtain the following directed  path
	\begin{align}\label{recover}
		&v_{1,0} \xrightarrow{ \gamma_{1, 1} } \cdots  
		\xrightarrow{ \gamma_{1, t_1} }v_{1,  t_1}=v_{2,  0}
		\xrightarrow{ \gamma_{2, 1} } 
		\cdots 
 \xrightarrow{ \gamma_{s-2, t_{s-2}} }v_{s-2, t_{s-2}} = v_{s-1,  0}  \xrightarrow{ \gamma_{s-1, 1} }  \cdots
 \nonumber
\\
&\cdots \xrightarrow{ \gamma_{s-1, t_{s-1}} }  v_{s-1,  t_{s-1}} =
		v_{s,  0}  \xrightarrow{ \gamma_{s, 1} }  \cdots \xrightarrow{ \gamma_{s, t_s}}   v_{s,  t_s} = v_{s+1,0}=v(\mu)v(\lambda_-)^{-1}
	\end{align}
	in $\QBG$. 
	Now, for $1\leq p \leq s$ and $1 \leq m \leq t_p$,
	we set $d_{p,m}  \eqdef 1-\tau_{p} \in \mathbb{Q}\cap [0,1)$, 
	$a_{p,m} \eqdef (d_{p,m}-1)\langle \lambda_- , {\gamma}^{\lor}_{p,m}   \rangle$, 
	and 
	$\widetilde{\gamma}_{p,m}  \eqdef a_{p,m} \tilde{\delta} - \gamma_{p,m}^{\lor}$.

\begin{nclaim}
$\widetilde{\gamma}_{p,m}   \in 
\ptrr \cap m^{-1}_{\mu} \ntrr$.
\end{nclaim}

\noindent $Proof \ of \ Claim \ 1.$
Since $\tau_p > 0$, and  since the path (\ref{path3}) is a  directed path in $\QBG_{\tau_p \lambda}$ whose edge labels are  increasing and lie in $\Delta^+ \setminus \Delta^+_S$,
we obtain 
$a_{p,m}=- \tau_p \langle \lambda_- , \gamma_{p,m}^\lor \rangle 
=\tau_p \langle \lambda , -\lon \gamma_{p,m}^\lor \rangle 
\in \mathbb{Z}_{>0}$.

We will show that $ a_{p,m} < \varsigma(v(\mu)v(\lambda_-)^{-1} (- \gamma_{p,m}) ) +\langle  \lambda , \lon \left(-\gamma_{p,m}^\lor \right) \rangle$.
Here we note that the inequality $\langle \lambda , \lon \left(-\gamma_{p,m}^\lor \right) \rangle = -\langle \lambda_- , \gamma_{p,m}^\lor \rangle \geq  -\tau_p \langle \lambda_- , \gamma_{p,m}^\lor \rangle=a_{p,m}$ holds,
with equality if and only if $p=s$.
Hence
it suffices to consider the case $p=s$.
In the case $p=s$, the path (\ref{path3}) is the unique  directed path in $\QBG$ from $v(\mu)\lons =v_{s+1}$ to $v_s$ whose edge labels are  increasing and lie in $\Delta^+ \setminus \Delta^+_S$.
Also, since  $\psi \in \QLS^{\mu, \infty} (\lambda)$ and  $\kappa(\psi)=y_s = \lfloor v_s \rfloor$,
we find that there exists $v'_s \in \EQB(v(\mu)\lons )$ such that $\lfloor v'_s \rfloor = y_s$.
By the definition of $\EQB(v(\mu) \lons )$, there exists a unique  directed path in $\QBG$ from $v(\mu)\lons$ to $v'_s$ whose edge labels are  increasing;
we see from (\ref{bunkai}) that this directed path is obtained as the concatenation of the two directed paths:
the one whose edge labels lie in $\Delta^+ \setminus \Delta^+_S$, and
the one whose edge labels lie in $\Delta^+_S$.
Therefore, by removing all the edges whose labels lie in $\Delta^+_S$ from the path above,
we obtain a  directed path in $\QBG$ from $v(\mu)\lons$ to some $v''_s \in y_s W_S \cap \EQB(v(\mu)\lons )$ whose edge labels are  increasing and lie in $\Delta^+ \setminus \Delta^+_S$.
Here, since $\lfloor v_s \rfloor = \lfloor v''_s \rfloor$ 
and $v_s = \min (y_s W_S, \leq_{v(\mu) \lons })$,
Lemma \ref{8.5} (2) shows that $v_s= v''_s$.
Hence we have $v_s \in \EQB(v(\mu) \lons )$.
Moreover, by the definition of $\EQB(v(\mu) \lons )$,
the edge labels $ - \lon \gamma_{s ,1} , \ldots ,  - \lon \gamma_{s ,t_s}$  in the given directed path in $\QBG$ from  $v(\mu) \lons =v_{s+1}$ to $v_s$  are elements of $ \Delta^+ \cap (v(\mu) \lons )^{-1}\Delta^-$,
and hence $v(\mu) \lons \left(- \lon \gamma_{s ,m} \right) 
\overset{\rm Lemma \ \ref{vmu}\ (2)}{=}
v(\mu)v(\lambda_-)^{-1} (- \gamma_{s,m}) \in \Delta^-$.
Therefore, in the case $p=s$, we have $\varsigma(v(\mu)v(\lambda_-)^{-1} (- \gamma_{s,m}) ) =1$.
Thus we have shown that $ a_{s,m} = \langle  \lambda , \lon \left(-\gamma_{s,m}^\lor \right) \rangle < \varsigma(v(\mu)v(\lambda_-)^{-1}  (-\gamma_{s,m} ) ) +\langle  \lambda , \lon \left(-\gamma_{s,m}^\lor \right) \rangle$.
Hence we conclude that $\widetilde{\gamma}_{p,m} \in \ptrr \cap m^{-1}_{\mu} \ntrr$ by Lemma \ref{shortest}.
\bqed

\begin{nclaim}
\mbox{}
\begin{enu}
\item
%The set $J'$ above satisfies the following: 
We have
\begin{equation*}
\widetilde{\gamma}_{s,t_s} \prec'  \cdots \prec' \widetilde{\gamma}_{s, 1} 
\prec'  
\widetilde{\gamma}_{s-1,t_{s-1}} \prec' \cdots \prec' \widetilde{\gamma}_{1,1},
\end{equation*}
where $\prec'$ denotes the weak reflection order on $\ptrr \cap m^{-1}_{\lambda_-} \ntrr$
introduced at the beginning of \S 3.3;
%(see the proof of Proposition \ref{goodreducedexpression})
we choose 
$J' = \{ j'_1 , \ldots , j'_{r'} \} \subset \{ K+1 , \ldots , L \}$ in such way that
\begin{equation*}
\left( \beta^{\OS}_{j'_1} , \cdots , \beta^{\OS}_{j'_{r'}} \right)
=
\left(  
\widetilde{\gamma}_{s,t_s} ,  \cdots , \widetilde{\gamma}_{s, 1} ,
\widetilde{\gamma}_{s-1,t_{s-1}} , \cdots , \widetilde{\gamma}_{1,1}
\right) .
\end{equation*}

\item
%For $1 \leq k \leq r'$, we set
Let $1 \leq k \leq r'$, and take $1 \leq p \leq s$,
$0< m \leq t_p$
  such that
$\left( \beta^{\OS}_{j'_1} \prec' \cdots \prec' \beta^{\OS}_{j'_{k}} \right)
=
\left(  \widetilde{\gamma}_{s,t_s} \prec'  \cdots \prec'  \widetilde{\gamma}_{p,m} 
 \right)$.
%$1 \leq p \leq s, 1 \leq m \leq t_p$.
Then, we have
$\dir(z^{\OS}_{k}) = v_{p,m-1}$.
Moreover, 
$\dir(z^{\OS}_{k-1})\xleftarrow{-\left(\overline{ \beta^\OS_{j'_k} } \right)^\lor} \dir(z^{\OS}_k)$ is an edge of $\QBG$.
\end{enu}
\end{nclaim}

%\begin{proof}[Proof of Claim $2$]
\noindent $Proof \ of \ Claim \ 2.$
(1)
It suffices to show the following:

(i)
for $1 \leq p \leq s$ and $1<m \leq t_p$,
we have
$  \widetilde{\gamma}_{p,m}\prec' \widetilde{\gamma}_{p,m-1}$;

(ii)
for $2 \leq p \leq s$,
we have
$\widetilde{\gamma}_{p,1} \prec' \widetilde{\gamma}_{p-1,t_{p-1}}$.

(i)
Because
$\frac{\langle \lambda_- , -\gamma_{p,m}^{\lor}\rangle - a_{p,m}}{\langle \lambda_- , -\gamma_{p,m}^{\lor}\rangle}=d_{p,m}$ and
$\frac{\langle \lambda_- , -\gamma_{p,m-1}^{\lor}\rangle - a_{p,m-1}}{\langle \lambda_- , -\gamma_{p,m-1}^{\lor}\rangle}=d_{p,m-1}$,
we have
\begin{align*}
\Phi (\widetilde{\gamma}_{p,m}) &=
 (d_{p,m}, - \lon \gamma_{p,m} ), \\
\Phi (\widetilde{\gamma}_{p,m-1}) &=
 (d_{p,m-1}, - \lon \gamma_{p,m-1} ).
\end{align*}
Therefore, the first component of $ \Phi (\widetilde{\gamma}_{p,m})$ is equal to that of $\Phi (\widetilde{\gamma}_{p,m-1})  $
since $d_{p,m} =1 - \tau_p =d_{p,m-1}$.
Moreover, since $- \lon \gamma_{p,m} \prec -\lon \gamma_{p,m-1}$,
we have $ \Phi (\widetilde{\gamma}_{p,m}) < \Phi (\widetilde{\gamma}_{p,m-1})  $.
This implies that $  \widetilde{\gamma}_{p,m}\prec' \widetilde{\gamma}_{p,m-1}$ by Proposition \ref{goodreducedexpression}.

(ii)
The proof of (ii) is similar to that of (i).
The first components of $\Phi (\widetilde{\gamma}_{p,1}) $ and $\Phi (\widetilde{\gamma}_{p-1,t_{p-1}}) $ are $d_{p,1}$ and $d_{p-1,t_{p-1}}$, respectively.
Since $d_{p,1}= 1-\tau_p < 1-\tau_{p-1}= d_{p-1,t_{p-1}}$,
we have $\Phi (\widetilde{\gamma}_{p,1}) < \Phi (\widetilde{\gamma}_{p-1,t_{p-1}}) $.
This implies that
$\widetilde{\gamma}_{p,1} \prec' \widetilde{\gamma}_{p-1,t_{p-1}}$.

(2)
We proceed 
by induction on $k$.
Since 
$\dir (z^\OS_0) = \dir(m_\mu)= v(\mu)v(\lambda_-)^{-1}$ 
and $\beta^{\OS}_{j'_1}  = \widetilde{\gamma}_{s,t_s}$,
we have
$\dir (z^\OS_1) = \dir (z^\OS_0) s_{-\overline{\beta^{\OS}_{j'_1}}} =v(\mu)v(\lambda_-)^{-1}  s_{\gamma_{s,t_s}} = v_{s, t_s-1}$.
Hence the assertion holds in the case $k=1$.

Assume that $\dir(z^{\OS}_{k-1}) = v_{p,m}$
for $0 \lneqq m \leq t_p$;
here we remark that $v_{p,m-1}$ is the predecessor of $v_{p,m}$ in the  directed path (\ref{recover}) since $0 \leq m-1 \leq t_{p-1}$.
Hence we have
$\dir(z^{\OS}_{k}) 
= \dir(z^{\OS}_{k-1}) s_{-\overline{ \beta^\OS_{j'_k} } }
= v_{p,m}s_{\gamma_{p,m}}
\overset{(\ref{recover})}=v_{p,m-1}$.
Also, since (\ref{recover}) is a directed path in $\QBG$, 
$ v_{p,m} = \dir(z^{\OS}_{k-1})\xleftarrow{-\left(\overline{ \beta^\OS_{j'_k} } \right)^\lor} \dir(z^{\OS}_k) = v_{p,m-1} $ is an edge of $\QBG$.
%\end{proof}
\bqed

Since $J' = \{ j_1 , \ldots , j'_{r'} \} \subset \{ K+1 , \ldots , L \}$,
we can define an element $p^{\OS}_{J'}$ to be
$\left( m_{\mu} = z^{\OS}_0 ,  z^{\OS}_{1} , \ldots , z^{\OS}_{r'}
; \beta^{\OS}_{j'_1} , \beta^{\OS}_{j'_2}, \ldots , \beta^{\OS}_{j'_{r'}} \right)$,
where
$z^{\OS}_{0}=m_{\mu}$, $z^{\OS}_{k}=z^\OS_{k-1} s_{\beta^{\OS}_{j'_{k}}}$ for $1\leq k \leq r'$;
it follows from Remark \ref{affinereflectionorder} and Claim 2 that
 $p^{\OS}_{J'} \in \overleftarrow{\QB}({\id}; m_{\mu})$.
Hence we can define a map $\Theta : \QLS^{\mu, \infty}(\lambda) \rightarrow \overleftarrow{\QB}({\id}; m_{\mu})$ by
$\Theta (\psi) \eqdef p^{\OS}_{J'}$.

%Hence we can define  
%$
%p^{\OS}_{J'} 
%= \left( m_{\mu} = z^{\OS}_0 ,  z^{\OS}_{1} , \ldots , z^{\OS}_{r'}
%; \beta^{\OS}_{j'_1} , \beta^{\OS}_{j'_2}, \ldots , \beta^{\OS}_{j'_{r'}} \right) 
%$
%by 
%$z^{\OS}_{0}=m_{\mu}$, $z^{\OS}_{k}=z^\OS_{k-1} s_{\beta^{\OS}_{j'_{k}}}$ for $1\leq k \leq r'$.
%
%If Claim 2 below holds,
%then
%we obtain $p^{\OS}_{J'}\in \overleftarrow{\QB}({\id}; m_{\mu})$ 
%by Remark \ref{affinereflectionorder},
%and hence we can define $\Theta :\QLS^{\mu, \infty}(\lambda) \rightarrow \overleftarrow{\QB}({\id}; m_{\mu})$ by
%$\Theta (\psi) \overset{\text{def}}{=} p^{\OS}_{J'}$.

It remains to show that the map $\Theta$ is the inverse to the map $\Xi$, i.e., the following two claims.

\begin{nclaim}
For $\psi =(y_{1} , \ldots \ , y_{s} ; {\tau}_{0},\ldots, {\tau}_{s})  \in {\QLS}(\lambda)$,
we have $\Xi \circ \Theta(\psi)=\psi$.
\end{nclaim}

\begin{nclaim}
For $p^{\OS}_{J} 
= \left( m_{\mu} = z^{\OS}_0 ,  z^{\OS}_{1} , \ldots , z^{\OS}_{r}
; \beta^{\OS}_{j_1} , \beta^{\OS}_{j_2}, \ldots , \beta^{\OS}_{j_r} \right) \in \overleftarrow{\QB}(\id; m_{\mu})$,
we have $\Theta \circ   \Xi (p^{\OS}_{J}) = p^{\OS}_{J}$.
\end{nclaim}

%\begin{proof}[Proof of Claim $3$]
\noindent $Proof \ of \ Claim \ 3.$
We set $\Theta(\psi)= p_{J'}^\OS$, with $J'=\{ j'_1 , \ldots , j'_r \}$;
in the following description of $\Theta (\psi) = p^\OS_{J'}$,
we employ the notation  $u_p$, $\sigma_p$, $w'_p$, and $w_p$ 
used  in the definition of $\Xi (p^\OS_{J})$.

For $1 \leq k \leq r'$, if we set $\beta^{\OS}_{j'_k}=\widetilde{\gamma}_{p,m}$
with $m>0$,
then we have
$d_{j'_k}=
1 + \frac{\dg(\beta^{\OS}_{j'_k})}{\langle \lambda_- , -\overline{ \beta^{\OS}_{j'_k} } \rangle}
=
1 + 
\frac{\dg(\widetilde{\gamma}_{p,m})}{\langle \lambda_- , -\overline{\widetilde{\gamma}_{p,m}} \rangle}
=
1 + \frac{a_{p,m}}{\langle \lambda_- , \gamma_{p,m}^\lor \rangle}
=
d_{p,m}
$.
Therefore,
the sequence (\ref{2.16}) determined by $\Theta(\psi)=p^{\OS}_{J'}$ is 
\begin{equation}\label{2.17}
		\underbrace{ 0 = d_{s,t_s} = \cdots d_{s,1} }_{=1-\tau_s}
		< \underbrace{d_{s-1,t_{s-1}} = \cdots =d_{s-1, 1}}_{=1-\tau_{s-1}}
		 < \cdots <
		\underbrace{ d_{1,t_1} = \cdots =d_{1,1} }_{=1-\tau_{1}} <1 =1- \tau_{0} = \sigma_s.
	\end{equation}
Because the sequence (\ref{2.17}) of rational numbers is just the sequence (\ref{2.16}) for 
$\Theta ( \psi ) = p_{J'}^\OS$,
we deduce that $\beta^\OS_{j'_{u_p}} = \widetilde{\gamma}_{s-p+1,1}$ for $1\leq p \leq s$,
and
$\sigma_p = 1 - \tau_{s-p}$ for $0 \leq p \leq s$.
Therefore, we have $w'_p = \dir(z^\OS_{u_p})=v_{s-p+1,0}$ and $w_p = v_{s-p+1,0}\lon =v_{s-p+1}$.
Since $\lfloor w_p \rfloor = \lfloor v_{s-p+1} \rfloor = y_{s-p+1}$,
we conclude that $\Xi \circ \Theta(\psi) = (\lfloor w_s \rfloor ,\ldots , \lfloor w_1 \rfloor ; 1 - \sigma_s, \ldots , 1-\sigma_0)=(y_1 ,\ldots, y_s ; \tau_0 ,\ldots , \tau_s)= \psi$.
%\end{proof}
\bqed

%\begin{proof}[Proof of Claim $4$]
\noindent $Proof \ of \ Claim \ 4.$
%\footnote{
%indices between $\tau$ and $\sigma$, $y$ and $w$
%should be complementary, not equal;
%end with some element in 
%$w_{p+1}W_S = v_{s-p}W_S$ (not $w_p W_S$, in the second paragraph).}
%
We set $\Xi(p^{\OS}_J)=\psi$,
and write it as $\psi = (y_1 ,\ldots, y_s ; \tau_0 ,\ldots , \tau_s)$,
where $y_p = \lfloor w_{s+1-p} \rfloor$  for $1\leq p \leq s$
and $\tau_p = 1- \sigma_{s-p}$ for $0 \leq p \leq s$
in the notation of (\ref{D}) (and the comment preceding it).
Also,
in the following description of $\Xi(p^{\OS}_J)=\psi$,
we employ the notation  $v_{p,m}$, $d_{p,m}$, $a_{p,m}$, $\gamma_{p,m}$, $\widetilde{\gamma}_{p,m}$, and $J'$ used in the definition of $\Theta(\psi)$.

Recall that $w_0= v(\mu) \lons =v_{s+1}$.
For $0 \leq p \leq s-1$,
\begin{equation*}
v_{s-p+1} 
\xrightarrow{- \lon \gamma_{s-p, t_{s-p}}}
\cdots  \xrightarrow{- \lon \gamma_{s-p,1}}
v_{s-p}
\end{equation*}
is a  directed path in $\QBG$ whose edge labels are  increasing and lie in $\Delta^+ \setminus \Delta^+_S$
(see (\ref{path3})).
Now we can show by induction on $p$ that $w_p = v_{s-p+1}$ for $1 \leq p \leq s$.
Indeed,
if $w_p=v_{s-p+1}$, then both of the path above and the path (\ref{w'_p}) start from $w_p$ and end with some element in $w_{p+1} W_S = v_{s-p}W_S$ 
(this equality follows from the definition of $v_{s-p}$), 
and have increasing edge labels in $\Delta^+ \setminus \Delta^+_S$.
Therefore, by Lemma \ref{8.5} (2), we deduce that the ends of these two paths are identical,
and hence that we have $w_{p+1} = v_{s-p}$.
%By induction on $p$, we deduce that $w_{p} = v_{s-p+1}$ for $0 \leq p \leq s$.
Moreover,
since these two paths are identical, so are the edge labels of them:
\begin{equation*}
\left(\lon \left( \overline{\beta^{\OS}_{j_{u_p+1}} }\right)^\lor  \prec \cdots \prec
\lon \left( \overline{\beta^{\OS}_{j_{u_{p+1}}}} \right)^\lor \right)
=
\left( -\lon \gamma_{s-p,t_{s-p}} \prec \cdots \prec -\lon \gamma_{s-p,1}\right),
\ \ 
0 \leq p \leq s-1.
\end{equation*}
From the above, we have $u_{p+1}-u_p=t_{s-p}$ and
 $-\left( \overline{\beta^{\OS}_{j_{u_p+k}} }\right)^\lor 
=
\gamma_{s-p, t_{s-p}-k+1}
$
for $0\leq p \leq s-1$, $1 \leq k \leq t_{s-p}$.
Since
$ \sigma_p = d_{j_{u_{p}+1}}=\cdots=d_{j_{u_{p+1}}} $ for $0 \leq p \leq s-1$,
$1-\sigma_p = \tau_{s-p}$ for $0 \leq p \leq s$, 
and $ 1-\tau_{s-p} = d_{s-p , 1}=\cdots=d_{s-p, t_{s-p}}$ for $0\leq p\leq s-1$,
we see that for $1\leq k \leq t_{s-p}$,
\begin{align*}
\beta^{\OS}_{j_{u_p+k}} &=
\overline{\beta^{\OS}_{j_{u_p+k}}} + a_{j_{u_p+k}} \widetilde{\delta} \\
&=
\overline{\beta^{\OS}_{j_{u_p+k}}} - (d_{j_{u_p+k}}-1) 
\langle \lambda_-, \overline{\beta^{\OS}_{j_{u_p+k}}} \rangle \widetilde{\delta} \\
&=
-\gamma^\lor_{s-p, t_{s-p}-k+1}+ (d_{s-p, t_{s-p}-k+1 }-1) 
\langle \lambda_- \gamma^\lor_{s-p, t_{s-p}-k+1} \rangle \widetilde{\delta} \\
&=
-\gamma^\lor_{s-p, t_{s-p}-k+1}+ a_{s-p, t_{s-p}-k+1 } \widetilde{\delta} \\
&=
\widetilde{\gamma}_{s-p, t_{s-p}-k+1}.
\end{align*}
Therefore, we have
$$
\left( \beta^{\OS}_{j_{u_p+1}} \prec' \cdots \prec'
\beta^{\OS}_{j_{u_{p+1}}}  \right)
=
\left( \widetilde{\gamma}_{s-p,t_{s-p}} \prec' \cdots \prec' \widetilde{\gamma}_{s-p,1} \right),
\ \
0\leq p \leq s-1.
$$
Concatenating the sequences above for $0 \leq p \leq s-1$,
we obtain 
\begin{align*}
\left( \beta^{\OS}_{j_1} \prec' \cdots \prec' \beta^{\OS}_{j_{r}} \right)
&=
\left(  \widetilde{\gamma}_{s,t_s} \prec'  \cdots \prec' \widetilde{\gamma}_{s, 1} \prec'  \widetilde{\gamma}_{s-1,t_{s-1}} \prec' \cdots \prec' \widetilde{\gamma}_{1,1}
\right)\\
&=
\left( \beta^{\OS}_{j'_1} \prec' \cdots \prec' \beta^{\OS}_{j'_{r'}} \right).
\end{align*}
Hence the set $J'$ determined by $\Xi (p^{\OS}_{J})  = \psi$ is identical to $J$.
Thus we conclude that
$\Theta \circ   \Xi (p^{\OS}_{J}) 
=p^{\OS}_{J'}
=p^{\OS}_{J}$.
%\end{proof}
\bqed

This completes the proof of Proposition \ref{bijective}.
\end{proof}
We recall from \eqref{eq:dfn_deg} and \eqref{eq:dfn_wt}
that 
$\dg(\beta)$ is defined by: $\beta = \overline{\beta} + \dg (\beta)\widetilde{\delta}$ for $\beta \in \mathfrak{h} \oplus \mathbb{C}\widetilde{\delta}$,
and 
$\wt(u) \in P$ and $\dir (u)$ are defined by:
$u=t({{\wt}(u)}) \dir (u)$ for $u \in \widetilde{W}_{\ext} = t(P) \rtimes W$.
\begin{prop}\label{qls_qb}
The bijection $\Xi:\overleftarrow{\QB}({\id}; m_{\mu}) \rightarrow \QLS^{\mu , \infty}(\lambda)$ satisfies the following:
\begin{enu}
\item
		$\wt (\ed (p^{\OS}_{J})) = {\wt}(\Xi(p^\OS_{J}) )${\rm;}

\item
		${\dg}({\qwt}^{*}(p^{\OS}_{J})) = -{\dg}_{\mu}(\Xi(p^{\OS}_{J}))$.
\end{enu}
\end{prop}
\setcounter{nclaim}{0} %%Claimの番号をリセット

\begin{proof}
We proceed by induction on $ \#J$.

If $J=\emptyset$, it is obvious that ${\dg}({\qwt}^{*}(p^{\OS}_{J})) = {\dg}_{\mu}(\Xi (p^{\OS}_{J}))=0$ and 
${\wt}(\ed (p^{\OS}_{J})) = \wt(\Xi(p^{\OS}_{J}) )=\mu$,
since
$\Xi(p^{\OS}_{J})=(v(\mu)\lons; 0,1)$.
	
Let	$J=\{ j_1 <j_2 <\cdots <j_r\}$, and set
	$K \eqdef J\setminus \{ j_r \}$;
	assume that $\Xi (p^{\OS}_{K})$ is of the form: 
$\Xi (p^{\OS}_{K}) =
  (\lfloor w_{s} \rfloor , \lfloor w_{s-1} \rfloor ,\ldots , \lfloor w_{1} \rfloor ; 1-{\sigma}_{s},\ldots, 1-{\sigma}_{0})  $.
In the following, we employ the notation $w_p$, $0\leq p \leq s$, used in the definition of the map $\Xi$. 
Note that
$\dir(p^\OS_K)=w_s \lon$
 and 
$w_0=v(\mu)\lons$
by the definition of $\Xi$.
Also, observe that
if $d_{j_{r}}=d_{j_{r-1}}={\sigma}_{s-1}$, 
then $\{ d_{j_1} \leq \cdots \leq d_{j_{r-1}}\leq d_{j_r} \}
=
\{ d_{j_1} \leq \cdots \leq d_{j_{r-1}} \} $,
and
if $d_{j_{r}}>d_{j_{r-1}}={\sigma}_{s-1}$,
then
$\{ d_{j_1} \leq \cdots \leq d_{j_{r-1}}\leq d_{j_r} \}
= 
\{ d_{j_1} \leq \cdots \leq d_{j_{r-1}}< d_{j_r} \}
$.
From these, we deduce that
	\begin{align*}
	\Xi (p^{\OS}_{J})=
\begin{cases}
			\rlap{$(\lfloor w_{s} s_{\lon \overline {\beta^{\OS}_{j_r} }  } \rfloor , \lfloor w_{s-1}\rfloor ,\ldots  , \lfloor w_{1} \rfloor ; 1-{\sigma}_{s},1-\sigma_{s-1}, \ldots, 1-{\sigma}_{0})$} &  \\
\phantom{(\lfloor w_{s} s_{\lon \overline {\beta^{\OS}_{j_r} }  } \rfloor , \lfloor w_{s-1}\rfloor ,\ldots  , \lfloor w_{1} \rfloor ; 1-{\sigma}_{s},1-\sigma_{s-1}, \ldots)}
& \mbox{if } d_{j_{r}}=d_{j_{r-1}} ={\sigma}_{s-1}, \\
\rlap{$(\lfloor w_{s}s_{\lon \overline {\beta^{\OS}_{j_r }} } \rfloor , \lfloor w_{s} \rfloor , \lfloor w_{s-1} \rfloor ,\ldots  , \lfloor w_{1} \rfloor ; 1-{\sigma}_{s}, 1-d_{j_{r}} ,1-\sigma_{s-1}  
			\ldots, 1-{\sigma}_{0})$}
& \\
	& \mbox{if } d_{j_{r}}>d_{j_{r-1}} ={\sigma}_{s-1}.
\end{cases}
	\end{align*}

For the induction step, it suffices to show the following claims.

\begin{nclaim}
\mbox{}
\begin{enu}
\item
We have
	\begin{equation*}
		\wt ( \Xi (p^{\OS}_{J}) ) 
		= \wt ( \Xi (p^{\OS}_{K}) )+ 
		a_{j_{r}}w_{s}\lon \left( -\overline {\beta^{\OS}_{j_r } } \right)^\lor.
	\end{equation*}

\item
We have
\begin{equation*}
		\dg_{\mu} ( \Xi (p^{\OS}_{J}) ) 
		= \dg_{\mu} ( \Xi (p^{\OS}_{K}) )- \chi a_{j_r},
	\end{equation*}
where
$\chi \eqdef 0$ (resp., $\chi \eqdef 1$) 
if $w_{s}s_{\lon \overline {\beta^{\OS}_{j_r } }  } \leftarrow w_{s}$ is a Bruhat (resp., quantum) edge.
\end{enu}
\end{nclaim}

\begin{nclaim}
\mbox{}
\begin{enu}
\item
We have
	\begin{equation*}
		\wt ( \ed (p^{\OS}_{J}) ) 
		= \wt ( \ed (p^{\OS}_{K}) )+ 
		a_{j_{r}}w_{s}\lon \left( -\overline {\beta^{\OS}_{j_r } } \right)^\lor.
	\end{equation*}

\item
We have
	\begin{equation*}
		\dg ( \qwt^* (p^{\OS}_{J}) )
		= \dg ( \qwt^* (p^{\OS}_{K}) ) + \chi a_{j_r} .
	\end{equation*}
\end{enu}
\end{nclaim}

%\begin{proof}[Proof of Claim 1]
\noindent $Proof \ of \ Claim \ 1.$
(1)
If $d_{j_{r}}=d_{j_{r-1}}={\sigma}_{s-1}$, then we compute:
\begin{align*}
		\wt ( \Xi (p^{\OS}_{J}) ) 
		=& (\sigma_s -\sigma_{s-1})
			\lfloor w_{s}s_{\lon \overline {\beta^{\OS}_{j_r } }  } \rfloor \lambda
			+\sum_{p=1}^{s-1} ({\sigma}_{p} - {\sigma}_{p-1} ) \lfloor w_{p}\rfloor \lambda \\
		=& (\sigma_s -\sigma_{s-1})
			 w_{s}s_{\lon \overline {\beta^{\OS}_{j_r } }  } \lambda
			+\sum_{p=1}^{s-1} ({\sigma}_{p} - {\sigma}_{p-1} ) w_{p}\lambda \\
=& \sum_{p=1}^{s} ({\sigma}_{p} - {\sigma}_{p-1} ) w_{p} \lambda 
		+ (\sigma_s- \sigma_{s-1}) w_{s}s_{\lon \overline {\beta^{\OS}_{j_r } }  }\lambda - (\sigma_s- \sigma_{s-1}) w_{s}\lambda \\
		\overset{\sigma_s=1,\ d_{j_r}=\sigma_{s-1}}{=}&
		 \sum_{p=1}^{s} ({\sigma}_{p} - {\sigma}_{p-1} ) w_{p}\lambda 
		+ (1-d_{j_{r}}) w_{s}s_{\lon \overline {\beta^{\OS}_{j_r } }  }\lambda - (1-d_{j_{r}}) w_{s}\lambda .
	\end{align*}
If $d_{j_{r}}>d_{j_{r-1}}={\sigma}_{s-1}$, then we compute:
\begin{align*}
		\wt ( \Xi (p^{\OS}_{J}) ) 
		=& (\sigma_s - d_{j_r})
			\lfloor w_{s}s_{\lon \overline {\beta^{\OS}_{j_r } }  } \rfloor \lambda
			+(d_{j_r}-\sigma_{s-1})\lfloor w_s \rfloor \lambda
			+\sum_{p=1}^{s-1} ({\sigma}_{p} - {\sigma}_{p-1} ) \lfloor w_{p}\rfloor \lambda \\
		=& (\sigma_s - d_{j_r})
			w_{s}s_{\lon \overline {\beta^{\OS}_{j_r } }  }  \lambda
			+(d_{j_r}-\sigma_{s-1}) w_s  \lambda
			+\sum_{p=1}^{s-1} ({\sigma}_{p} - {\sigma}_{p-1} ) w_{p} \lambda \\
		=& \sum_{p=1}^{s} ({\sigma}_{p} - {\sigma}_{p-1} ) w_{p}\lambda 
			-({\sigma}_{s} - {\sigma}_{s-1} ) w_{s}\lambda +
			(\sigma_s - d_{j_r})
			w_{s}s_{\lon \overline {\beta^{\OS}_{j_r } }  }\lambda
			+(d_{j_r}-\sigma_{s-1})w_s \lambda		\\	
=& \sum_{p=1}^{s} ({\sigma}_{p} - {\sigma}_{p-1} ) w_{p} \lambda 
		+ (\sigma_s- d_{j_r}) w_{s}s_{\lon \overline {\beta^{\OS}_{j_r } }  }\lambda - (\sigma_s-d_{j_r}) w_{s}\lambda \\
		\overset{\sigma_s=1}=&
		 \sum_{p=1}^{s} ({\sigma}_{p} - {\sigma}_{p-1} ) w_{p}\lambda 
		+ (1-d_{j_{r}}) w_{s}s_{\lon \overline {\beta^{\OS}_{j_r } }  }\lambda - (1-d_{j_{r}}) w_{s}\lambda .
	\end{align*}

In both cases above, since 
\begin{equation*}
\wt(\Xi(p^{\OS}_K)) =\sum_{p=1}^{s} ({\sigma}_{p} - {\sigma}_{p-1} ) \lfloor w_{p} \rfloor \lambda = \sum_{p=1}^{s} ({\sigma}_{p} - {\sigma}_{p-1} ) w_{p}\lambda,
\end{equation*}
and since
\begin{eqnarray*}
 (1-d_{j_{r}}) w_{s}s_{\lon \overline {\beta^{\OS}_{j_r } }  }\lambda - (1-d_{j_{r}}) w_{s}\lambda &=&
-(1-d_{j_{r}}) w_s {\langle \lambda,  \lon \overline {\beta^{\OS}_{j_r } }   \rangle} \lon \left( \overline { \beta^{\OS}_{j_r } }\right)^\lor  \\
&\overset{{\rm Remark} \ \ref{2.15}}{=}&
- \frac{a_{j_r}}{\langle \lambda_-, \overline{ \beta^{\OS}_{j_r} }     \rangle}
{\langle \lambda_-, \overline{ \beta^{\OS}_{j_r} }    \rangle}
 w_s \lon \left( \overline { \beta^{\OS}_{j_r } }\right)^\lor \\
&=&
 a_{j_{r}}w_{s}\lon \left( -\overline { \beta^{\OS}_{j_r } }\right)^\lor,
\end{eqnarray*}
it follows that
	\begin{eqnarray*}
		\wt ( \Xi (p^\OS_{J}) ) &=& \sum_{p=1}^{s} ({\sigma}_{p} - {\sigma}_{p-1} ) w_{p}\lambda 
		+ (1-d_{j_{r}}) w_{s}s_{\lon \overline {\beta^{\OS}_{j_r } }  }\lambda - (1-d_{j_{r}}) w_{s}\lambda \\
		&=& \wt ( \Xi (p^\OS_{K}) )+ a_{j_{r}}w_{s}\lon \left( -\overline {\beta^{\OS}_{j_r } } \right)^\lor.
	\end{eqnarray*}

(2)
From the relation between $p^\OS_J$ and $p^\OS_K$, and from the definition of $\overleftarrow{\QB}({\id}; m_{\mu})$,
we find that $w_{s}\lon s_{-\overline {\beta^{\OS}_{j_r } }  } \xrightarrow{-\left( \overline {\beta^{\OS}_{j_r } }\right)^\lor } w_{s}\lon$
is an edge of $\QBG$.
Hence, by Lemma \ref{involution}, 
$w_{s}s_{\lon \overline {\beta^{\OS}_{j_r } }  } \xleftarrow{\lon \left( \overline {\beta^{\OS}_{j_r } }\right)^\lor } w_{s}$
is an edge of $\QBG$.

If $d_{j_{r}}=d_{j_{r-1}}={\sigma}_{s-1}$,
then by the definition of $\dg_{\mu}$ (along with \cite[Lemma 7.2]{LNSSS2}), we see that
\begin{align}\label{eq2}
\dg_{\mu} ( \Xi (p^{\OS}_{J}) ) 
		=& -\sum_{p=0}^{s-2}(1-{\sigma}_p)\wt_{\lambda}( \lfloor w_{p+1} \rfloor \Leftarrow \lfloor w_{p} \rfloor )
				-(1-d_{j_{r}})\wt_{\lambda}( \lfloor w_{s}s_{\lon \overline {\beta^{\OS}_{j_r } }  } \rfloor \Leftarrow \lfloor w_{s-1} \rfloor) \nonumber \\
%\overset{\rm Lemma \ \ref{8.1} \ (2)}{=}
=& -\sum_{p=0}^{s-2}(1-{\sigma}_p)\wt_{\lambda}(w_{p+1} \Leftarrow w_{p} )
				-(1-d_{j_{r}})\wt_{\lambda}(w_{s}s_{\lon \overline {\beta^{\OS}_{j_r } }  } \Leftarrow w_{s-1}).
\end{align}
Here, $w_0 = v(\mu)\lons$ as mentioned in the proof of Lemma \ref{final_direction}, so that
$\lfloor w_0 \rfloor = v(\mu)$.
Since $d_{j_{r}}=d_{j_{r-1}}={\sigma}_{s-1}$,
we have 
$\lon \left( \overline{\beta^{\OS}_{j_{r-1}}} \right)^\lor 
\prec 
\lon \left( \overline{\beta^{\OS}_{j_{r}}} \right)^\lor$
by
Lemma \ref{remark2.11}.
Because
the (unique) label-increasing  directed  path in $\QBG$ from $w_{s-1}$ to $w_s $
has the final label $\lon \left( \overline{\beta^{\OS}_{j_{r-1}}} \right)^\lor$,
by concatenating this directed path from $w_{s-1}$ to $w_s $ with 
$w_{s} \xrightarrow{\lon \left( \overline {\beta^{\OS}_{j_r } }\right)^\lor } w_{s}s_{\lon \overline {\beta^{\OS}_{j_r } }  }$,
we obtain a label-increasing (hence shortest)  directed path from $w_{s-1}$ to $w_{s}s_{\lon \overline {\beta^{\OS}_{j_r } }  }$
passing through $w_s $.
Therefore, we deduce that
\begin{equation}\label{wtbunkai}
\wt_{\lambda}(w_{s}s_{\lon \overline {\beta^{\OS}_{j_r } }  } \Leftarrow w_{s-1})
=
\wt_{\lambda}(w_{s}s_{\lon \overline {\beta^{\OS}_{j_r } }  } \leftarrow w_{s})
+
\wt_{\lambda}(w_{s} \Leftarrow w_{s-1}).
\end{equation}
It follows from (\ref{eq2}) and (\ref{wtbunkai}) that
\begin{equation*}
\dg_{\mu} ( \Xi (p^{\OS}_{J}) ) 
		= -\sum_{p=0}^{s-1}(1-{\sigma}_p)\wt_{\lambda}(w_{p+1} \Leftarrow w_{p} )
				-(1-d_{j_{r}})\wt_{\lambda}(w_{s}s_{\lon \overline {\beta^{\OS}_{j_r } }  } \leftarrow w_{s}).
\end{equation*}

If $d_{j_{r}}>d_{j_{r-1}}={\sigma}_{s-1}$,
then by the definition of $\dg_{\mu}$ (along with \cite[Lemma 7.2]{LNSSS2}), we see that
\begin{equation*}
\dg_{\mu} ( \Xi (p^{\OS}_{J}) ) 
		= -\sum_{p=0}^{s-1}(1-{\sigma}_p)\wt_{\lambda}(w_{p+1} \Leftarrow w_{p} )
				-(1-d_{j_{r}})\wt_{\lambda}(w_{s}s_{\lon \overline {\beta^{\OS}_{j_r } }  } \leftarrow w_{s}),
\end{equation*}
where $w_0 = v(\mu) \lons$.
Also, by the definition of $\dg_{\mu}$ (along with \cite[Lemma 7.2]{LNSSS2}), we have 
\begin{equation*}
\dg_{\mu} ( \Xi (p^{\OS}_{K}) ) 
		= -\sum_{p=0}^{s-1}(1-{\sigma}_p)\wt_{\lambda}(w_{p+1} \Leftarrow w_{p} ),
\end{equation*}
where $w_0 = v(\mu) \lons$.

In both cases above, we deduce that 
\begin{equation*}
\dg_{\mu} ( \Xi (p^{\OS}_{J}) ) 
		= \dg_{\mu} ( \Xi (p^{\OS}_{K}) ) 
				-(1-d_{j_{r}})\wt_{\lambda}(w_{s}s_{\lon \overline {\beta^{\OS}_{j_r } }  } \leftarrow w_{s}).
\end{equation*}
%Here, since
%$w_{s}s_{\lon \overline {\beta^{\OS}_{j_r } }  } \xleftarrow{\lon \left( \overline {\beta^{\OS}_{j_r } }\right)^\lor } w_{s}$
%is an edge of $\QBG$,
%we have
%$\wt_{\lambda}(w_{s}s_{\lon \overline {\beta^{\OS}_{j_r } }  } \Leftarrow w_{s}) =
%\wt_{\lambda}(w_{s}s_{\lon \overline {\beta^{\OS}_{j_r } }  } \leftarrow w_{s})$.
If the edge $w_{s}s_{\lon \overline {\beta^{\OS}_{j_r } }  } \leftarrow w_{s}$ is a Bruhat edge,
then we have 
$\wt_{\lambda}(w_{s}s_{\lon \overline {\beta^{\OS}_{j_r } }  } \leftarrow w_{s})=0$.
If the edge $w_{s}s_{\lon \overline {\beta^{\OS}_{j_r } }  } \leftarrow w_{s}$ is a quantum edge,
then we have  $\wt_{\lambda}(w_{s}s_{\lon \overline {\beta^{\OS}_{j_r } }  } \leftarrow w_{s})=
\langle \lambda , \lon \overline {\beta^{\OS}_{j_r } } \rangle$; 
note that
\begin{equation*}
(1-d_{j_{r}})\langle \lambda ,\lon \overline {\beta^{\OS}_{j_r } }  \rangle
\overset{{\rm Remark} \ \ref{2.15}}{=}
 \frac{a_{j_r}}{\langle \lambda_-, \overline{ \beta^{\OS}_{j_r} }   \rangle}
{\langle \lambda_- , \overline{ \beta^{\OS}_{j_r} }    \rangle}
=a_{j_r}.
\end{equation*}
Thus, in both cases,
we have 
$		\dg_{\mu} ( \Xi (p^{\OS}_{J}) ) 
		= \dg_{\mu} ( \Xi (p^{\OS}_{K}) )- \chi a_{j_r}$,
and Claim 1 (2) is proved.
%\end{proof}
\bqed

%\begin{proof}[Proof of Claim 2]
\noindent $Proof \ of \ Claim \ 2.$
Let us prove (1).
Note that $\ed (p^{\OS}_{J}) ) =  \ed (p^{\OS}_{K})   s_{\beta^{\OS}_{j_r }}$,
and that
\begin{equation*}
 \ed (p^{\OS}_{K}) = t(\wt ( \ed (p^{\OS}_{K}) ) ) \dir  ( \ed (p^{\OS}_{K}) )  = t(\wt ( \ed (p^{\OS}_{K}) ) )w_s \lon ;
\end{equation*}
the second equality follows from the comment at the beginning of the proof of Proposition \ref{qls_qb}.
Also, we have
$s_{\beta^{\OS}_{j_r }} 
= s_{a_{j_r}\tilde{\delta}+\overline {\beta^{\OS}_{j_r } }}
=t\left( a_{j_r}\left( -\overline {\beta^{\OS}_{j_r }}\right)^\lor \right) s_{\overline {\beta^{\OS}_{j_r }}}$.
Combining these, we obtain
\begin{eqnarray*}
 \ed (p^{\OS}_{J}) 
&=&
\left(
 t(\wt ( \ed (p^{\OS}_{K}) ) )w_s \lon
\right)
\left(
 t\left(a_{j_r}\left( -\overline {\beta^{\OS}_{j_r }}\right)^\lor \right)s_{\overline {\beta^{\OS}_{j_r }}}
\right) \\
&=&t\left( \wt ( \ed (p^{\OS}_{K}) ) +a_{j_r}w_s \lon \left( -\overline {\beta^{\OS}_{j_r }}\right)^\lor \right)
w_s \lon s_{\overline {\beta^{\OS}_{j_r }}},
\end{eqnarray*}
and hence
\begin{equation*}
\wt( \ed (p^{\OS}_{J}) )
=
\wt ( \ed (p^{\OS}_{K}) ) +a_{j_r}w_s \lon \left( -\overline {\beta^{\OS}_{j_r }}\right)^\lor .
\end{equation*}

Let us prove (2).
Since $\dir (\ed( p^{\OS}_{K}) ) =w_{s} \lon $,
we have $\dir (\ed( p^{\OS}_{J}) ) =w_{s}\lon s_{\overline {\beta^{\OS}_{j_r } }  } $.
If $w_{s}s_{\lon \overline{\beta^{\OS}_{j_r }} }\xleftarrow{\lon \left(\overline{\beta^{\OS}_{j_r }}\right)^\lor }  w_{s}$
is a Bruhat edge,
then it follows from Lemma \ref{involution} that
$w_{s}\lon s_{-\overline {\beta^{\OS}_{j_r } }}  \xrightarrow{\left(-\overline {\beta^{\OS}_{j_r } } \right)^\lor}
w_{s}\lon$ is also a Bruhat edge.
Hence we obtain $J^+ = K^+$. This implies that $\dg ( \qwt^* (p^{\OS}_{J}) )
		= \dg ( \qwt^* (p^{\OS}_{K}) )$.
If $w_{s}s_{\lon \overline{\beta^{\OS}_{j_r }} }\xleftarrow{\lon \left(\overline{\beta^{\OS}_{j_r }}\right)^\lor }  w_{s}$
is a quantum edge,
then it follows from Lemma \ref{involution} that
$w_{s}\lon s_{-\overline {\beta^{\OS}_{j_r } } } \xrightarrow{\left(-\overline {\beta^{\OS}_{j_r } } \right)^\lor}
w_{s}\lon$ is also a quantum edge.
Hence we obtain
$J^+ = K^+ \sqcup \{ j_r \}$.
This implies that
$\dg ( \qwt^* (p^{\OS}_{J}) )
= \dg ( \qwt^* (p^{\OS}_{K}) ) + \dg (\beta^{\OS}_{j_r })
= \dg ( \qwt^* (p^{\OS}_{K}) ) + a_{j_r}
$.
Therefore, in both cases, we have $\dg ( \qwt^* (p^{\OS}_{J}) )
= \dg ( \qwt^* (p^{\OS}_{K}) ) +\chi  a_{j_r}
$, and Claim 2 (2) is proved.
%\end{proof}
\bqed

This completes the proof of Proposition \ref{qls_qb}.
\end{proof}

%\begin{proof}[Proof of Theorem \ref{main}]
%We know from Proposition \ref{os} that
%	\begin{equation*}
%		E_{\mu}(q,\infty)=
%		\sum_{p^{\OS}_{J} \in {\overleftarrow{\QB}}({\id} ; m_{\mu}) } 
%		q^{-{\dg}({\qwt}^{*}(p^{\OS}_{J}))}e^{\wt({\ed}(p^{\OS}_{J}))}.
%	\end{equation*}
%Therefore, it follows from Propositions \ref{bijective} and \ref{qls_qb} that
%\begin{equation*}
%		E_{\mu}(q,\infty)=
%		\sum_{\psi \in {\QLS^{w , \infty}(\lambda)} }
%		q^{{\dg}_\mu(\psi)}e^{\wt(\psi)}.
%	\end{equation*}
%Hence we conclude that $E_{\mu}(q,\infty) = \gch_\mu \QLS^{w, \infty}(\lambda)$, as desired.
%\end{proof}

%\subsection{Orr-Shimozono formula in terms of $\QLS$ paths}
%Let $\lambda \in P$ be a dominant weight, and $\mu\in W\lambda$.
%We set
%	\begin{equation*}
%		{\gch}_{\mu}{\QLS}^{\mu,\infty}(\lambda) \eqdef \sum_{\psi \in {\QLS}^{\mu,\infty}(\lambda)} q^{\dg_{\mu}(\psi)} e^{\wt(\psi)}.
%	\end{equation*}
%%We are now ready to state the main result of this paper.
%\begin{thm}\label{main}
%Let $\lambda \in P$ be a dominant weight, and $\mu\in W\lambda$.
%Then,
%	\begin{equation*}
%		E_{\mu}(q, \infty)= {\gch}_{\mu}{\QLS}^{\mu,\infty}(\lambda) .
%	\end{equation*}
%\end{thm}

\begin{proof}[Proof of Theorem~$\ref{main}$]
We know from Proposition \ref{os} that
	\begin{equation*}
		E_{\mu}(q,\infty)=
		\sum_{p^{\OS}_{J} \in {\overleftarrow{\QB}}({\id} ; m_{\mu}) } 
		e^{\wt({\ed}(p^{\OS}_{J}))}q^{-{\dg}({\qwt}^{*}(p^{\OS}_{J}))}.
	\end{equation*}
Therefore, it follows from Propositions \ref{bijective} and \ref{qls_qb} that
\begin{equation*}
		E_{\mu}(q,\infty)=
		\sum_{\psi \in {\QLS^{\mu , \infty}(\lambda)} }
		e^{\wt(\psi)}q^{{\dg}_\mu(\psi)}.
	\end{equation*}
Hence we conclude that $E_{\mu}(q,\infty) = \gch_\mu \QLS^{\mu, \infty}(\lambda)$, as desired.
\end{proof}

%
%=========================%
%     START SECTION 01    %
%=========================%
%
\section{Demazure submodules of level-zero extremal weight modules}
\label{sec:review}
%
%==============================%
%     START SUBSECTION 0101    %
%==============================%
%
\subsection{Untwisted affine root data}
\label{subsec:affalg}
Let $\Fg_{\aff}$ be the untwisted affine Lie algebra over $\BC$ associated to 
the finite-dimensional simple Lie algebra $\Fg$, and
$\Fh_{\aff}=
\bigl(\bigoplus_{j \in I_{\aff}} \BC \alpha_{j}^{\vee}\bigr) \oplus \BC D$
its Cartan subalgebra, where 
$\bigl\{\alpha_{j}^{\vee}\bigr\}_{j \in I_{\aff}} \subset \Fh_{\aff}$ 
is the set of simple coroots, with $I_{\aff}=I \sqcup \{0\}$, and 
$D \in \Fh_{\aff}$ is the degree operator. 
We denote by 
$\bigl\{\alpha_{j}\bigr\}_{j \in I_{\aff}} \subset \Fh_{\aff}^{\ast}$ 
the set of simple roots, and by 
$\Lambda_{j} \in \Fh_{\aff}^{\ast}$, $j \in I_{\aff}$, 
the fundamental weights; 
note that $\pair{\alpha_{j}}{D}=\delta_{j,0}$ and 
$\pair{\Lambda_{j}}{D}=0$ for $j \in I_{\aff}$, 
where $\pair{\cdot\,}{\cdot}:
\Fh_{\aff}^{\ast} \times \Fh_{\aff} \rightarrow \BC$ denotes 
the canonical pairing of $\Fh_{\aff}$ and 
$\Fh_{\aff}^{\ast} \eqdef \Hom_{\BC}(\Fh_{\aff},\,\BC)$. 
Also, let $\delta=\sum_{j \in I_{\aff}} a_{j}\alpha_{j} \in \Fh_{\aff}^{\ast}$ and 
$c=\sum_{j \in I_{\aff}} a^{\vee}_{j} \alpha_{j}^{\vee} \in \Fh_{\aff}$ denote 
the null root and the canonical central element of 
$\Fg_{\aff}$, respectively. % Here we should note that $a_{0}=1$. 
Here we note that $\Fh_{\aff}=\Fh \oplus \BC c \oplus \BC D$; if we regard an element 
$\lambda \in \Fh^{\ast}$ as an element of $\Fh_{\aff}^{\ast}$ by: 
$\pair{\lambda}{c}=\pair{\lambda}{D}=0$, then we have 
$\varpi_i = \Lambda_i - a_{i}^{\vee} \Lambda_0$ for $i \in I$. 
We take a weight lattice $P_{\aff}$ for $\Fg_{\aff}$ as follows: 
$P_{\aff} = 
\bigl(\bigoplus_{j \in I_{\aff}} \BZ \Lambda_{j}\bigr) \oplus 
   \BZ \delta \subset \Fh_{\aff}^{\ast}$, and
 set $Q_{\aff}\eqdef \bigoplus_{j \in I_{\aff}} \BZ \alpha_{j}$.

%%%%%%%%%%
%%%%%%%%%%
%begin3
%%%%%%%%%%
%%%%%%%%%%
\begin{rem}
We should warn the reader that the root datum of the affine Lie algebra $\Fg_\aff$ is not necessarily dual to that of the untwisted affine Lie algebra
associated to $\widetilde{\Fg}$ in \S 2.2, though the root datum of $\widetilde{\Fg}$ is dual to that of $\Fg$.
In particular, for the index $0 \in I_\aff$, the simple coroot $\alpha_0^\lor = c - \theta^\lor$, with $\theta \in \Delta^+$ the highest root of $\Fg$,
does not agree with the simple root $\widetilde{\delta}-\varphi^\lor$ in \S 2.2, which is denoted by $\alpha_0^\lor$ there.
\end{rem}
%%%%%%%%%%
%%%%%%%%%%
%end3
%%%%%%%%%%
%%%%%%%%%%

The Weyl group $W_{\aff}$ of $\Fg_{\aff}$ is defined to be the subgroup
$\langle \sss_j \mid j \in I_{\aff} \rangle \subset \GL(\Fh_{\aff}^{\ast})$ 
generated by the simple reflections $\sss_{j}$
associated to $\alpha_{j}$ for $j \in I_{\aff}$, 
with length function $\ell:W_{\aff} \rightarrow \BZ_{\ge 0}$ and 
identity element $e \in W_{\aff}$. 
For $\xi \in Q^{\vee}=\bigoplus_{i \in I} \BZ \alpha_{i}^{\vee}$, 
let $t(\xi) \in W_{\aff}$ denote 
the translation in $\Fh^{\ast}_{\aff}$ by $\xi$ (see \cite[\S6.5]{Kac}). 
Then we know from \cite[Proposition 6.5]{Kac} that 
$\bigl\{ t(\xi) \mid \xi \in Q^{\vee} \bigr\}$ forms 
an abelian normal subgroup of $W_{\aff}$, 
and that $t(\xi) t(\zeta) = t(\xi + \zeta)$, 
$\xi,\,\zeta \in Q^{\vee}$, and 
$W_{\aff} = W \ltimes \bigl\{ t(\xi) \mid \xi \in Q^{\vee} \bigr\}$. 
%
% \clr{For $x \in W_{\aff}$, we define $\dir(x) \in W$ and $\wt(x) \in Q^{\vee}$ by: 
% $x=\dir(x) t(\wt(x))$.} 
%
%ここで1ページ目終了
%
We denote by $\rr$ the set of real roots, i.e., 
$\rr\eqdef \bigl\{x\alpha_{j} \mid x \in W_{\aff},\,j \in I_{\aff}\bigr\}$, 
and by $\prr \subset \rr$ the set of positive real roots; 
we know from \cite[Proposition 6.3]{Kac} that
\begin{align*}
\begin{split}
\rr & = 
\bigl\{ \alpha + n \delta \mid \alpha \in \Delta,\, n \in \BZ \bigr\}, \\
\prr & = 
\Delta^{+} \sqcup 
\bigl\{ \alpha + n \delta \mid \alpha \in \Delta,\, n \in \BZ_{> 0}\bigr\}. 
\end{split}
\end{align*}
For $\beta \in \rr$, we denote by $\beta^{\vee} \in \Fh_{\aff}$ 
the dual root of $\beta$, and by $\sss_{\beta} \in W_{\aff}$ 
the reflection with respect to $\beta$; 
note that if $\beta \in \Delta_{\aff}$ is 
of the form $\beta = \alpha + n \delta$ 
with $\alpha \in \Delta $ and $n \in \BZ$, then 
$\sss_{\beta} =\sss_{\alpha} t(n\alpha^{\vee})$.

%Now, for $\xi,\,\zeta \in Q^{\vee}$, we write 
%%
%%%%%%%%%%%%%%%
%%%% eq:dom %%%
%%%%%%%%%%%%%%%
%%
%\begin{equation} \label{eq:dom}
%\xi \ge \zeta \quad \text{if} \quad
%\xi-\zeta \in Q^{\vee,+}=\sum_{i \in I} \BZ_{\ge 0} \alpha_{i}^{\vee}.
%\end{equation}
%%
%Further, if $S$ is a subset of $I$, then we denote by
%%
%%%%%%%%%%%%%%%
%%%% eq:prj %%%
%%%%%%%%%%%%%%%
%%
%\begin{equation} \label{eq:prj}
%[\,\cdot\,]=[\,\cdot\,]_{I \setminus S} : 
%Q^{\vee} \twoheadrightarrow Q_{I \setminus S}^{\vee}
%=\bigoplus_{i \in I \setminus S}\BZ\alpha_{i}
%\end{equation}
%%
%the projection from $Q^{\vee}=Q_{I \setminus S}^{\vee} \oplus Q_{S}^{\vee}$
%onto $Q_{I \setminus S}^{\vee}$ with 
%kernel $Q_{S}^{\vee}=\bigoplus_{i \in S}\BZ\alpha_{i}$. 
%
%==============================%
%     START SUBSECTION 0202    %
%==============================%
%
\subsection{Peterson's coset representatives}
\label{subsec:W^J_af}

Let $S$ be a subset of $I$. Following \cite{Pet97} 
(see also \cite[\S10]{LS10}), we set
\begin{align}
Q_S^\lor &\eqdef  \sum_{i \in S} \mathbb{Z}\alpha_i^\lor, \\
(\Delta_S)_{\aff} 
  & \eqdef  \bigl\{ \alpha + n \delta \mid 
  \alpha \in \Delta_S , n \in \BZ \bigr\} \subset \Delta_{\aff}, \\
(\Delta_S)_{\aff}^{+}
  &\eqdef  (\Delta_S)_{\aff} \cap \prr = 
  \Delta_S^+ \sqcup \bigl\{ \alpha + n \delta \mid 
  \alpha \in \Delta_S,\,n \in \BZ_{> 0} \bigr\}, \\
\label{eq:stabilizer}
(W_S)_{\aff} 
 & \eqdef  W_S \ltimes \bigl\{ t(\xi) \mid \xi \in Q_S^{\vee} \bigr\}
   = \langle \sss_{\beta} \mid \beta \in (\Delta_S)_{\aff}^{+} \rangle, \\
\label{eq:Pet}
(W^S)_{\aff}
 &\eqdef  \bigl\{ x \in W_{\aff} \mid 
 \text{$x\beta \in \prr$ for all $\beta \in (\Delta_S)_{\aff}^+$} \bigr\}. 
\end{align}
Then we know the following from \cite{Pet97} 
(see also \cite[Lemma~10.6]{LS10}).
%
%%%%%%%%%%%%%%
%%% prop:P %%%
%%%%%%%%%%%%%%
%
\begin{prop} \label{prop:P}
For each $x \in W_{\aff}$, there exist a unique 
$x_1 \in (W^S)_{\aff}$ and a unique $x_2 \in (W_S)_{\aff}$ 
such that $x = x_1 x_2$.
\end{prop}

We define a (surjective) map $\PJ : W_{\aff} \rightarrow (W^S)_{\aff}$ 
by $\PJ (x) \eqdef  x_1$ if $x= x_1 x_2$ with $x_1 \in (W^S)_{\aff}$ and 
$x_2 \in (W_S)_{\aff}$.
%
%%%%%%%%%%%%%%
%%% lem:PJ %%%
%%%%%%%%%%%%%%
%
\begin{lem}[{\cite{Pet97}; see also \cite[Proposition 10.10]{LS10}}] 
\label{lem:PJ} \mbox{}
\begin{enu}
\item $\PJ (w) = \lfloor w \rfloor $ for every $w \in W$.

\item $\PJ (x t(\xi)) = \PJ (x) \PJ (t(\xi))$ 
for every $x \in W_{\aff}$ and $\xi \in Q^{\vee}$.
\end{enu}
\end{lem}

An element $\xi \in Q^{\vee}$ is said to be $S$-adjusted 
if $\pair{\gamma}{\xi} \in \bigl\{ -1,\,0 \bigr\}$ 
for all $\gamma \in \Delta_{S}^{+}$ (see \cite[Lemma~3.8]{LNSSS1}). 
Let $\Qad$ denote the set of $S$-adjusted elements.
%
%%%%%%%%%%%%%%%%%
%%% lem:S-adj %%%
%%%%%%%%%%%%%%%%%
%
\begin{lem}[{\cite[Lemma~2.3.5]{INS}}] \label{lem:S-adj}
\mbox{}
\begin{enu}
\item For each $\xi \in Q^{\vee}$, there exists 
a unique $\phi_S (\xi) \in Q_{S}^{\vee}$ such that 
$\xi + \phi_S (\xi) \in \Qad$. In particular, 
$\xi \in \Qad$ if and only if $\phi_{S}(\xi)=0$. 

\item For each $\xi \in Q^{\vee}$, 
the element $\PJ (t(\xi)) \in (W^{S})_{\aff}$ is of the form 
$\PJ (t(\xi)) = z_{\xi} t(\xi + \phi_S (\xi))$ 
for a specific element $z_{\xi} \in W_S$. 
Also, $\PJ (w t(\xi)) = 
\mcr{w} z_{\xi} t(\xi + \phi_S (\xi))$ 
for every $w \in W$ and $\xi \in Q^{\vee}$.

\item We have 
%
%%%%%%%%%%%%%%%
%%% eq:WSaf %%%
%%%%%%%%%%%%%%%
%
\begin{equation}\label{eq:WSaf}
(W^S)_{\aff} = 
\bigl\{ w z_{\xi} t(\xi) \mid w \in W^S,\,\xi \in \Qad \bigr\}.
\end{equation}
\end{enu}
\end{lem}
%
%%%%%%%%%%%%%%%%%
%%% rem:trans %%%
%%%%%%%%%%%%%%%%%
%
\begin{rem} \label{rem:trans}
(1) Let $\xi,\,\zeta \in Q^{\vee}$. If $\xi \equiv \zeta$ mod $Q_{S}^{\vee}$, i.e., $\xi-\zeta \in Q_{S}^{\vee}$,
then $\PJ(t(\xi))=\PJ(t(\zeta))$ since $t(\xi-\zeta) \in (W_{S})_{\aff}$.
Hence we see by Lemma~\ref{lem:S-adj}\,(2) that 
$\xi + \phi_{S}(\xi) = \zeta + \phi_{S}(\zeta)$ and 
$z_{\xi} = z_{\zeta}$. In particular, 
$z_{\xi+\phi_{S}(\xi)} = z_{\xi}$ for every $\xi \in Q^{\vee}$. 

(2) Let $x=wz_{\xi}t(\xi) \in (W^{S})_{\aff}$, 
with $w \in W^{S}$ and $\xi \in \Qad$; note that $\PJ(x)=x$. 
Then it follows from Lemma~\ref{lem:PJ}\,(2) that 
for every $\zeta \in Q^\lor$, 
\begin{equation} \label{eq:t1}
x \PJ(t(\zeta)) = \PJ(x)\PJ(t(\zeta)) 
 = \PJ(xt(\zeta)) \in (W^{S})_{\aff}. 
\end{equation}
%
% and
%
% \begin{equation} \label{eq:t2}
% \wt (x \PJ(t(\zeta))) = 
% z_{\zeta}^{-1}\xi + \zeta + \phi_{S}(\zeta) \equiv 
% \xi+\zeta = \wt(x) + \zeta \mod Q_{S}^{\vee}. 
% \end{equation}
%
\end{rem}

%==============================%
%     START SUBSECTION 0203    %
%==============================%
%
\subsection{Parabolic semi-infinite Bruhat graph} 
\label{subsec:SBG}
In this subsection,
we prove some technical lemmas, which we use later.

%
%%%%%%%%%%%%%%%%
%%% dfn:sell %%%
%%%%%%%%%%%%%%%%
%
\begin{dfn}[{\cite{Pet97}}] \label{dfn:sell}
Let $x \in W_{\aff}$, and 
write it as $x = w t(\xi)$ for $w \in W$ and $\xi \in Q^{\vee}$. 
Then we define the semi-infinite length $\sell(x)$ of $x$ by
$\sell (x) \eqdef  \ell (w) + 2 \pair{\rho}{\xi}$,
where $\rho = (1/2)\sum_{\alpha \in \Delta^+} \alpha$. 
\end{dfn}
Let us fix a subset $S$ of $I$. 
%
%%%%%%%%%%%%%%%
%%% dfn:SiB %%%
%%%%%%%%%%%%%%%
%
\begin{dfn}\label{def:SiB}
(1) We define the (parabolic) semi-infinite Bruhat graph $\SBG^{S}$ 
to be the $\prr$-labeled, directed graph with vertex set $(W^S)_{\aff}$ 
and $\prr$-labeled, directed edges of the following form:
$x \edge{\beta} \sss_{\beta} x$ for $x \in (W^S)_{\aff}$ and $\beta \in \prr$, 
where $\sss_{\beta } x \in (W^S)_{\aff}$ and 
$\sell (\sss_{\beta} x) = \sell (x) + 1$.

(2) The semi-infinite Bruhat order is a partial order 
$\sile$ on $(W^S)_{\aff}$ defined as follows: 
for $x,\,y \in (W^S)_{\aff}$, we write $x \sile y$ 
if there exists a directed path from $x$ to $y$ in $\SBG^{S}$; 
also, we write $x \sil y$ if $x \sile y$ and $x \ne y$. 
\end{dfn}

Let
%recall from \eqref{eq:prj} that 
$[\,\cdot\,]=[\,\cdot\,]_{I \setminus S} : 
Q^{\vee} \twoheadrightarrow Q_{I \setminus S}^{\vee}$ 
denote the projection from $Q^{\vee}$ onto $Q_{I \setminus S}^{\vee}$ 
with kernel $Q_{S}^{\vee}$.
Also, for $\xi, \zeta \in Q^\lor$, we write
\begin{equation}\label{eq:4.7}
\xi \geq \zeta \
\mbox{ if } \
\xi-\zeta \in Q^{\lor, +} \eqdef
\sum_{i \in I}
\mathbb{Z}_{\geq 0}
\alpha_i^\lor.
\end{equation}
The next lemma follows from \cite[Remark~2.3.3]{NS-D}.
%
%%%%%%%%%%%%%%%%
%%% lem:sig0 %%%
%%%%%%%%%%%%%%%%
%
\begin{lem} \label{lem:sig0}
Let $u,\,v \in W^{S}$, $\xi,\,\zeta \in \Qad$, 
and $\beta \in \Delta_{\aff}^{+}$. 
If $uz_{\zeta}t(\zeta) \edge{\beta} vz_{\xi}t(\xi)$ in $\SBG^{S}$, 
then $[\xi] \ge [\zeta]$. 
\end{lem}
%
%%%%%%%%%%%%%%%%
%%% lem:sig1 %%%
%%%%%%%%%%%%%%%%
%
\begin{lem} \label{lem:sig1}
Let $x \in W^{S}$, and $\xi,\,\zeta \in \Qad$. Then, 
$xz_{\xi}t(\xi) \sige xz_{\zeta}t(\zeta)$ 
if and only if $[\xi] \ge [\zeta]$. 
\end{lem}

\begin{proof}
The ``only if'' part is obvious by Lemma~\ref{lem:sig0}. 
We show the ``if'' part by induction on $\ell(x)$. 
If $\ell(x) = 0$, i.e., $x=e$, then 
the assertion $z_{\xi}t(\xi) \sige z_{\zeta}t(\zeta)$ follows from 
\cite[Lemma~6.2.1]{INS} (with $a=1$ and $J$ replaced by $S$). 
Assume now that $\ell(x) > 0$, and take $i \in I$ such that 
$\ell(\sss_{i}x) = \ell(x) - 1$; note that $\sss_{i}x \in W^{S}$ and 
$-x^{-1}\alpha_{i} \in \Delta^{+} \setminus \Delta_{S}^{+}$. 
By induction hypothesis, we have 
$\sss_{i}xz_{\xi}t(\xi) \sige \sss_{i}xz_{\zeta}t(\zeta)$. 
If we take a dominant weight $\lambda \in P^+$ 
such that 
$S_{\lambda}= \bigl\{ i \in I \mid \pair{\lambda}{\alpha_{i}^{\vee}}=0\bigr\} = S$ , then we see that 
\begin{equation*}
\pair{\sss_{i}xz_{\xi}t(\xi)\lambda}{\alpha_{i}^{\vee}} = 
\pair{\sss_{i}xz_{\zeta}t(\zeta)\lambda}{\alpha_{i}^{\vee}} = 
\pair{\sss_{i}x\lambda}{\alpha_{i}^{\vee}} > 0.
\end{equation*}
Therefore, we deduce from \cite[Lemma~2.3.6\,(3)]{NS-D} that 
$xz_{\xi}t(\xi) \sige xz_{\zeta}t(\zeta)$, as desired. 
\end{proof}
%
%%%%%%%%%%%%%%%%%
%%% lem:trans %%%
%%%%%%%%%%%%%%%%%
%
\begin{lem} \label{lem:trans}
Let $x,\,y \in (W^S)_{\aff}$ and $\beta \in \prr$ be such that 
$x \edge{\beta} y$ in $\SBG^{S}$. Then, 
$\PJ(xt(\xi)) \edge{\beta} \PJ(yt(\xi))$ in $\SBG^{S}$ 
for every $\xi \in Q^{\vee}$. 
Therefore, if $x,\,y \in (W^S)_{\aff}$ satisfy $x \sile y$, 
then $\PJ(xt(\xi)) \sile \PJ(yt(\xi))$. 
\end{lem}

\begin{proof}
We see \eqref{eq:t1} that $\PJ(xt(\xi)) = x\PJ(t(\xi))$ and 
$\PJ(yt(\xi)) = y\PJ(t(\xi))$. Since $y=\sss_{\beta}x$ by the assumption, 
we obtain $\PJ(yt(\xi)) = \sss_{\beta} \PJ(xt(\xi))$. 
Hence it suffices to show that 
\begin{equation} \label{eq:trans1}
\sell(\PJ(yt(\xi))) = 
\sell(\PJ(xt(\xi)))+1. 
\end{equation}
We write $x \in (W^{S})_{\aff}$ as
$x = w z_{\zeta}t(\zeta)$, with $w \in W^{S}$ and 
$\zeta \in \Qad$ (see \eqref{eq:WSaf}). 
Then we see from \cite[Lemma~A.2.1 and (A.2.1)]{INS} that
\begin{align*}
\sell(\PJ(xt(\xi))) 
& = \ell(w) + 2 \pair{\rho-\rho_{S}}{\zeta+\xi} \\
& = \ell(w) + 2 \pair{\rho-\rho_{S}}{\zeta} + 2 \pair{\rho-\rho_{S}}{\xi} \\
& = \sell(\PJ(x))+ 2 \pair{\rho-\rho_{S}}{\xi} \\ 
& = \sell(x)+2 \pair{\rho-\rho_{S}}{\xi}. 
\end{align*}
Similarly, we see that 
$\sell(\PJ(yt(\xi))) = \sell(y)+2 \pair{\rho-\rho_{S}}{\xi}$. 
Since $\sell(y)=\sell(x)+1$ by the assumption, 
we obtain \eqref{eq:trans1}, as desired. 
\end{proof}

Let $x,\,y \in W^{S}$, and take a shortest directed path 
\begin{equation*}
\bp : x=x_{0} \edge{\gamma_{1}} x_{1} \edge{\gamma_{2}} x_{2} \edge{\gamma_{3}} \cdots 
\edge{\gamma_{p}} x_{p} = y
\end{equation*}
from $x$ to $y$ in $\QBG^{S}$; recall from \S 2.1 that the weight $\wt^S(\bp)$ of 
this directed path is defined to be
\begin{equation*}
\wt^S(\bp) = \sum_{
 \begin{subarray}{c} 
 1 \le k \le p \\[1mm]
 \text{$x_{k-1} \edge{\gamma_{k}} x_{k}$ is } \\[1mm]
 \text{a quantum edge}
 \end{subarray}} \gamma_{k}^{\vee} \in Q^{\vee,+}.
\end{equation*}
We set 
%
%%%%%%%%%%%%%
%%% eq:xi %%%
%%%%%%%%%%%%%
%
\begin{equation} \label{eq:xi}
\xi_{x,y}\eqdef \wt^S(\bp)+\phi_{S}(\wt^S(\bp)) \in \Qad
\end{equation}
in the notation of Lemma~\ref{lem:S-adj}\,(1). 
We now claim that $\xi_{x,y}$ does not depend on 
the choice of a shortest directed path $\bp$ from $x$ to $y$ in $\QBG^{S}$. 
Indeed, let $\bp'$ be another directed path from $x$ to $y$ in $\QBG^{S}$. 
We know from \cite[Proposition~8.1]{LNSSS1} that $\wt^S (\bp )= \wt^S (\bp')$ mod $Q_{S}^{\vee}$.
Therefore, by Remark~\ref{rem:trans} (1), we obtain 
$\wt^S(\bp)+\phi_{S}(\wt^S(\bp)) =\wt^S(\bp')+\phi_{S}(\wt^S(\bp'))$.
%
%we see from \cite[Lemma~7.2]{LNSSS2} that 
%$\zeta\eqdef [\wt(\bp)] \in Q_{I \setminus S}^{\vee}$ 
%does not depend on the choice of $\bp$. Recall from Lemma~\ref{lem:S-adj}\,(1) 
%that $\zeta + \phi_{S}(\zeta) \in \Qad$. Also, we have
%%
%\begin{align*}
%\Qad \ni \xi_{x,y} & = \wt(\bp)+\phi_{S}(\wt(\bp)) \\
%& = \zeta + \underbrace{(\wt(\bp) - [\wt(\bp)]) + \phi_{S}(\wt(\bp))}_{\in Q_{S}^{\vee}}. 
%\end{align*}
%%
%By the uniqueness of $\phi_{S}(\zeta) \in Q_{S}^{\vee}$, we obtain
%%
%\begin{equation*}
%\phi_{S}(\zeta) = (\wt(\bp) - [\wt(\bp)]) + \phi_{S}(\wt(\bp)),
%\end{equation*}
%%
%and hence $\xi_{x,y} = \zeta + \phi_{S}(\zeta)$. 
This proves the claim.  

%%%%%%%%%%%%%%%%
%%% lem:min1 %%%
%%%%%%%%%%%%%%%%
%
\begin{lem} \label{lem:min1}
Let $x,\,y \in W^{S}$. Then we have $yz_{\xi_{x,y}}t(\xi_{x,y}) \sige x$.  
\end{lem}

\begin{proof}
We proceed by induction on the length $p$ of a shortest directed path 
from $x$ to $y$ in $\QBG^S$. 
If $p=0$, i.e., $x=y$, then $\xi_{x,y}=\xi_{x,x}=0$, 
and hence $z_{\xi_{x,y}}=t(\xi_{x,y})=e$. 
Thus the assertion of the lemma is obvious. 
Assume now that $p > 0$; let
\begin{equation*}
\bp : x=x_{0} \edge{\gamma_1} x_1 \edge{\gamma_2} \cdots \edge{\gamma_p} x_{p}=y
\end{equation*}
be a shortest directed path from $x$ to $y$ in $\QBG^S$. 
Then we deduce from \cite[Proposition~A.1.2]{INS} that 
$x \edge{\beta} \sss_{\beta}x$ in $\SBG^{S}$ 
(in particular, $\sss_{\beta}x \sige x$), 
where 
\begin{equation*}
\beta\eqdef 
 \begin{cases}
 x_{0}\gamma_1 & \text{if $x=x_{0} \edge{\gamma_1} x_1$ is a Bruhat edge}, \\[1.5mm]
 x_{0}\gamma_1+\delta & \text{if $x=x_{0} \edge{\gamma_1} x_1$ is a quantum edge}; 
 \end{cases}
\end{equation*}
note that 
\begin{equation*}
\sss_{\beta}x = \sss_{\beta}x_{0} = 
  \begin{cases}
  x_{1} 
  & \text{if $x=x_{0} \edge{\gamma_1} x_1$ is a Bruhat edge}, \\[1.5mm]
  x_{1}t(\gamma_{1}^{\vee})
  & \text{if $x=x_{0} \edge{\gamma_1} x_1$ is a quantum edge}.
  \end{cases}
\end{equation*}
In the case that $x=x_{0} \edge{\gamma_1} x_1$ is a quantum edge, 
we have $x_{1}t(\gamma_{1}^{\vee})= \sss_{\beta}x \in (W^{S})_{\aff}$, 
which implies, by \eqref{eq:WSaf} and the fact that $x_{1} \in W^{S}$, that
%
%%%%%%%%%%%%%%%%
%%% eq:min1a %%%
%%%%%%%%%%%%%%%%
%
\begin{equation} \label{eq:min1a}
\text{$\gamma_{1}^{\vee} \in \Qad$ 
and $z_{\gamma_{1}^{\vee}}=e$.}
\end{equation}

Assume first that $x=x_{0} \edge{\gamma_1} x_1$ is a Bruhat edge. 
Note that $\bp':x_1 \edge{\gamma_2} \cdots \edge{\gamma_p} x_{p}=y$ 
is a shortest directed path from $x_{1}$ to $y$ in $\QBG^S$. 
Since $\wt^S(\bp)=\wt^S(\bp')$ by the definition, 
we deduce that $\xi_{x,y} = \xi_{x_1,y}$. 
Also, by the induction hypothesis, we have $yz_{\xi_{x_1,y}}t(\xi_{x_1,y}) \sige x_1$. 
Combining these, we obtain $yz_{\xi_{x,y}}t(\xi_{x,y}) = 
yz_{\xi_{x_1,y}}t(\xi_{x_1,y}) \sige x_1 = \sss_{\beta}x \sige x$, as desired. 

Next, assume that $x=x_{0} \edge{\gamma_1} x_1$ is a quantum edge; 
we have $\wt^S(\bp)=\wt^S(\bp')+\gamma_{1}^{\vee}$, which implies that 
$\xi_{x,y} \equiv \xi_{x_{1},y} + \gamma_{1}^{\vee}$ mod $Q_{S}^{\vee}$. 
We compute
\begin{align*}
yz_{\xi_{x,y}}t(\xi_{x,y}) 
& = y\PJ(t(\xi_{x,y})) \quad 
  \text{by Lemma~\ref{lem:S-adj}\,(2)} \\
& = y\PJ(t(\xi_{x_1,y})t(\xi_{x,y}-\xi_{x_1,y})) \\
& = y\PJ(t(\xi_{x_1,y}))\PJ(t(\xi_{x,y}-\xi_{x_1,y})) 
  \quad \text{by Lemma~\ref{lem:PJ}\,(2)} \\
& = yz_{\xi_{x_1,y}}t(\xi_{x_1,y})\PJ(t(\xi_{x,y}-\xi_{x_1,y})).
\end{align*}
Since $\xi_{x,y} \equiv \xi_{x_{1},y} + \gamma_{1}^{\vee}$ mod $Q_{S}^{\vee}$, 
we see from Remark~\ref{rem:trans}\,(1) and \eqref{eq:min1a} that 
$\PJ(t(\xi_{x,y}-\xi_{x_1,y})) = t(\gamma_{1}^{\vee})$. 
Therefore, we deduce that, using the induction hypothesis 
$yz_{\xi_{x_1,y}}t(\xi_{x_1,y}) \sige x_1$ and Lemma~\ref{lem:trans}, 
\begin{align*} 
\underbrace{yz_{\xi_{x,y}}t(\xi_{x,y})}_{\in (W^{S})_{\aff}} 
& = \bigl(yz_{\xi_{x_1,y}}t(\xi_{x_1,y})\bigr) t(\gamma_{1}^{\vee})
  = \PJ\bigl(\bigl(yz_{\xi_{x_1,y}}t(\xi_{x_1,y})\bigr) t(\gamma_{1}^{\vee})\bigr) 
  \sige \PJ(x_{1}t(\gamma_{1}^{\vee})) \\
& = \PJ(\sss_{\beta}x) = \sss_{\beta}x \sige x. 
\end{align*}
This proves the lemma. 
\end{proof}
%
%%%%%%%%%%%%%%%%
%%% lem:min2 %%%
%%%%%%%%%%%%%%%%
%
\begin{lem} \label{lem:min2}
Let $x,\,y \in W^{S}$, and $\zeta \in \Qad$. 
If $yz_{\zeta}t(\zeta) \sige x$, then 
$[\zeta] \ge [\xi_{x,y}]$. 
\end{lem}

\begin{proof}
We set
\begin{equation*}
\ti{s}_j \eqdef  
  \begin{cases} 
   \sss_{j} & \text{if $j \ne 0$}, \\[1mm]
   \sss_{\theta} & \text{if $j = 0$},
  \end{cases}
\qquad \text{and} \qquad
\ti{\alpha}_{j} \eqdef  
   \begin{cases} 
   \alpha_{j} & \text{if $j \ne 0$}, \\[1mm]
   -\theta & \text{if $j = 0$}.
  \end{cases}
\end{equation*}
We know from \cite[Lemma~6.12]{LNSSS1} that 
there exist a sequence $x=x_{0}$, $x_{1}$, $\dots$, $x_{n}=e$ of 
elements of $W^{S}$ and a sequence $i_{1}$, $\dots$, $i_{n} \in 
I_{\aff}=I \sqcup \{0\}$ such that 
\begin{equation*}
x=x_{0} \edge{x_{0}^{-1}\ti{\alpha}_{i_1}} x_{1} 
\edge{x_{1}^{-1}\ti{\alpha}_{i_2}} \cdots 
\edge{x_{n-1}^{-1}\ti{\alpha}_{i_n}} x_{n}=e \quad \text{in $\QBG^{S}$};
\end{equation*}
note that $x_{k-1}^{-1}\ti{\alpha}_{i_k} \in \Delta^{+} \setminus \Delta^{+}_{S}$ 
for all $1 \le k \le n$. We prove the assertion of the lemma by induction on $n$. 

Assume first that $n=0$, i.e., $x=e$. 
Because $y \in W^{S}$ is greater than or equal to $e$ 
in the (ordinary) Bruhat order, there exists a directed path $\bp$ 
from $e$ to $y$ in $\QBG^{S}$ whose edges are all Bruhat edges 
(see, e.g., \cite[Theorem~2.5.5]{BB}); 
since $\wt^S(\bp)=0$, we obtain $\xi_{e,y}=
\wt^S(\bp)+\phi_{S}(\wt^S(\bp))=0$. 
Also, if $yz_{\zeta}t(\zeta) \sige x = e = ez_{0}t(0)$, then 
it follows from Lemma~\ref{lem:sig0} 
that $[\zeta] \ge [0] = [\xi_{e, y}]$, which proves 
the assertion in the case $n=0$. 

Assume next that $n > 0$; we set $i\eqdef i_{1}$ for simplicity of notation. 
Then, $x^{-1}\ti{\alpha}_{i}=x_{0}^{-1}\ti{\alpha}_{i} 
\in \Delta^{+} \setminus \Delta_{S}^{+}$, and 
the assertion of the lemma holds for $x_{1} = \ti{s}_{i}x_{0} = \ti{s}_{i}x$ 
by the induction hypothesis. 

%%%%%%%%%%%%%
{\bf Case 1.} Assume that $y^{-1}\ti{\alpha}_{i} \in (-\Delta^{+}) \cup \Delta^{+}_{S}$. 
We deduce by \cite[Lemma~7.7\,(3)]{LNSSS1} that 
\begin{equation} \label{eq:xi1}
\xi_{\ti{s}_{i}x,y} \equiv \xi_{x,y}-
 \delta_{i,0} x^{-1}\ti{\alpha}_{i}^{\vee} \mod Q_{S}^{\vee}. 
\end{equation}
Assume first that $i \ne 0$. 
Let $\zeta \in \Qad$ be such that $yz_{\zeta}t(\zeta) \sige x$. 
Because $x^{-1}\alpha_{i} \in \Delta^{+} \setminus \Delta_{S}^{+}$, 
and $y^{-1}\alpha_{i} \in (-\Delta^{+}) \cup \Delta_{S}^{+}$, 
we see from \cite[Lemma~4.1.6\,(2)]{INS} that 
$yz_{\zeta}t(\zeta) \sige \sss_{i}x = \ti{s}_{i}x$. 
Therefore, by the induction hypothesis, we obtain 
$[\zeta] \ge [\xi_{\ti{s}_{i}x,y}] \stackrel{\eqref{eq:xi1}}{=} [\xi_{x,y}]$.

Assume next that $i=0$. 
Let $\zeta \in Q^{\vee}$ be such that $yz_{\zeta}t(\zeta) \sige x$. 
Because $x^{-1}\ti{\alpha}_{0}=-x^{-1}\theta$ ($=$ 
the finite part of $x^{-1}\alpha_{0}$) $\in \Delta^{+} \setminus \Delta_{S}^{+}$, 
and $y^{-1}\ti{\alpha}_{0}=-y^{-1}\theta$ 
($=$ the finite part of $y^{-1}\alpha_{0}$) $\in (-\Delta^{+}) \cup \Delta_{S}^{+}$, 
we see from \cite[Lemma~4.1.6\,(2)]{INS} that 
\begin{equation*}
y z_{\zeta} t(\zeta) \sige \sss_{0}x = \sss_{\theta} x t(-x^{-1}\theta^{\vee}) = 
\underbrace{\ti{s}_{0}x}_{=x_{1}}t(x^{-1}\ti{\alpha}_{0}^{\vee})
\end{equation*}
Therefore, by Lemma~\ref{lem:trans}, 
\begin{align*}
\PJ(y z_{\zeta} t(\zeta-x^{-1}\ti{\alpha}_{0}^{\vee})) & = 
\PJ\bigl( (y z_{\zeta} t(\zeta)) t(-x^{-1}\ti{\alpha}_{0}^{\vee})\bigr) \\
& \sige 
\PJ(\ti{s}_{0}xt(x^{-1}\ti{\alpha}_{0}^{\vee})t(-x^{-1}\ti{\alpha}_{0}^{\vee})) 
= \PJ(\ti{s}_{0}x) \\
& = \PJ(x_{1}) = x_{1}= \ti{s}_{0}x. 
\end{align*}
If we write the left-hand side of this inequality as
$\PJ(y z_{\zeta} t(\zeta-x^{-1}\ti{\alpha}_{0}^{\vee}))= yz_{\zeta'}t(\zeta')$ 
for some $\zeta' \in \Qad$ (see Lemma~\ref{lem:S-adj}\,(2)), 
then we have $\zeta' \equiv \zeta-x^{-1}\ti{\alpha}_{0}^{\vee}$ mod $Q_{S}^{\vee}$. 
Also, by the induction hypothesis, we have $[\zeta'] \ge [\xi_{\ti{s}_{0}x,y}]$. 
Combining these, we obtain 
\begin{equation*}
[\zeta] = [\zeta'+x^{-1}\ti{\alpha}_{0}^{\vee}] \ge 
[\xi_{\ti{s}_{0}x,y} + x^{-1}\ti{\alpha}_{0}^{\vee}] 
\stackrel{\eqref{eq:xi1}}{=} [\xi_{x,y}], 
\end{equation*}
as desired. 

%%%%%%%%%%%%%
{\bf Case 2.} Assume that $y^{-1}\ti{\alpha}_{i} \in 
\Delta^{+} \setminus \Delta_{S}^{+}$. 
By \cite[Lemma~7.7\,(4)]{LNSSS1}, we have
\begin{equation} \label{eq:xi2}
\xi_{\ti{s}_{i}x,\mcr{\ti{s}_{i}y}} \equiv \xi_{x,y}
-\delta_{i,0} x^{-1}\ti{\alpha}_{i}^{\vee}
+\delta_{i,0} y^{-1}\ti{\alpha}_{i}^{\vee} \mod Q_{S}^{\vee}. 
\end{equation}
Assume first that $i \ne 0$; note that $\ti{s}_{i}y= \sss_{i}y \in W^{S}$ 
(see, e.g., \cite[Proposition~5.10]{LNSSS1}). 
Let $\zeta \in Q^{\vee}$ be such that $yz_{\zeta}t(\zeta) \sige x$. 
Because $x^{-1}\alpha_{i} \in \Delta^{+} \setminus \Delta_{S}^{+}$ and 
$y^{-1}\alpha_{i} \in \Delta^{+} \setminus \Delta_{S}^{+}$, we see that
\begin{equation*}
\ti{s}_{i}yz_{\zeta}t(\zeta) = \sss_{i}yz_{\zeta}t(\zeta) \sige \sss_{i}x = \ti{s}_{i}x
\quad \text{by \cite[Lemma~2.3.6\,(3)]{NS-D}}.
\end{equation*}
Therefore, by the induction hypothesis, we obtain 
$[\zeta] \ge [\xi_{\ti{s}_{i}x,\ti{s}_{i}y}] \stackrel{\eqref{eq:xi2}}{=} 
[\xi_{x,y}]$. 

Assume next that $i=0$. 
Let $\zeta \in Q^{\vee}$ be such that $yz_{\zeta}t(\zeta) \sige x$. 
Because $x^{-1}\ti{\alpha}_{0}=-x^{-1}\theta$ ($=$ 
the finite part of $x^{-1}\alpha_{0}$) $\in \Delta^{+} \setminus \Delta_{S}^{+}$ and 
$y^{-1}\ti{\alpha}_{0}=-y^{-1}\theta$ ($=$ 
the finite part of $y^{-1}\alpha_{0}$) $\in \Delta^{+} \setminus \Delta_{S}^{+}$, 
we see from \cite[Lemma~2.3.6\,(3)]{NS-D} that 
$\sss_{0}yz_{\zeta}t(\zeta) \sige \sss_{0}x$. Therefore, by Lemma~\ref{lem:trans}, 
we have
\begin{equation*}
\PJ\bigl((\sss_{0}yz_{\zeta}t(\zeta))t(-x^{-1}\ti{\alpha}_{0}^{\vee})\bigr) 
\sige \PJ\bigl((\sss_{0}x)t(-x^{-1}\ti{\alpha}_{0}^{\vee})\bigr).
\end{equation*}
Here we have
\begin{equation*}
\PJ\bigl((\sss_{0}x)t(-x^{-1}\ti{\alpha}_{0}^{\vee})\bigr) = 
\PJ\bigl((\ti{s}_{0}xt(x^{-1}\ti{\alpha}_{0}^{\vee}))t(-x^{-1}\ti{\alpha}_{0}^{\vee})\bigr) = 
\ti{s}_{0}x = x_{1}. 
\end{equation*}
Also, using Lemma~\ref{lem:S-adj}\,(2), we compute
\begin{align*}
& \PJ\bigl((\sss_{0}yz_{\zeta}t(\zeta))t(-x^{-1}\ti{\alpha}_{0}^{\vee})\bigr) 
  = \PJ(\sss_{0}yz_{\zeta}t(\zeta-x^{-1}\ti{\alpha}_{0}^{\vee})) \\
& \qquad 
  = \PJ(\sss_{0}yz_{\zeta})\PJ(t(\zeta-x^{-1}\ti{\alpha}_{0}^{\vee}))
  = \PJ(\sss_{0}y)\PJ(t(\zeta-x^{-1}\ti{\alpha}_{0}^{\vee})) \\
& \qquad = \PJ(\ti{s}_{0}yt(y^{-1}\ti{\alpha}_{0}^{\vee}))
    \PJ(t(\zeta-x^{-1}\ti{\alpha}_{0}^{\vee}))
  = \PJ(\ti{s}_{0}yt(y^{-1}\ti{\alpha}_{0}^{\vee})t(\zeta-x^{-1}\ti{\alpha}_{0}^{\vee})) \\
& \qquad 
  = \PJ(\ti{s}_{0}yt(\zeta+y^{-1}\ti{\alpha}_{0}^{\vee}-x^{-1}\ti{\alpha}_{0}^{\vee})). 
\end{align*}
If we write this element as
$\PJ\bigl((\sss_{0}yz_{\zeta}t(\zeta))t(-x^{-1}\ti{\alpha}_{0}^{\vee})\bigr) 
= \mcr{\sss_{0}y}z_{\zeta''}t(\zeta'')$ for some $\zeta'' \in \Qad$
(see Lemma~\ref{lem:S-adj}\,(2)), we see that 
$\zeta'' \equiv \zeta +y^{-1}\ti{\alpha}_{0}^{\vee}-x^{-1}\ti{\alpha}_{0}^{\vee}$ mod $Q_{S}^{\vee}$. 
In addition, by the induction hypothesis, 
we have $[\zeta''] \ge [\xi_{\ti{s}_{0}x,\mcr{\ti{s}_{0}y}}]$. 
Combining these, we obtain 
\begin{align*}
[\zeta] & = [\zeta''-y^{-1}\ti{\alpha}_{0}^{\vee}+x^{-1}\ti{\alpha}_{0}^{\vee}] \\
& \ge [\xi_{\ti{s}_{0}x,\mcr{\ti{s}_{0}y}} -y^{-1}\ti{\alpha}_{0}^{\vee} + x^{-1}\ti{\alpha}_{0}^{\vee}] 
\stackrel{\eqref{eq:xi2}}{=} [\xi_{x,y}], 
\end{align*}
as desired. This completes the proof of the lemma.
\end{proof}
%
%==============================%
%     START SUBSECTION 0104    %
%==============================%
%
\subsection{Semi-infinite Lakshmibai-Seshadri paths}
\label{subsec:SLS}
Let $\lambda \in P^+$ be a dominant weight; 
we set $S = S_{\lambda} = \bigl\{ i \in I \mid 
\pair{\lambda}{\alpha_i^{\vee}}=0 \bigr\} \subset I$.
%
%%%%%%%%%%%%%%%
%%% def:QBa %%%
%%%%%%%%%%%%%%%
%
\begin{dfn} \label{def:QBa}
For a rational number $0 < \sigma \le 1$, 
define $\SBa$ to be the subgraph of $\SBG^{S}$ 
with the same vertex set but having only the edges of the form:
$x \edge{\beta} y$ with 
$\sigma \pair{x\lambda}{\beta^{\vee}} \in \BZ$; 
note that $\SBx{1} = \SBG^{S}$. 
\end{dfn}
%
%%%%%%%%%%%%%%%%
%%% dfn:SiLS %%%
%%%%%%%%%%%%%%%%
%
\begin{dfn}\label{dfn:SiLS}
A semi-infinite Lakshmibai-Seshadri (SiLS for short) path of 
shape $\lambda $ is, by definition, a pair $\eta=(x_1 \sig \cdots \sig x_s \,;\,0 = \sigma_0 < \sigma_1 < \cdots  < \sigma_s =1)$ of 
a (strictly) decreasing sequence $x_1 \sig \cdots \sig x_s$ 
of elements in $(W^S)_{\aff}$ and an increasing sequence 
$0 = \sigma_0 < \sigma_1 < \cdots  < \sigma_s =1$ of rational numbers 
such that there exists a directed path 
from $x_{u+1}$ to  $x_u$ in $\SBx{\sigma_u}$ 
for all $u = 1,\,2,\,\dots,\,s-1$. 
We denote by $\SLS(\lambda)$ 
the set of all SiLS paths of shape $\lambda$. 
\end{dfn}

Following \cite[\S3.1]{INS} (see also \cite[\S2.4]{NS-D}), 
we endow the set $\SLS(\lambda)$ with a crystal structure 
with weights in $P_{\aff}$ by the root operators 
$e_{i}$, $f_{i}$, $i \in I_{\aff}$, and the map $\wt : \SLS(\lambda) \rightarrow P_\aff$ defined by  
\begin{equation} \label{eq:wt}
\begin{split}
& \wt (\eta) : = \sum_{u=1}^{s}(\sigma_{u}-\sigma_{u-1}) x_{u}\lambda \in P_{\aff} \\
& \qquad 
\text{for $\eta=(x_{1},\,\dots,\,x_{s}\,;\,\sigma_{0},\,\sigma_{1},\,\dots,\,\sigma_{s}) 
\in \SLS(\lambda)$}.  
\end{split}
\end{equation}
Let $\Conn(\SLS(\lambda))$ denote the set of 
all connected components of $\SLS(\lambda)$, and 
let $\SLS_{0}(\lambda) \in \Conn(\SLS(\lambda))$ denote 
the connected component of $\SLS(\lambda)$ 
containing $\eta_{e}\eqdef (e\,;\,0,\,1) \in \SLS(\lambda)$. 

Also, we define a surjective map $\cl : (W^{S})_{\aff} \twoheadrightarrow W^{S}$ by
\begin{equation*}
\cl (x) = w \quad \text{if $x = wz_{\xi}t(\xi)$, 
   with $w \in W^{S}$ and $\xi \in \Qad$},
\end{equation*}
and for $\eta = (x_{1},\,\dots,\,x_{s}\,;\,\sigma_{0},\,\sigma_{1},\,\dots,\,\sigma_{s}) 
\in \SLS(\lambda)$, we set
\begin{equation*}
\cl(\eta)\eqdef (\cl(x_{1}),\,\dots,\,\cl(x_{s})\,;\,\sigma_{0},\,\sigma_{1},\,\dots,\,\sigma_{s});
\end{equation*}
where, for each $1 \le p < q \le s$ such that $\cl(x_{p})= \cdots = \cl(x_{q})$, 
we drop $\cl(x_{p}),\,\dots,\,\cl(x_{q-1})$ and $\sigma_{p},\,\dots,\,\sigma_{q-1}$. 
We know from \cite[\S6.2]{NS-D} that $\cl(\eta) \in \QLS(\lambda)$. 
Thus we obtain a map $\cl:\SLS(\lambda) \rightarrow \QLS(\lambda)$. 
%
%%%%%%%%%%%%%%%
%%% rem:pie %%%
%%%%%%%%%%%%%%%
%
\begin{rem} \label{rem:pie}
Recall that $\psi_{e}\eqdef (e\,;\,0,\,1) \in \QLS(\lambda)$. 
We see from the definition that an element in $\cl^{-1}(\psi_{e})$ 
is of the form: 
%
%%%%%%%%%%%%%%
%%% eq:pie %%%
%%%%%%%%%%%%%%
%
\begin{equation} \label{eq:pie}
 (z_{\xi_1}t(\xi_1),\, 
  z_{\xi_2}t(\xi_2),\,\dots,\,z_{\xi_{s-1}}t(\xi_{s-1}),\,
  z_{\xi_{s}}t(\xi_{s}) \,;\, 
  \sigma_{0},\,\sigma_{1},\,\dots,\,\sigma_{s-1},\,\sigma_{s})
\end{equation}
for some $s \ge 1$ and $\xi_{1},\,\xi_{2},\,\ldots,\,\xi_{s} \in \Qad$. 
\end{rem}

The final direction of $\eta \in \SLS(\lambda)$ is defined to be
%
%%%%%%%%%%%%%%%%
%%% eq:kappa %%%
%%%%%%%%%%%%%%%%
%
\begin{equation} \label{eq:kappa}
\kappa(\eta)\eqdef x_{s} \in (W^{S})_{\aff} \quad 
\text{if $\eta=(x_{1},\,\dots,\,x_{s}\,;\,\sigma_{0},\,\sigma_{1},\,\dots,\,\sigma_{s})$}.
\end{equation}
Then, for $x \in (W^{S})_{\aff}$, we set
%
%%%%%%%%%%%%%%
%%% eq:BBx %%%
%%%%%%%%%%%%%%
%
\begin{equation} \label{eq:BBx}
\SLS_{\sige x}(\lambda)\eqdef 
 \bigl\{ \eta \in \SLS(\lambda) \mid \kappa(\eta) \sige x \bigr\}. 
\end{equation}

The next lemma follows from \cite[Lemma~7.1.4]{INS}. 
%
%%%%%%%%%%%%%%%%%%
%%% lem:INS714 %%%
%%%%%%%%%%%%%%%%%%
%
\begin{lem} \label{lem:INS714}
Let $\eta \in \SLS_{0}(\lambda)$, and let $X$ be a monomial in root operators 
such that $\eta=X\eta_{e}$. Assume that $\eta_{0} \in \SLS(\lambda)$ is 
of the form \eqref{eq:pie}. Then, $\kappa(X\eta_{0})=\kappa(\eta)\kappa(\eta_{0})$. 
\end{lem}

Now, we recall from \S 3.2 the degree function $\dg_{\lambda}:\QLS(\lambda) \rightarrow \BZ_{\le 0}$ for the case $\mu = \lambda$. 
We know the following lemma from \cite[Lemma~6.2.3]{NS-D}. 
%
%%%%%%%%%%%%%%%%%%
%%% lem:pi-eta %%%
%%%%%%%%%%%%%%%%%%
%
\begin{lem} \label{lem:pi-eta}
For each $\psi \in \QLS(\lambda)$, 
there exists a unique $\eta_{\psi} \in \SLS_{0}(\lambda)$ 
such that $\cl(\eta_{\psi})=\psi$ and $\kappa(\eta_{\psi}) \in W^{S}$. 
\end{lem}

Let $\psi \in \QLS(\lambda)$. 
We know from \cite[(6.2.5)]{NS-D} that $\wt(\eta_{\psi})$ is of the form: 
%
%%%%%%%%%%%%%%%
%%% eq:wtep %%%
%%%%%%%%%%%%%%%
%
\begin{equation} \label{eq:wtep}
\wt(\eta_{\psi}) = 
\underbrace{\lambda-\gamma}_{=\wt(\psi)} + K \delta \quad 
\text{for some $\gamma \in Q^{+}$ and $K \in \BZ_{\le 0}$};
\end{equation}
also,
we know from \cite[Corollary~4.8]{LNSSS2} 
(see also the comment after \cite[(6.2.5)]{NS-D}) that
%
%%%%%%%%%%%%%%
%%% eq:deg %%%
%%%%%%%%%%%%%%
%
\begin{equation} \label{eq:deg}
K = 
- \sum_{u=1}^{s-1} \sigma_{u} \wt_{\lambda}(w_{u+1} \Rightarrow w_{u}) = \dg_\lambda (\psi)
\end{equation}
for $\psi = (w_{1},\,\dots,\,w_{s}\,;\,\sigma_{0},\,\sigma_{1},\,\dots,\,\sigma_{s}) \in \QLS(\lambda)$;
here
we should note that in the definition of $\dg_\lambda (\psi)$, $w_{s+1} = v(\lambda) = e$, and hence that
$\wt_\lambda (w_{s+1} \Rightarrow w_{s}) = \wt_\lambda (e \Rightarrow w_{s})=0$.

Let us write a dominant weight $\lambda \in P^+$ as 
$\lambda = \sum_{i \in I} m_{i} \varpi_{i}$ 
with $m_{i} \in \BZ_{\ge 0}$ for $i \in I$, and define 
$\ol{\Par(\lambda)}$ (resp., $\Par(\lambda)$) to be 
the set of multi-partitions $\brho = (\rho^{(i)})_{i \in I}$ such that  $\rho^{(i)}$ is a partition of length less than or equal to $m_i$ (resp., strictly less than $m_i$) 
for each $i \in I$; a partition of length less 
than $0$ is understood to be the empty partition $\emptyset$;
note that $\Par(\lambda) \subset \ol{\Par(\lambda)}$.
Also, for $\brho = (\rho^{(i)})_{i \in I} \in \ol{\Par(\lambda)}$, we set 
$|\brho|\eqdef \sum_{i \in I} |\rho^{(i)}|$, where for a partition 
$\chi = (\chi_1 \ge \chi_2 \ge \cdots \ge \chi_{m})$, 
we set $|\chi| \eqdef  \chi_{1}+\cdots+\chi_{m}$. 
Following \cite[(3.2.2)]{INS}, 
we endow the set $\Par(\lambda)$ with a crystal structure 
with weights in $P_{\aff}$; note  that $\wt (\brho) = - |\brho| \delta$. 
%
%%%%%%%%%%%%%%%%
%%% prop:SLS %%%
%%%%%%%%%%%%%%%%
%
\begin{prop} \label{prop:SLS}
Keep the notation above. 

\begin{enu}
\item Each connected component $C \in \Conn(\SLS(\lambda))$ of 
$\SLS(\lambda)$ contains a unique element of the form:
%
%%%%%%%%%%%%%%%
%%% eq:etaC %%%
%%%%%%%%%%%%%%%
%
\begin{equation} \label{eq:etaC}
\eta^{C} = 
 (z_{\xi_1}t(\xi_1),\, 
  z_{\xi_2}t(\xi_2),\,\dots,\,z_{\xi_{s-1}}t(\xi_{s-1}),\,e \,;\, 
  \sigma_{0},\,\sigma_{1},\,\dots,\,\sigma_{s-1},\,\sigma_{s})
\end{equation}
for some $s \ge 1$ and $\xi_{1},\,\xi_{2},\,\ldots,\,\xi_{s-1} \in \Qad$ 
{\rm(}see \cite[Proposition~7.1.2]{INS}{\rm)}.

\item There exists a bijection 
$\Theta:\Conn(\SLS(\lambda)) \rightarrow \Par(\lambda)$ such that 
$\wt(\eta^{C})=\lambda-|\Theta(C)|\delta$ 
{\rm(}see \cite[Proposition~7.2.1 and its proof]{INS}{\rm)}. 

\item Let $C \in \Conn(\SLS(\lambda)$. 
Then, there exists an isomorphism $C \stackrel{\sim}{\rightarrow}
\bigl\{\Theta(C)\bigr\} \otimes \SLS_{0}(\lambda)$ 
of crystals that maps $\eta^{C}$ to $\Theta(C) \otimes \eta_{e}$. 
Consequently, $\SLS(\lambda)$ is 
isomorphic as a crystal to $\Par(\lambda) \otimes \SLS_{0}(\lambda)$ 
{\rm(}see \cite[Proposition~3.2.4 and its proof]{INS}{\rm)}. 
\end{enu}
\end{prop}
%
%==============================%
%     START SUBSECTION 0105    %
%==============================%
%
\subsection{Extremal weight modules}
\label{subsec:ext}

In this and the next subsection, we mainly follow the notation
of \cite[Sections 4 and 5]{NS-D};
we use the symbol  ``$q_s$" for the
quantum parameter in order to distinguish it from $q = e^{\delta}$.
Let $\lambda \in P^+$ be a dominant  weight. 
We denote by $V(\lambda)$ the extremal weight module of 
extremal weight $\lambda$ over $\qad$; this is 
the integrable $\qad$-module generated 
by a single element $v_{\lambda}$ with 
the defining relation that $v_{\lambda}$ is 
an ``extremal weight vector'' of weight $\lambda$
(for details, see \cite[\S8]{K-mod} and \cite[\S3]{K-lv0}).
We know from \cite[Proposition~8.2.2]{K-mod} that $V(\lambda)$ has 
a crystal basis $(\CL(\lambda),\,\CB(\lambda))$ with global basis 
$\bigl\{G(b) \mid b \in \CB(\lambda)\bigr\}$. 
Denote by $u_{\lambda}$ the element of $\CB(\lambda)$ 
such that $G(u_{\lambda})=v_{\lambda} \in V(\lambda)$, 
and by $\CB_{0}(\lambda)$ the connected component of $\CB(\lambda)$ 
containing $u_{\lambda}$. 

Let $\qa \subset \qad$ denote 
the quantized universal enveloping algebra without the degree operator. 
We know the following from \cite{K-lv0} (see also \cite[\S5.2]{NS-D}): 
\begin{enu}
\item[(i)]
for each $i \in I$, there exists a $\qa$-module 
automorphism $z_{i}:V(\vpi_i) \rightarrow V(\vpi_i)$ 
that maps $v_{\vpi_{i}}$ to $v_{\vpi_{i}}^{[1]}\eqdef G(u_{\vpi_{i}}^{[1]})$, 
where $u_{\vpi_{i}}^{[1]} \in \CB(\vpi_{i})$ is 
a (unique) element of weight $\vpi_{i}+\delta$;

\item[(ii)] the map $z_{i}:V(\vpi_{i}) \rightarrow V(\vpi_{i})$ 
induces a bijection $z_{i}:\CB(\vpi_{i}) \rightarrow \CB(\vpi_{i})$ 
that maps $u_{\vpi_{i}}$ to $u_{\vpi_{i}}^{[1]}$; this map 
commutes with the Kashiwara operators $e_{j}$, $f_{j}$, $j \in I_{\aff}$, 
on $\CB(\vpi_{i})$. 
\end{enu}

Let us write a dominant weight $\lambda \in P^+$ as 
$\lambda = \sum_{i \in I} m_{i} \vpi_{i}$, 
with $m_{i} \in \BZ_{\ge 0}$ for $i \in I$. 
We fix an arbitrary total ordering on $I$, and then set 
$\ti{V}(\lambda)\eqdef 
\bigotimes_{i \in I} V(\vpi_{i})^{\otimes m_{i}}$.
By \cite[eq.\,(4.8) and Corollary~4.15]{BN}, 
there exists a $\qad$-module embedding
$\Phi_{\lambda}:V(\lambda) \hookrightarrow \ti{V}(\lambda)$
that maps $v_{\lambda}$ to $\ti{v}_{\lambda}\eqdef 
\bigotimes_{i \in I} v_{\vpi_{i}}^{\otimes m_{i}}$. 
Also, for each $i \in I$ and $1 \le k \le m_{i}$, we define $z_{i,k}$ to be 
the $\qa$-module automorphism of $\ti{V}(\lambda)$ 
which acts as $z_{i}$ only on the $k$-th factor of $V(\vpi_{i})^{\otimes m_{i}}$ 
in $\ti{V}(\lambda)$, and as the identity map on the other factors of $\ti{V}(\lambda)$; 
these $z_{i,k}$'s, $i \in I$, $1 \le k \le m_{i}$, 
commute with each other. 
Now, for $\brho=(\rho^{(i)})_{i \in I} \in \ol{\Par(\lambda)}$, we set
%
%%%%%%%%%%%%%%%%
%%% eq:Schur %%%
%%%%%%%%%%%%%%%%
%
\begin{equation} \label{eq:Schur}
s_{\brho}(z^{-1})\eqdef \prod_{i \in I} 
s_{\rho^{(i)}}(z_{i,1}^{-1},\,\dots,\,z_{i,m_i}^{-1}). 
\end{equation}
Here, for a partition 
$\rho=(\rho_{1} \ge \cdots \ge \rho_{m-1} \ge 0)$ 
of length less than $m \in \BZ_{\ge 1}$, 
$s_{\rho}(x)=s_{\rho}(x_{1},\,\dots,\,x_{m})$ denotes 
the Schur polynomial in the variables $x_{1},\,\dots,\,x_{m}$ 
corresponding to the partition $\rho$. 
We can easily show (see \cite[\S7.3]{NS-D}) that 
$s_{\brho}(z^{-1})(\Img \Phi_{\lambda}) \subset \Img \Phi_{\lambda}$ 
for each $\brho = (\rho^{(i)})_{i \in I} \in \ol{\Par(\lambda)}$. 
Hence we can define a $\qa$-module homomorphism
$z_{\brho}:V(\lambda) \rightarrow V(\lambda)$ in such a way that 
the following diagram commutes:
%
%%%%%%%%%%%%%%%
%%% eq:zrho %%%
%%%%%%%%%%%%%%%
%
\begin{equation} \label{eq:zrho}
\begin{CD}
V(\lambda) @>{\Phi_{\lambda}}>> \ti{V}(\lambda) \\
@V{z_{\brho}}VV @VV{s_{\brho}(z^{-1})}V \\
V(\lambda) @>{\Phi_{\lambda}}>> \ti{V}(\lambda);
\end{CD}
\end{equation}
note that $z_{\brho}v_{\lambda} = S_{\brho}^{-}v_{\lambda}$ 
in the notation of \cite{BN} (and \cite{NS-D}). 
The map $z_{\brho}:V(\lambda) \rightarrow V(\lambda)$ induces a $\BC$-linear map 
$z_{\brho}:\CL(\lambda)/q_{s}\CL(\lambda) \rightarrow \CL(\lambda)/q_{s}\CL(\lambda)$; 
this map commutes with Kashiwara operators. 
It follows from \cite[p.\,371]{BN} that
%
%%%%%%%%%%%%%%%%
%%% eq:CBlam %%%
%%%%%%%%%%%%%%%%
%
\begin{equation} \label{eq:CBlam}
\CB(\lambda) = \bigl\{
 z_{\brho}b \mid 
 \brho \in \Par(\lambda),\,b \in \CB_{0}(\lambda) \bigr\};
\end{equation}
for $\brho \in \Par(\lambda)$, we set 
%
%%%%%%%%%%%%%%%%
%%% eq:u-rho %%%
%%%%%%%%%%%%%%%%
%
\begin{equation} \label{eq:u-rho}
u^{\brho} \eqdef  z_{\brho}u_{\lambda} \in \CB(\lambda). 
\end{equation}
%
%%%%%%%%%%%%%%%
%%% rem:zGb %%%
%%%%%%%%%%%%%%%
%
\begin{rem} \label{rem:zGb}
We see from \cite[Theorem~4.16\,(ii)]{BN} (see also the argument after 
\cite[(7.3.8)]{NS-D}) that $z_{\brho}G(b) = G(z_{\brho}b)$ for 
$b \in \CB_{0}(\lambda)$ and $\brho \in \ol{\Par(\lambda)}$. 
\end{rem}
%
%==============================%
%     START SUBSECTION 0106    %
%==============================%
%
\subsection{Demazure submodules}
\label{subsec:Dem}
Let $\lambda \in P^+$ be a dominant weight.
For each $x \in W_{\aff}$, we set
%
%%%%%%%%%%%%%%
%%% eq:dem %%%
%%%%%%%%%%%%%%
%
\begin{equation} \label{eq:dem}
V_{x}^{-}(\lambda)\eqdef U_{q_s}^{-}(\Fg_{\aff})S_{x}^{\norm}v_{\lambda} 
\subset V(\lambda),
\end{equation}
where $S_{x}^{\norm}v_{\lambda}$ denotes the extremal weight vector 
of weight $x\lambda$ (see, e.g., \cite[(3.2.1)]{NS-D});
%%%%%%%%%%
%%%%%%%%%%
%begin 4
%%%%%%%%%%
%%%%%%%%%%
since $V_{x}^-(\lambda) = V_{\PJ (x)}^- (\lambda)$
for $x \in W_\aff$ by \cite[Lemma 4.1.2]{NS-D},
we consider Demazure submodules $V_x^- (\lambda)$
only for  $x \in (W^{S})_{\aff}$ in what follows.
%%%%%%%%%%
%%%%%%%%%%
%end 4
%%%%%%%%%%
%%%%%%%%%%
%
We know from \cite[\S2.8]{K-rims} and \cite[\S4.1]{NS-D} that 
$V_{x}^{-}(\lambda)$ is ``compatible'' 
with the global basis of $V(\lambda)$, that is, 
there exists a subset $\CB_{x}^{-}(\lambda) \subset 
\CB(\lambda)$ such that 
%
%%%%%%%%%%%%%%%
%%% eq:deme %%%
%%%%%%%%%%%%%%%
%
\begin{equation} \label{eq:deme}
V_{x}^{-}(\lambda) = 
\bigoplus_{b \in \CB_{x}^{-}(\lambda)} \BC(q_{s}) G(b) 
\subset
V(\lambda) = 
\bigoplus_{b \in \CB(\lambda)} \BC(q_{s}) G(b).
\end{equation}

We know the following theorem 
from \cite[Theorem~3.2.1]{INS} and \cite[Theorem~4.2.1]{NS-D}. 
%
%%%%%%%%%%%%%%%%
%%% thm:isom %%%
%%%%%%%%%%%%%%%%
%
\begin{thm} \label{thm:isom}
Let $\lambda \in P^+$ be a dominant weight. 
There exists an isomorphism $\Psi_{\lambda}:\CB(\lambda) 
\stackrel{\sim}{\rightarrow} \SLS(\lambda)$ of crystals 
such that 
\begin{enu}
\item[(a)] $\Psi_{\lambda}(u^{\brho}) = \eta^{\Theta^{-1}(\brho)}$ 
for all $\brho \in \Par(\lambda)$ {\rm(}in particular, $\Psi_{\lambda}(u_{\lambda})=\eta_{e}${\rm)}{\rm;} 

\item[(b)] $\Psi_{\lambda}(\CB_{x}^{-}(\lambda)) = \SLS_{\sige x}(\lambda)$ 
for all $x \in (W^{S})_{\aff}$. 
\end{enu}

\end{thm}
%
%==============================%
%     START SUBSECTION 0107    %
%==============================%
%
\subsection{Weyl group action}
\label{subsec:Weyl}

Let $\CB$ be a regular crystal
for $\qad$
%such as $\CB(\lambda) \cong \SLS(\lambda)$, 
in the sense of \cite[\S2.2]{K-lv0} (or \cite[p.\,389]{K-mod}); 
in particular, as a crystal for $U_{q_s} (\Fg) \subset \qad$, 
it decomposes into a disjoint union of ordinary highest weight crystals.
By \cite[\S7]{K-mod}, the Weyl group $W_{\aff}$ acts on $\CB$ by
%
%%%%%%%%%%%%%
%%% eq:W1 %%%
%%%%%%%%%%%%%
%
\begin{equation} \label{eq:W1}
\sss_{j} \cdot b\eqdef 
\begin{cases}
f_{j}^{n}b & \text{if $n\eqdef \pair{\wt b}{\alpha_{j}^{\vee}} \ge 0$}, \\[1.5mm]
e_{j}^{-n}b & \text{if $n\eqdef \pair{\wt b}{\alpha_{j}^{\vee}} \le 0$}
\end{cases}
\end{equation}
for $b \in \CB$ and $j \in I_{\aff}$. 
Here we note that $\SLS(\lambda)$ is a regular crystal for $\qad$
for a dominant weight $\lambda \in P^+$.
%
%%%%%%%%%%%%%%%%
%%% rem:extp %%%
%%%%%%%%%%%%%%%%
%
\begin{rem}[{\cite[Remark~3.5.2]{NS-D}}] \label{rem:extp}
Recall from Remark~\ref{rem:pie} that 
every element $\eta \in \cl^{-1}(\psi_{e})$ is of the form \eqref{eq:pie}. 
Then, for each $x \in W_{\aff}$, 
%
%%%%%%%%%%%%%%%
%%% eq:xeta %%%
%%%%%%%%%%%%%%%
%
\begin{equation} \label{eq:xeta}
x \cdot \eta=
 \bigl(\PJ(xz_{\zeta_1}t(\zeta_1)),\,\dots,\,\PJ(xz_{\zeta_{s}}t(\zeta_{s})) \,;\, 
  \sigma_{0},\,\sigma_{1},\,\dots,\,\sigma_{s}\bigr),
\end{equation}
where 
$S = S_\lambda = \{ i \in I \mid \pair{\lambda}{\alpha_i^\lor} =0 \}$.
In particular, we see by \eqref{eq:xeta} and 
the uniqueness of $\eta^{C}$ that 
$\eta= (z_{\zeta_{s}}t(\zeta_{s})) \cdot \eta^{C}$, 
with $C \in \Conn(\SLS(\lambda))$ 
the connected component containing the $\eta$. 
\end{rem}
%
%%%%%%%%%%%%%%%
%%% rem:par %%%
%%%%%%%%%%%%%%%
%
\begin{rem} \label{rem:par}
Let $\brho = (\rho^{(i)})_{i \in I} \in \ol{\Par(\lambda)}$.
Denote by $c_{i} \in \BZ_{\ge 0}$, $i \in I$, the number of columns of 
length $m_{i}$ in the Young diagram corresponding to the partition $\rho^{(i)}$, 
and set $\xi : = \sum_{i \in I} c_{i}\alpha_{i}^{\vee} \in Q^{\vee,+}$; 
note that $c_{i}=0$ for all $i \in S$. 
Also, for $i \in I$, let $\varrho^{(i)}$ denote
the partition corresponding to the Young diagram 
obtained from 
that of $\varrho^{(i)}$
%the Young diagram corresponding to $\rho^{(i)}$
by removing all columns of length $m_{i}$ 
(i.e., the first $c_{i}$-columns), 
and set $\bvrho\eqdef (\varrho^{(i)})_{i \in I}$; 
note that $\bvrho \in \Par(\lambda)$. Then we deduce from 
\cite[Lemma~4.14 and its proof]{BN} that
%
%%%%%%%%%%%%%%
%%% eq:zcu %%%
%%%%%%%%%%%%%%
%
\begin{equation} \label{eq:zcu}
z_{\brho}u_{\lambda} = t(\xi) \cdot (z_{\bvrho}u_{\lambda}) = 
t(\xi) \cdot u^{\bvrho}.
\end{equation}
\end{rem}

%
%=========================%
%     START SECTION 02    %
%=========================%
%
\section{Graded character formulas for Demazure submodules \\ and their certain quotients}
\label{sec:gch}
%
%==============================%
%     START SUBSECTION 0201    %
%==============================%
%
\subsection{Graded character formula for Demazure submodules}
\label{subsec:gch}

Fix a dominant weight $\lambda \in P^+$; 
recall that $S=S_{\lambda}=\bigl\{ i \in I \mid 
\pair{\lambda}{\alpha_{i}^{\vee}}=0 \bigr\}$.

Because every weight space of the Demazure submodule 
$V_{x}^{-}(\lambda)$ corresponding to $x \in W^S = W \cap (W^{S})_{\aff}$ 
is finite-dimensional, we can define the (ordinary) character 
$\ch V_{x}^{-}(\lambda)$ of $V_{x}^{-}(\lambda)$ by
\begin{equation*}
\ch V_{x}^{-}(\lambda) \eqdef  \sum_{\beta \in Q_{\aff}}
\dim V_{x}^{-}(\lambda)_{\lambda-\beta}\,e^{\lambda-\beta},
\end{equation*}
where $V_{x}^{-}(\lambda)_{\lambda-\beta}$ denotes the $\lambda-\beta$ weight space of $V_{x}^{-}(\lambda)$.
Here we recall that an element $\beta \in Q_{\aff}$ 
can be written uniquely in the form: $\beta = \gamma + k\delta$ 
for $\gamma \in Q$ and $k \in \BZ$; if we set 
$e^{\delta}\eqdef q$, then $e^{\lambda-\beta} = 
e^{\lambda-\gamma} q^{-k}$. 
Now we define the graded character
$\gch V_{x}^{-}(\lambda)$ of $V_{x}^{-}(\lambda)$ to be 
\begin{equation*}
\gch V_{x}^{-}(\lambda) \eqdef  \sum_{\gamma \in Q,\,k \in \BZ} 
\dim V_{x}^{-}(\lambda)_{\lambda-\gamma-k\delta}\,e^{\lambda-\gamma}q^{-k},
\end{equation*}
which is obtained from the ordinary character $\ch V_{x}^{-}(\lambda)$ 
by replacing $e^{\delta}$ with $q$. 
%
%%%%%%%%%%%%%%%
%%% thm:gch %%%
%%%%%%%%%%%%%%%
%
\begin{thm} \label{thm:gch}
Keep the notation and setting above. 
Let $x \in W^{S}$. 
The graded character $\gch V_{x}^{-}(\lambda)$ of 
$V_{x}^{-}(\lambda)$ can be expressed as 
%
%%%%%%%%%%%%%%
%%% eq:gch %%%
%%%%%%%%%%%%%%
%
\begin{equation} \label{eq:gch}
\gch V_{x}^{-}(\lambda) = 
	\left( \prod_{i \in I} \prod_{r=1}^{m_{i}}(1-q^{-r})^{-1} \right)
  \sum_{\psi \in \QLS(\lambda)} e^{\wt(\psi)} q^{\dg_{x\lambda}(\psi)}.
\end{equation}
%
%where $\dg_{x (\psi) = \dg_{x\lambda} (\psi)$ in the notation of \S $3$.
\end{thm}
%
%%%%%%%%%%%%%%
%%% ((1)) %%%
%%%%%%%%%%%%%%
%
By combining the special case $x = \lfloor \lon \rfloor \in W^S$ of Theorem~\ref{thm:gch}
with the special case $\mu = \lon \lambda$ of Theorem~\ref{main}, we obtain the following theorem;
recall from Remark~\ref{qls_qlsw0} that $\QLS^{\lon \lambda, \infty} (\lambda)=\QLS(\lambda)$.

\begin{thm} \label{thm:demazure_character}
Let $\lambda \in P^+$ be a dominant weight of the from $\lambda = \sum_{i \in I} m_i \varpi_i$,
with $m_i \in \BZ_{\geq 0}$, $i \in I$.
Then, the graded character $\gch V_{\lon}^{-}(\lambda)$ is equal to 
\begin{equation*}
\left(
\prod_{i \in I} \prod_{r=1}^{m_i}(1 - q^{-r})^{-1} \right) E_{\lon \lambda} (q, \infty)
.
\end{equation*}
\end{thm}

%
%==============================%
%     START SUBSECTION 0202    %
%==============================%
%
\subsection{Proof of Theorem~\ref{thm:gch}}
\label{subsec:prf-gch}

We see from Theorem~\ref{thm:isom}  that 
\begin{equation*}
\ch V_{x}^{-}(\lambda) = 
\sum_{\eta \in \SLS_{\sige x}(\lambda)} e^{\wt (\eta)};
\end{equation*}
since 
\begin{equation*}
\SLS_{\sige x}(\lambda) = 
\bigsqcup_{\psi \in \QLS(\lambda)} 
 \bigl(\cl^{-1}(\psi) \cap \SLS_{\sige x}(\lambda)\bigr), 
\end{equation*}
we deduce that
\begin{equation} \label{eq:ch1}
\ch V_{x}^{-}(\lambda) = 
\sum_{\psi \in \QLS(\lambda)} 
\Biggl(\underbrace{\sum_{\eta \in \cl^{-1}(\psi) \cap \SLS_{\sige x}(\lambda)} 
e^{\wt (\eta)}}_{(\ast)}\Biggr).
\end{equation}
In order to obtain the graded character formula \eqref{eq:gch} for $V_{x}^{-}(\lambda)$, 
we will compute the sum ($\ast$) of the terms $e^{\wt(\eta)}$ 
over all $\eta \in \cl^{-1}(\psi) \cap \SLS_{\sige x}(\lambda)$
for each $\psi \in \QLS(\lambda)$. 
Let $\psi \in \QLS(\lambda)$, and take 
$\eta_{\psi} \in \SLS_{0}(\lambda)$ as in Lemma~\ref{lem:pi-eta}. 
Let $X$ be a monomial in root operators 
such that $\eta_{\psi}=X\eta_{e}$, where $\eta_{e}=(e\,;\,0,1)$. 
We see by \cite[Lemma~6.2.2]{NS-D} that
%
%%%%%%%%%%%%%%%
%%% eq:gch1 %%%
%%%%%%%%%%%%%%%
%
\begin{equation} \label{eq:gch1}
\cl^{-1}(\psi)=\bigl\{X (t(\zeta) \cdot \eta^{C}) \mid 
 C \in \Conn(\SLS(\lambda)),\,
 \zeta \in Q^{\vee} \bigr\};
\end{equation}
for the definition of $\eta^{C}$, see \eqref{eq:etaC}. We claim that
%
%%%%%%%%%%%%%%%%%%
%%% eq:grch1-1 %%%
%%%%%%%%%%%%%%%%%%
%
\begin{equation} \label{eq:grch1-1} 
\cl^{-1}(\psi) \cap \SLS_{\sige x}(\lambda) = 
\left\{ X (t(\zeta) \cdot \eta^{C}) \ \Biggm| \ 
 \begin{array}{l}
 C \in \Conn(\SLS(\lambda)), \\[1mm]
 \zeta \in Q^{\vee},\,
 [\zeta] \ge [\xi_{x,\kappa(\psi)}]
 \end{array}
\right\}. 
\end{equation}
We first show the inclusion $\subset$. 
Let $\eta \in \cl^{-1}(\psi) \cap \SLS_{\sige x}(\lambda)$, 
and write it as $\eta=X (t(\zeta) \cdot \eta^{C})$ 
for some $C \in \Conn(\SLS(\lambda))$ and 
some $\zeta \in Q^{\vee}$ (see \eqref{eq:gch1}). 
Also, we set $y\eqdef \kappa(\psi) = \kappa(\eta_\psi) \in W^{S}$.
We see by \eqref{eq:xeta} that $t(\zeta) \cdot \eta^{C}$ is
of the form \eqref{eq:pie}, with $\kappa(t(\zeta) \cdot \eta^{C}) = 
\PJ(t(\zeta)) = z_{\zeta}t(\zeta+\phi_{S}(\zeta))$. 
Therefore, we deduce from Lemma~\ref{lem:INS714} that 
$\kappa(X (t(\zeta) \cdot \eta^{C}) ) = \kappa(\eta_\psi) \kappa(t(\zeta) \cdot \eta^C)  = yz_{\zeta} t(\zeta+\phi_{S}(\zeta))$. 
Since $\eta = X (t(\zeta) \cdot \eta^{C}) \in \SLS_{\sige x}(\lambda)$ by the assumption, 
we have $yz_{\zeta} t(\zeta+\phi_{S}(\zeta)) \sige x$. 
Hence it follows from Lemma~\ref{lem:min2} that 
$[\zeta]=[\zeta+\phi_{S}(\zeta)] \ge [\xi_{x,y}]=[\xi_{x,\kappa(\psi)}]$. 
Thus, $\eta$ is contained in the set on the right-hand side of \eqref{eq:grch1-1}. 

For the opposite inclusion $\supset$, let $C \in \Conn(\SLS(\lambda))$, and 
let $\zeta \in Q^{\vee}$ be such that $[\zeta] \ge [\xi_{x,\kappa(\psi)}]$. 
It is obvious by \eqref{eq:gch1} that 
$X(t(\zeta) \cdot \eta^{C}) \in \cl^{-1}(\psi)$. 
Hence it suffices to show that 
$X(t(\zeta) \cdot \eta^{C}) \in \SLS_{\sige x}(\lambda)$. 
The same argument as above shows that 
$\kappa(X (t(\zeta) \cdot \eta^{C}) ) = yz_{\zeta} t(\zeta+\phi_{S}(\zeta))$, with $y \eqdef  \kappa(\psi) \in W^{S}$. 
Therefore, we see that
\begin{align*}
\kappa(X (t(\zeta) \cdot \eta^{C}) )
 & = yz_{\zeta} t(\zeta+\phi_{S}(\zeta))
 \sige yz_{\xi_{x,y}}t(\xi_{x,y}) \quad \text{by Lemma~\ref{lem:sig1}} \\
 & \sige x \quad \text{by Lemma~\ref{lem:min1}}, 
\end{align*}
which implies that $X(t(\zeta) \cdot \eta^{C}) \in 
\SLS_{\sige x}(\lambda)$. 
This proves \eqref{eq:grch1-1}. 

Let $C \in \Conn(\SLS(\lambda))$, and 
write $\Theta(C) \in \Par(\lambda)$ as: 
$\Theta(C)=(\rho^{(i)})_{i \in I}$, with 
$\rho^{(i)} = (\rho^{(i)}_{1} \ge \cdots \ge \rho^{(i)}_{m_{i}-1})$ for each $i \in I$. 
Also, let $\zeta \in Q^{\vee}$ be such that 
$[\zeta] \ge [\xi_{x,\kappa(\psi)}]$, and 
write the difference $[\zeta] - [\xi_{x,\kappa(\psi)}] \in Q^{\vee,+}$ as 
\begin{equation*}
[\zeta] - [\xi_{x,\kappa(\psi)}] = \sum_{i \in I} c_{i}\alpha_{i}^{\vee};
\end{equation*} 
note that $c_{i} = 0$ for all $i \in S$. 
Now, for each $i \in I$, we set $c_{i}+\rho^{(i)} \eqdef  
(c_{i}+\rho^{(i)}_{1} \ge \cdots \ge c_{i}+\rho^{(i)}_{m_{i}-1} \ge c_{i})$, 
which is a partition of length less than or equal to $m_{i}$, and then set 
%
%%%%%%%%%%%%%%%
%%% eq:par+ %%%
%%%%%%%%%%%%%%%
%
\begin{equation} \label{eq:par+}
(c_{i})_{i \in I} + \Theta(C) : = 
(c_{i}+\rho^{(i)})_{i \in I} \in \ol{\Par(\lambda)}. 
\end{equation}
Noting that $\pair{\lambda}{Q_{S}^{\vee}}=\{0\}$, 
we compute: 
\begin{align*}
& \wt (t(\zeta) \cdot \eta^{C}) = t(\zeta) (\wt (\eta^{C})) \\
& \quad 
= t(\zeta)\bigl(\lambda-|(\rho^{(i)})_{i \in I}|\delta\bigr) \quad 
  \text{by Proposition~\ref{prop:SLS}\,(2)} \\
& \quad
= \lambda - \pair{\lambda}{\zeta}\delta - |(\rho^{(i)})_{i \in I}|\delta \\[1.5mm]
& \quad
= \lambda - \pair{\lambda}{\xi_{x,\kappa(\psi)}}\delta 
  - \Bpair{\lambda}{ \sum_{i \in I} c_{i}\alpha_{i}^{\vee}}\delta
  - |(\rho^{(i)})_{i \in I}|\delta \\[1.5mm]
& \quad
= \lambda - \wt_{\lambda} \bigl(x \Rightarrow \kappa(\psi)\bigr)\delta
  - \left( \sum_{i \in I} m_{i}c_{i} \right) \delta
  - |(\rho^{(i)})_{i \in I}|\delta \\
& \quad
  = \wt (\eta_{e}) - \wt_{\lambda} \bigl(x \Rightarrow \kappa(\psi)\bigr)\delta
    - |(c_{i}+\rho^{(i)})_{i \in I}|\delta. 
\end{align*}
From this computation, together with \eqref{eq:wtep}, 
we deduce that
%
%%%%%%%%%%%%%%%
%%% eq:wtXS %%%
%%%%%%%%%%%%%%%
%
\begin{equation} \label{eq:wtXS}
\begin{split}
& \wt ( X (t(\zeta) \cdot \eta^{C}) ) = 
  \wt ( X \eta_{e} ) - \wt_{\lambda} \bigl(x \Rightarrow \kappa(\psi)\bigr)\delta
    - |(c_{i}+\rho^{(i)})_{i \in I}|\delta \\
& \quad 
  = \wt ( \eta_{\psi} ) - \wt_{\lambda} \bigl(x \Rightarrow \kappa(\psi)\bigr)\delta 
  - |(c_{i}+\rho^{(i)})_{i \in I}|\delta \\
& \quad 
  = \wt (\psi) + \Bigl\{\dg_{\lambda}(\psi) - 
    \wt_{\lambda}\bigl(x \Rightarrow \kappa(\psi)\bigr)\Bigr\}\delta
  - |(c_{i}+\rho^{(i)})_{i \in I}| \delta. 
\end{split}
\end{equation}
Because
$\dg_{\lambda}(\psi) - \wt_{\lambda} \bigl(x \Rightarrow \kappa(\psi)\bigr) 
 = \dg_{x\lambda}(\psi)$ by the definitions of $\dg_{x\lambda} (\psi)$
 and $\dg_{\lambda} (\psi)$ , 
we obtain
\begin{equation*}
\wt ( X (t(\zeta) \cdot \eta^{C}) ) 
= \wt (\psi) + \bigl(\dg_{x\lambda}(\psi) 
- |(c_{i}+\rho^{(i)})_{i \in I}|\bigr)\delta.
\end{equation*}
Summarizing, we find that
for each $\psi \in \QLS(\lambda)$, 
\begin{align*}
& \sum_{\eta \in \cl^{-1}(\psi) \cap \SLS_{\sige x}(\lambda)} e^{\wt (\eta)} 
\stackrel{\eqref{eq:grch1-1}}{=} 
  \sum_{
      \begin{subarray}{c} 
       C \in \Conn(\SLS(\lambda)) \\[1mm] 
       \zeta \in Q^{\vee}, \, [\zeta] \ge [\xi_{x,\kappa(\psi)}]
      \end{subarray}}
     e^{\wt ( X (t(\zeta) \cdot \eta^{C}) )} \\[3mm]
& \quad 
  =  e^{\wt(\psi)}e^{\dg_{x\lambda}(\psi)\delta} 
     \sum_{\brho \in \ol{\Par(\lambda)}} x^{-|\brho|\delta}
  \stackrel{(e^{\delta}=q)}{=}  e^{\wt(\psi)} q^{\dg_{x\lambda}(\psi)} 
     \sum_{\brho \in \ol{\Par(\lambda)}} q^{-|\brho|} \\[3mm]
&\quad 
 = e^{\wt(\psi)} q^{\dg_{x\lambda}(\psi)}
     \prod_{i \in I} \prod_{r=1}^{m_{i}}(1-q^{-r})^{-1}.
\end{align*}
Substituting this into \eqref{eq:ch1}, 
we finally obtain \eqref{eq:gch}. 
This completes the proof of Theorem~\ref{thm:gch}. \qed
%
%=========================%
%     START SECTION 03    %
%=========================%
%
%\section{Certain quotients of Demazure submodules.}
%\label{sec:quotient}
%
%==============================%
%     START SUBSECTION 0301    %
%==============================%
%
\subsection{Graded character formula for certain quotients of Demazure submodules}
\label{subsec:quotient}

Let $\lambda \in P^+$ be a dominant weight; 
recall that $S=S_{\lambda}=\bigl\{ i \in I \mid 
\pair{\lambda}{\alpha_{i}^{\vee}}=0 \bigr\}$.

For each $x \in W^S = W \cap (W^{S})_{\aff}$, we set
%
%%%%%%%%%%%%%
%%% eq:X1 %%%
%%%%%%%%%%%%%
%
\begin{equation} \label{eq:X1}
X_{x}^{-}(\lambda)\eqdef 
\sum_{
 \begin{subarray}{c}
  \brho \in \ol{\Par(\lambda)} \\[1.5mm]
  \brho \ne (\emptyset)_{i \in I}
 \end{subarray}
 } U_{q_s}^{-} S_{x}^{\norm} z_{\brho}v_{\lambda} = 
\sum_{
 \begin{subarray}{c}
  \brho \in \ol{\Par(\lambda)} \\[1.5mm]
  \brho \ne (\emptyset)_{i \in I}
 \end{subarray}
 } z_{\brho} \Bigl(V_{x}^{-}(\lambda)\Bigr); 
\end{equation}
for the definition of $z_{\brho}:V(\lambda) \rightarrow V(\lambda)$, 
see \eqref{eq:zrho}. 

For $\psi \in \QLS(\lambda)$, we take and 
fix a monomial $X_{\psi}$ in root operators such that 
$X_{\psi}\eta_{e}=\eta_{\psi}$, and set 
\begin{equation*}
\eta_{\psi} \cdot t(\xi) \eqdef  X_{\psi}(t(\xi) \cdot \eta_{e}) 
 \quad \text{for $\xi \in Q^{\vee}$}. 
\end{equation*}
%
%%%%%%%%%%%%%%%%%%
%%% rem:etatxi %%%
%%%%%%%%%%%%%%%%%%
%
\begin{rem} \label{rem:etatxi}
Note that $t(\xi) \cdot \eta_{e}=(\PJ(t(\xi))\,;\,0,\,1)$ 
(see \eqref{eq:xeta}). We deduce from \cite[Lemma~7.1.4]{INS} that 
if $\eta_{\psi}=X_{\psi}\eta_{e}$ is of the form
$\eta_{\psi} = (x_{1},\,\dots,\,x_{s}\,;\,\sigma_{0},\,\sigma_{1},\,\dots,\,\sigma_{s})$, 
then 
\begin{equation*}
\eta_{\psi} \cdot t(\xi) = X_{\psi}(t(\xi) \cdot \eta_{e}) = 
(x_{1}\PJ(t(\xi)),\,\dots,\,x_{s}\PJ(t(\xi))\,;\,\sigma_{0},\,\sigma_{1},\,\dots,\,\sigma_{s}).
\end{equation*}
In particular, the element $\eta_{\psi} \cdot t(\xi)$ 
does not depend on the choice of $X_{\psi}$. Also, 
since $x_u \PJ (t(\xi)) \lambda = x_u \lambda - \pair{\lambda}{\xi}\delta$
for all $1 \le u \le s$, we see by \eqref{eq:wt} that 
\begin{equation} \label{eq:etatxi1}
\begin{split}
\wt (\eta_{\psi} \cdot t(\xi)) 
& = \wt (\eta_{\psi}) - \pair{\lambda}{\xi}\delta \\
& \stackrel{\eqref{eq:wtep}}{=} 
 \wt (\psi) + \bigl(\dg_{\lambda}(\psi) - \pair{\lambda}{\xi}  \bigr)\delta,
\end{split}
\end{equation}
and that 
\begin{equation} \label{eq:etatxi2}
\cl(\eta_{\psi} \cdot t(\xi)) = \psi.
\end{equation}
\end{rem}
%
%%%%%%%%%%%%%%%%%%%%
%%% thm:quotient %%%
%%%%%%%%%%%%%%%%%%%%
%
\begin{thm} \label{thm:quotient} 
Keep the notation and setting above. 
For each $x \in W^S$, there
 exists a subset $\CB(X_{x}^{-}(\lambda))$ of $\CB(\lambda)$ such that 
%
%%%%%%%%%%%%%
%%% eq:GX %%%
%%%%%%%%%%%%%
%
\begin{equation} \label{eq:GX}
X_{x}^{-}(\lambda) = 
 \bigoplus_{b \in \CB(X_{x}^{-}(\lambda))} \BC(q_{s}) G(b).
\end{equation}
Moreover, under the isomorphism 
$\Psi_{\lambda}:\CB(\lambda) \stackrel{\sim}{\rightarrow} \SLS(\lambda)$ 
of crystals in Theorem~\ref{thm:isom}, 
the subset $\CB(X_{x}^{-}(\lambda)) \subset \CB(\lambda)$ 
is mapped to the following subset of $\SLS(\lambda)$: 
\begin{equation}
\SLS_{\sige x}(\lambda) \setminus 
 \bigl\{\eta_{\psi} \cdot t(\xi_{x,\kappa(\psi)}) \mid \psi \in \QLS(\lambda)\bigr\}. 
\end{equation}
\end{thm}
From Theorem~\ref{thm:quotient}, we immediately obtain the following corollary;
cf. \cite[Theorem 6.1.1 combined with Proposition 6.2.4]{NS-D} for the case $x=e$.
%
%%%%%%%%%%%%%%%%%%%%
%%% cor:quotient %%%
%%%%%%%%%%%%%%%%%%%%
%
\begin{cor} \label{cor:quotient}
For each $x \in W^S$, there holds the equality
\begin{equation}
\gch ( V_{x}^{-}(\lambda)/X_{x}^{-}(\lambda) ) = 
 \sum_{\psi \in \QLS(\lambda)} e^{\wt(\psi)}q^{\dg_{x\lambda}(\psi)}.
\end{equation}
%where $\dg_{x\lambda}(\psi) = \dg_{x\lambda}(\psi)$ in the notation of \S $3$.
%
\end{cor}
%%%%%%%%%%%%%%%%%%%%
%%% begin 5 %%%
%%%%%%%%%%%%%%%%%%%%
By combing the special case $x = \lfloor \lon \rfloor \in W^S$ of Corollary~\ref{cor:quotient}
with the special case $\mu = \lon \lambda$ of Theorem~\ref{main}, we obtain the equality:
\begin{equation*}
\gch ( V_{\lon}^{-}(\lambda)/X_{\lon}^{-}(\lambda) ) = E_{\lon \lambda}(q, \infty). 
\end{equation*}
%%%%%%%%%%%%%%%%%%%%
%%% end 5 %%%
%%%%%%%%%%%%%%%%%%%%
%
%==============================%
%     START SUBSECTION 0302    %
%==============================%
%
\subsection{Proof of Theorem~\ref{thm:quotient}}
\label{subsec:prf-quotient}
%
%%%%%%%%%%%%%%%
%%% lem:Be- %%%
%%%%%%%%%%%%%%%
%
\begin{lem}[cf. \eqref{eq:CBlam}] \label{lem:Be-}
Let $x \in W^{S}$. Then, we have
%
%%%%%%%%%%%%%%
%%% eq:Be- %%%
%%%%%%%%%%%%%%
%
\begin{equation} \label{eq:Be-1}
\CB_{x}^{-}(\lambda) = \bigl\{
 z_{\brho}b \mid 
 \brho \in \Par(\lambda),\,
 b \in \CB_{x}^{-}(\lambda) \cap \CB_{0}(\lambda) \bigr\}.
\end{equation}
Moreover, for every 
$\brho \in \ol{\Par(\lambda)}$ and 
$b \in \CB_{x}^{-}(\lambda) \cap \CB_{0}(\lambda)$, 
the element $z_{\brho}b$ is contained in $\CB_{x}^{-}(\lambda)$. 
\end{lem}

\begin{proof}
We first prove the inclusion $\supset$. 
Let $b \in \CB_{x}^{-}(\lambda) \cap \CB_{0}(\lambda)$, and 
write it as $b = X u_{\lambda}$ for a monomial $X$ in Kashiwara operators. 
For $\brho \in \Par(\lambda)$, we have $z_{\brho}b=X z_{\brho}u_{\lambda}=Xu^{\brho}$ 
since $z_{\brho}$ commutes with Kashiwara operators (see \S\ref{subsec:ext}). 
Now we set $\eta : = \Psi_{\lambda}(b)$ and $\eta' : = \Psi_{\lambda}(z_{\brho}b)$, 
where $\Psi_{\lambda}:\CB(\lambda) \stackrel{\sim}{\rightarrow} \SLS(\lambda)$ 
is the isomorphism of crystals in Theorem~\ref{thm:isom}. 
Then, we have $\eta = X \eta_{e}$ and 
$\eta' = X \Psi_{\lambda}(u^{\brho})=X\eta^{C}$, 
with $C\eqdef \Theta^{-1}(\brho) \in \Conn(\SLS(\lambda))$. 
Therefore, noting that $\kappa(\eta^{C})=e$, 
we deduce from Lemma~\ref{lem:INS714} that 
$\kappa(\eta') = \kappa(\eta) \kappa(\eta^C) = \kappa(\eta)$. 
Also, since $b \in \CB_{x}^{-}(\lambda)$, it follows that 
$\kappa(\eta) \sige x$, and hence $\kappa(\eta') = \kappa(\eta) \sige x$. 
Hence we obtain $\eta' \in \SLS_{\sige x}(\lambda)$, which implies that 
$z_{\brho}b \in \CB_{x}^{-}(\lambda)$. 

Next we prove the opposite inclusion $\subset$. 
Let $b' \in \CB_{x}^{-}(\lambda)$, and write it as 
$b' = z_{\brho}b$ for some $\brho \in \Par(\lambda)$ and 
$b \in \CB_{0}(\lambda)$ (see \eqref{eq:CBlam}); 
we need to show that $b \in \CB_{x}^{-}(\lambda)$. 
We set $\eta : = \Psi_{\lambda}(b) \in \SLS(\lambda)$ and 
$\eta' : = \Psi_{\lambda}(b') \in \SLS(\lambda)$. 
Then, the same argument as above shows that 
$\kappa(\eta) = \kappa(\eta') \sige x$. 
Hence we obtain $\eta \in \SLS_{\sige x}(\lambda)$, 
which implies that $b \in \CB_{x}^{-}(\lambda)$. 

For the second assertion, 
let $\brho = (\rho^{(i)})_{i \in I} \in \ol{\Par(\lambda)}$ and $b \in 
\CB_{x}^{-}(\lambda) \cap \CB_{0}(\lambda)$; remark that 
\begin{equation*}
z_{\brho}b \in \CB_{x}^{-}(\lambda) \iff 
\Psi_{\lambda}(z_{\brho}b) \in \SLS_{\sige x}(\lambda) \iff 
\kappa(\Psi_{\lambda}(z_{\brho}b)) \sige x.
\end{equation*}
We write $b$ as $b=Xu_{\lambda}$ 
for a monomial $X$ in Kashiwara operators. 
Also, define $\bvrho\eqdef (\varrho^{(i)})_{i \in I} \in \Par(\lambda)$ 
and $\xi \eqdef  \sum_{i \in I} c_{i}\alpha_{i}^{\vee} \in Q^{\vee,+}$ 
as in Remark~\ref{rem:par}. Then it follows that 
$z_{\brho}b = z_{\brho}Xu_{\lambda} = X z_{\brho}u_{\lambda} 
\stackrel{\eqref{eq:zcu}}{=} X ( t(\xi) \cdot u^{\bvrho} )$.
If we set $C\eqdef \Theta^{-1}(\bvrho) \in \Conn(\SLS(\lambda))$, then we have
\begin{equation*}
\Psi_{\lambda}(z_{\brho}b) 
  = \Psi_{\lambda}\bigl( X ( t(\xi) \cdot u^{\bvrho} ) \bigr) 
  = X \bigl(t(\xi) \cdot \Psi_{\lambda}(u^{\bvrho})\bigr)
  = X\bigl( t(\xi) \cdot \eta^{C}\bigr);
\end{equation*}
note that $t(\xi) \cdot \eta^{C}$ is of the form \eqref{eq:pie} 
with $\kappa(t(\xi) \cdot \eta^{C})=\PJ(t(\xi))$ by 
Remark~\ref{rem:extp} and the fact that $\kappa(\eta^{C})=e$. 
Therefore, we see from Lemma~\ref{lem:INS714} that
\begin{equation} \label{eq:kappa1}
\kappa(\Psi_{\lambda}(z_{\brho}b))= 
\kappa(X(t(\xi) \cdot \eta^{C}))=\kappa(X\eta_{e})\PJ(t(\xi)).
\end{equation} 
Here we recall that $\kappa(X\eta_{e}) \sige x$ since 
$b \in \CB_{x}^{-}(\lambda) \cap \CB_{0}(\lambda)$. 
Also, recall that $\xi \in Q^{\vee,+}$. From these, we deduce that 
\begin{align*}
\kappa(\Psi_{\lambda}(z_{\brho}b))
 & = \kappa(X\eta_{e})\PJ(t(\xi)) 
 \sige \kappa(X\eta_{e}) \quad \text{by Lemma~\ref{lem:sig1}} \\
 & \sige x. 
\end{align*}
This proves the lemma. 
\end{proof}

\begin{proof}[Proof of Theorem~\ref{thm:quotient}.]
We will prove that if we set
%
%%%%%%%%%%%%%%
%%% eq:GX1 %%%
%%%%%%%%%%%%%%
%
\begin{equation} \label{eq:GX1}
\CB \eqdef  
 \bigl\{z_{\brho}b \mid 
 \brho \in \ol{\Par(\lambda)} \setminus (\emptyset)_{i \in I},\,
 b \in \CB_{x}^{-}(\lambda) \cap \CB_{0}(\lambda) \bigr\} \subset \CB(\lambda),
\end{equation}
then 
%
%%%%%%%%%%%%%%
%%% eq:GX2 %%%
%%%%%%%%%%%%%%
%
\begin{equation} \label{eq:GX2}
X_{x}^{-}(\lambda) =
\bigoplus_{b \in \CB} \BC(q_{s}) G(b).
\end{equation}
We first show the inclusion $\supset$ in \eqref{eq:GX2}. 
Let $\brho \in \ol{\Par(\lambda)} \setminus (\emptyset)_{i \in I}$ and 
$b \in \CB_{x}^{-}(\lambda) \cap \CB_{0}(\lambda)$. 
We see from Remark~\ref{rem:zGb} that 
$G(z_{\brho}b) = z_{\brho} G(b)$. 
Since $G(b) \in V_{x}^{-}(\lambda)$ and 
\begin{equation*}
X_{x}^{-}(\lambda)=
\sum_{
 \begin{subarray}{c}
  \brho \in \ol{\Par(\lambda)} \\[1.5mm]
  \brho \ne (\emptyset)_{i \in I}
 \end{subarray}
 } z_{\brho} \Bigl(V_{x}^{-}(\lambda)\Bigr)
\end{equation*}
by the definition, we have 
$G(z_{\brho}b) = z_{\brho} G(b) \in X_{x}^{-}(\lambda)$. 
Thus we have shown the inclusion $\supset$ in \eqref{eq:GX2}. 
Next we show the opposite inclusion $\subset$ in \eqref{eq:GX2}. 
Since $\bigl\{G(b) \mid b \in \CB_{x}^{-}(\lambda)\bigr\}$
is a $\BC(q_{s})$-basis of $V_{x}^{-}(\lambda)$, 
we deduce from \eqref{eq:X1} that
%
%%%%%%%%%%%%%%%%
%%% eq:Xspan %%%
%%%%%%%%%%%%%%%%
%
\begin{equation} \label{eq:Xspan}
X_{x}^{-}(\lambda) = \Span_{\BC(q_{s})}
\bigl\{
 z_{\brho}G(b) \mid 
 \brho \in \ol{\Par(\lambda)} \setminus (\emptyset)_{i \in I},\,
 b \in \CB_{x}^{-}(\lambda)
\bigr\}.
\end{equation}
Let $\brho \in \ol{\Par(\lambda)} \setminus (\emptyset)_{i \in I}$ 
and $b \in \CB_{x}^{-}(\lambda)$. By Lemma~\ref{lem:Be-}, 
we can write the $b$ as $b = z_{\brho'}b'$ for some 
$\brho' \in \Par(\lambda)$ and $b' \in \CB_{x}^{-}(\lambda) \cap 
\CB_{0}(\lambda)$. It follows that $z_{\brho}b = z_{\brho} z_{\brho'}b'$. 
Because $z_{\brho}$ and $z_{\brho'}$ are defined to be a certain product of 
Schur polynomials (see \eqref{eq:Schur}), 
the element $z_{\brho} z_{\brho'}$ can be expressed as:
\begin{equation*}
z_{\brho} z_{\brho'} = 
 \sum_{
   \begin{subarray}{c}
     \brho'' \in \ol{\Par(\lambda)} \\[1.5mm]
     |\brho''|=|\brho|+|\brho'|
   \end{subarray}
 } n_{\brho''} z_{\brho''}, \quad 
\text{with $n_{\brho''} \in \BZ$}; 
\end{equation*}
here we remark that $|\brho|+|\brho'| \ge 1$ since 
$\brho \ne (\emptyset)_{i \in I}$. Therefore, we deduce that
\begin{align*}
z_{\brho}G(b) & = z_{\brho}G(z_{\brho'}b') = z_{\brho}z_{\brho'}G(b') \\
& = 
 \sum_{
   \begin{subarray}{c}
     \brho'' \in \ol{\Par(\lambda)} \\[1.5mm]
     |\brho''|=|\brho|+|\brho'|
   \end{subarray}
 } n_{\brho''} G(z_{\brho''}b') \in 
\bigoplus_{b \in \CB} \BC(q_{s}) G(b).
\end{align*}
From this, together with \eqref{eq:Xspan}, we obtain
the inclusion $X_{x}^{-}(\lambda) \subset 
\bigoplus_{b \in \CB} \BC(q_{s}) G(b)$ in \eqref{eq:GX2}.
Thus, we obtain \eqref{eq:GX2}, 
as desired; in what follows, we write $\CB(X_{x}^{-}(\lambda))$ 
for the subset $\CB \subset \CB(\lambda)$ in \eqref{eq:GX1}. 

Furthermore, we will prove that 
\begin{equation*}
\Psi_{\lambda}\bigl(\CB(X_{x}^{-}(\lambda))\bigr) = 
\SLS_{\sige x}(\lambda) \setminus 
\bigl\{\eta_{\psi} \cdot t(\xi_{x,\kappa(\psi)}) 
\mid \psi \in \QLS(\lambda)\bigr\}. 
\end{equation*}
For this purpose, 
it suffices to show that for each $\psi \in \QLS(\lambda)$, 
%
%%%%%%%%%%%%%
%%% eq:X3 %%%
%%%%%%%%%%%%%
%
\begin{equation} \label{eq:X3}
\cl^{-1}(\psi) \cap 
\Psi_{\lambda}\bigl(\CB(X_{x}^{-}(\lambda))\bigr) = 
\Bigl(\cl^{-1}(\psi) \cap \SLS_{\sige x}(\lambda)\Bigr) 
 \setminus \bigl\{\eta_{\psi} \cdot t(\xi_{x,\kappa(\psi)}) \bigr\}.
\end{equation}
Let $\psi \in \QLS(\lambda)$; recall that 
$X_{\psi}$ is a monomial in root operators such that 
$\eta_{\psi}=X_{\psi} \eta_{e}$. Then we know from 
\eqref{eq:grch1-1} that
\begin{align*}
& \cl^{-1}(\psi) \cap \SLS_{\sige x}(\lambda) \\
& \quad = 
\bigl\{X_{\psi} (t(\zeta) \cdot \eta^{C}) \mid 
 C \in \Conn(\SLS(\lambda)),\,
 \zeta \in Q^{\vee},\,[\zeta] \ge [\xi_{x,\kappa(\psi)}] \bigr\}.
\end{align*}
We first show the inclusion $\supset$ in \eqref{eq:X3}. 
Let $\eta$ be an element in the set on the right-hand side of \eqref{eq:X3}, 
and write it as
$\eta = X_{\psi} (t(\zeta) \cdot \eta^{C})$ 
for some $C \in \Conn(\SLS(\lambda))$ and $\zeta \in Q^{\vee}$ 
such that $[\zeta] \ge [\xi_{x,\kappa(\psi)}]$. 
We write the difference $[\zeta]-[\xi_{x,\kappa(\psi)}] \in Q^{\vee,+}$ as 
$[\zeta]-[\xi_{x,\kappa(\psi)}]=
\sum_{i \in I} c_{i}\alpha_{i}^{\vee}$ with $c_{i} \in \BZ_{\ge 0}$ for $i \in I$ 
(note that $c_{i}=0$ for all $i \in S$), and  
define 
$\brho\eqdef (c_{i})_{i \in I} + \Theta(C) \in \ol{\Par(\lambda)}$ 
as in \eqref{eq:par+}.
We claim that $\brho \ne (\emptyset)_{i \in I}$. 
Suppose, for a contradiction, that $\brho = (\emptyset)_{i \in I}$. 
Then we have $\Theta(C)=(\emptyset)_{i \in I}$ and 
$c_{i}=0$ for all $i \in I$, and hence
\begin{align*}
\eta & = X_\psi (t(\zeta) \cdot \eta^{C}) 
 = X_\psi (t(\zeta) \cdot \eta_{e}) 
 = X_\psi (\PJ(t(\zeta))\,;\,0,\,1) \\
& = X_\psi (\PJ(t(\xi_{x,\kappa(\psi)}))\,;\,0,\,1)
\quad \text{since $[\zeta]=[\xi_{x,\kappa(\psi)}]$} \\
& = X_\psi (t(\xi_{x,\kappa(\psi)}) \cdot \eta_{e}) 
  = \eta_{\psi} \cdot t(\xi_{x,\kappa(\psi)}), 
\end{align*}
which contradicts the assumption that $\eta$ is an element in the set 
on the right-hand side of \eqref{eq:X3}. Thus we obtain 
$\brho \ne (\emptyset)_{i \in I}$. Now, we set 
\begin{equation*}
b\eqdef \Psi_{\lambda}^{-1}(\eta_{\psi} \cdot t(\xi_{x,\kappa(\psi)}))
  =\Psi_{\lambda}^{-1}\bigl( X_\psi (t(\xi_{x,\kappa(\psi)}) \cdot \eta_{e}) \bigr)
  \in \CB_{x}^{-}(\lambda) \cap \CB_{0}(\lambda);
\end{equation*}
note that $\eta_{\psi} \cdot t(\xi_{x,\kappa(\psi)}) \in 
\SLS_{\sige x}(\lambda)$ by \eqref{eq:grch1-1}, and that 
$b=X_\psi (t(\xi_{x,\kappa(\psi)}) \cdot u_{\lambda})$. 
Then we see by \eqref{eq:GX1} that $z_{\brho}b \in \CB(X_{x}^{-}(\lambda))$. 
Also, we have
\begin{align*}
z_{\brho}b & = z_{\brho} \bigl( X_\psi (t(\xi_{x,\kappa(\psi)}) \cdot u_{\lambda}) \bigr) = 
X_\psi \bigl( t(\xi_{x,\kappa(\psi)}) \cdot (z_{\brho}u_{\lambda}) \bigr) \\
& = X_\psi \bigl( 
        t(\xi_{x,\kappa(\psi)}) \cdot 
        t([\zeta]-[\xi_{x,\kappa(\psi)}]) \cdot u^{\Theta(C)}\bigr)
  \quad \text{by Remark~\ref{rem:par}} \\
& = X_\psi ( t(\zeta+\gamma) \cdot u^{\Theta(C)}) 
  \quad \text{for some $\gamma \in Q_{S}^{\vee}$} \\
& = X_\psi ( t(\zeta) \cdot u^{\Theta(C)}).
\end{align*}
Therefore, $\Psi_{\lambda}(z_{\brho}b) = X_\psi (t(\zeta) \cdot \eta^{C}) = \eta$, 
which implies that $\eta$ is contained in $\Psi_{\lambda}(\CB(X_{x}^{-}(\lambda)))$. 
Thus we have shown the inclusion $\supset$ in \eqref{eq:X3}. 

Next we show the opposite inclusion $\subset$ in \eqref{eq:X3}. 
Since $\CB(X_{x}^{-}(\lambda)) \subset \CB_{x}^{-}(\lambda)$, 
it follows that 
\begin{equation*}
\cl^{-1}(\psi) \cap 
 \Psi_{\lambda}\bigl(\CB(X_{x}^{-}(\lambda))\bigr) \subset
\cl^{-1}(\psi) \cap \SLS_{\sige x}(\lambda). 
\end{equation*}
Hence it suffices to show that 
$\eta_{\psi} \cdot t(\xi_{x,\kappa(\psi)}) 
\not\in \Psi_{\lambda}\bigl(\CB(X_{x}^{-}(\lambda))\bigr)$. 
Suppose, for a contradiction, that 
there exists $b' \in \CB(X_{x}^{-}(\lambda))$ 
such that $\Psi_{\lambda}(b') = \eta_{\psi} \cdot t(\xi_{x,\kappa(\psi)})$. 
By \eqref{eq:GX1}, we can write it as $b'= z_{\brho}b$ 
for some $\brho \in \ol{\Par(\lambda)} \setminus (\emptyset)_{i \in I}$ and 
$b \in \CB_{x}^{-}(\lambda) \cap \CB_{0}(\lambda)$. 
We set $\eta\eqdef \Psi_{\lambda}^{-1}(b) \in 
\SLS_{\sige x}(\lambda) \cap \SLS_{0}(\lambda)$, 
and write $\kappa(\eta) \in (W^{S})_{\aff}$ as 
$\kappa(\eta)=yz_{\xi}t(\xi)$ for some $y \in W^{S}$ and $\xi \in \Qad$. 
Then, $\kappa(\eta)=yz_{\xi}t(\xi) \sige x$ 
since $\eta \in \SLS_{\sige x}(\lambda)$, and hence
\begin{equation} \label{eq:ge}
[\xi] \ge [\xi_{x,y}] \qquad \text{by Lemma~\ref{lem:min2}}. 
\end{equation}
Let us write $b$ as $b=Yu_{\lambda}$ 
for some monomial $Y$ in Kashiwara operators (note that $\eta =Y\eta_{e}$),
and define $\zeta = \sum_{i \in I} c_{i}\alpha_{i}^{\vee} \in Q^{\vee,+}$ and 
$\bvrho=(\varrho^{(i)})_{i \in I} \in \Par(\lambda)$ in such a way that 
$\brho = (c_{i})_{i \in I} + \bvrho$ (see Remark~\ref{rem:par} and \eqref{eq:par+}); 
note that $c_{i}=0$ for all $i \in S$. 
Then, by \eqref{eq:zcu}, we have
\begin{equation*}
b' = z_{\brho}b = z_{\brho}Yu_{\lambda} 
   = Yz_{\brho}u_{\lambda} = Y (t(\zeta) \cdot u^{\bvrho}).
\end{equation*}
Therefore, we see that 
\begin{equation} \label{eq:cep}
\begin{split}
& \eta_{\psi} \cdot t(\xi_{x,\kappa(\psi)}) = \Psi_{\lambda}(b') = 
  \Psi_{\lambda}\bigl( Y(t(\zeta) \cdot u^{\bvrho}) \bigr) = 
  Y ( t(\zeta) \cdot \eta^{C}), \\
& \hspace*{10mm} 
\text{with $C\eqdef \Theta^{-1}(\bvrho) \in \Conn(\SLS(\lambda))$}. 
\end{split}
\end{equation}
Since $\eta_{\psi} \cdot t(\xi_{x,\kappa(\psi)}) = 
X_{\psi}(t(\xi_{x,\kappa(\psi)}) \cdot \eta_{e}) \in \SLS_{0}(\lambda)$, 
it follows that $\eta^{C} = \eta_{e}$, and 
hence $\bvrho = (\emptyset)_{i \in I}$. 
Hence we obtain $\eta_{\psi} \cdot t(\xi_{x,\kappa(\psi)}) = 
Y ( t(\zeta) \cdot \eta_{e})$. 
Since $t(\zeta) \cdot \eta_{e} = ( \PJ(t(\zeta))\,;\,0,\,1)$, 
we see from Lemma~\ref{lem:INS714} that $\kappa(Y(t(\zeta) \cdot \eta_{e})) = 
\kappa(\eta)\kappa(t(\zeta) \cdot \eta_{e}) = y z_{\xi}t(\xi)\PJ(t(\zeta))$. 
Similarly, we see that $\kappa(\eta_{\psi} \cdot t(\xi_{x,\kappa(\psi)})) = 
\kappa(\psi)\PJ(t(\xi_{x,\kappa(\psi)}))$. Combining these equalities, 
we obtain
$\kappa(\psi)\PJ(t(\xi_{x,\kappa(\psi)})) = y z_{\xi}t(\xi)\PJ(t(\zeta))$, 
and hence ($y=\kappa(\psi)$ and) $[\zeta+\xi]=[\xi_{x,\kappa(\psi)}]$. 
Since $[\xi] \ge [\xi_{x,y}]$ by \eqref{eq:ge} and $\zeta \in Q^{\vee,+}$, 
it follows that ($[\xi] = [\xi_{x,y}]$ and) $[\zeta]=0$, 
which implies that $c_{i}=0$ for all $i \in I \setminus S$; 
recall that $c_{i}=0$ for all $i \in S$ by the definition. 
Therefore, we conclude that $\brho = (c_{i})_{i \in I} + \bvrho 
= (\emptyset)_{i \in I}$; 
this contradicts our assumption that 
$\brho \in \ol{\Par(\lambda)} \setminus (\emptyset)_{i \in I}$. 
Thus we have shown the inclusion $\subset$. 
This completes the proof of Theorem~\ref{thm:quotient}. 
\end{proof}


\begin{thebibliography}{99}

\bibitem[BB]{BB}
A. Bj\"{o}rner and F. Brenti, 
``Combinatorics of Coxeter Groups'', 
Graduate Texts in Mathematics Vol.~231, 
Springer, New York, 2005.

\bibitem[BFP]{BFP}
F. Brenti, S. Fomin, and A. Postnikov,
Mixed Bruhat operators and Yang-Baxter equations for Weyl groups, 
{\it Int. Math. Res. Not.} {\bf 8} (1999), 419--441.

\bibitem[BN]{BN}
J. Beck and H. Nakajima, 
Crystal bases and two-sided cells of quantum affine algebras, 
{\it Duke Math. J.} {\bf 123} (2004), 335--402. 

%\bibitem[Ca]{Ca}
%R. Carter,
%Lie Algebras of Finite and Affine Type,
%Cambridge Studies in Advanced Mathematics, vol. 96,
%Cambridge University Press, Cambridge, 2005.

%\bibitem[Ch]{C}
%I. Cherednik,
%Double Affine Hecke Algebras,
%London Mathematical Society Lecture
%Note Series, vol. 319, Cambridge University Press, Cambridge, 2006.

\bibitem[Ch1]{C1}
I. Cherednik, 
Double affine Hecke algebras,
Knizhnik-Zamolodchikov equations, and Macdonald's operators,
{\it Int. Math. Res. Not.} {\bf 9} (1992), 171--179.


\bibitem[Ch2]{C3}
I. Cherednik, 
Non-symmetric Macdonald polynomials, 
{\it Int. Math. Res. Not.} {\bf 10} (1995), 483--515.


\bibitem[CO]{CO}
I. Cherednik and D. Orr,
Nonsymmetric difference Whittaker functions,
{\it Math. Z.} {\bf 279} (2015), 
%no. 3-4, 
879--938.

\bibitem[FM]{FM}
E. Feigin and I. Makedonskyi,
Generalized Weyl modules, alcove paths and Macdonald polynomials,
preprint 2015, 
arXiv:1512.03254.

%\bibitem[HHL1]{HHL1}
%J. Haglund, M. Haiman, and N. Loehr,
%A combinatorial formula for Macdonald polynomials,
%J. Amer. Math. Soc., 18 (2005), no. 3, 735-761.
%
%\bibitem[HHL2]{HHL2}
%J. Haglund, M. Haiman,  and N. Loehr,
%A combinatorial formula for non-symmetric Macdonald polynomials,
%Amer. J. Math., 130 (2008), 
%no. 2, 
%359-383. 

\bibitem[HK]{HK}
J. Hong and S.-J. Kang, 
``Introduction to Quantum Groups and Crystal Bases'', 
Graduate Studies in Mathematics Vol.~42, 
Amer. Math. Soc., Providence, RI, 2002.

\bibitem[I]{I}
B. Ion. Nonsymmetric Macdonald polynomials and Demazure characters,
{\it Duke Math. J.} {\bf 116} (2003), 
%no. 2, 
299--318.


\bibitem[INS]{INS}
M. Ishii, S. Naito, and D. Sagaki, 
Semi-infinite Lakshmibai-Seshadri path model 
for level-zero extremal weight modules over quantum affine algebras,
{\it Adv. Math.} {\bf 290} (2016), 967--1009. 

\bibitem[Kac]{Kac}
V. G. Kac, 
``Infinite Dimensional Lie Algebras'', 3rd Edition, 
Cambridge University Press, Cambridge, UK, 1990.

\bibitem[Kas1]{K-mod}
M. Kashiwara, Crystal bases of modified quantized 
enveloping algebra, 
{\it Duke Math. J.} {\bf 73} (1994), 383--413. 

\bibitem[Kas2]{K-lv0}
M. Kashiwara, On level-zero representations of 
quantized affine algebras, 
{\it Duke Math. J.} {\bf 112} (2002), 117--175.

\bibitem[Kas3]{K-rims}
M. Kashiwara, Level zero fundamental representations 
over quantized affine algebras and Demazure modules, 
{\it Publ. Res. Inst. Math. Sci.} {\bf 41} (2005), 223--250.

\bibitem[Kat]{Kat}
S. Kato,
Demazure character formula for semi-infinite flag manifolds,
preprint 2016.

\bibitem[L]{L1}
P. Littelmann,
A Littlewood-Richardson rule for symmetrizable Kac-Moody algebras,
{\it Invent. Math.} {\bf 116} (1994), 
%no. 1, 
329--346.

%\bibitem[L2]{L2}
%P. Littelmann, 
%Paths and root operators in representation theory, 
%Ann.  Math., 142 (1995),
%no. 3,
%499-525.

%\bibitem[LNSSS1.1]{1}
%C. Lenart, S. Naito, D. Sagaki, A. Schilling, and M. Shimozono,
%A uniform model for Kirillov-Reshetikhin crystals I: Lifting the parabolic quantum Bruhat graph,
%Int. Math. Res. Not., 2015 (2015), 1848-1901.

\bibitem[LNSSS1]{LNSSS1}
C. Lenart, S. Naito, D. Sagaki, A. Schilling, and M. Shimozono, 
A uniform model for Kirillov-Reshetikhin crystals I: 
Lifting the parabolic quantum Bruhat graph, 
{\it Int. Math. Res. Not.} {\bf 2015} (2015), 1848--1901. 

%\bibitem[LNSSS2.1]{2}
%C. Lenart, S. Naito, D. Sagaki, A. Schilling, and M. Shimozono,
%A uniform model for Kirillov-Reshetikhin crystals II: Alcove model, path model, and $P=X$,
%preprint
%2014,
%arXiv:1402.2203.

\bibitem[LNSSS2]{LNSSS2}
C. Lenart, S. Naito, D. Sagaki, A. Schilling, and M. Shimozono,
A uniform model for Kirillov-Reshetikhin crystals I\hspace{-1.5pt}I: Alcove model, path model, and $P=X$,
preprint
2014,
arXiv:1402.2203.

%\bibitem[LNSSS3.9]{3}
%C. Lenart, S. Naito, D. Sagaki, A. Schilling, and M. Shimozono,
%A uniform model for Kirillov-Reshetikhin crystals III: 
%Nonsymmetric Macdonald polynomials at $t=0$
%and level-zero Demazure characters,
%preprint
%2015,
%arXiv:1511.00465.

\bibitem[LNSSS3]{LNSSS3}
C. Lenart, S. Naito, D. Sagaki, A. Schilling, and M. Shimozono,
A uniform model for Kirillov-Reshetikhin crystals I\hspace{-1.5pt}I\hspace{-1.5pt}I: 
Nonsymmetric Macdonald polynomials at $t=0$
and level-zero Demazure characters,
preprint
2015,
arXiv:1511.00465.

%\bibitem[LNSSS4.5]{4}
%C. Lenart, S. Naito, D. Sagaki, A. Schilling, and M. Shimozono,
%Quantum Lakshmibai-Seshadri paths and root operators,
%preprint
%2013,
%arXiv:1308.3529.

\bibitem[LNSSS4]{LNSSS4}
C. Lenart, S. Naito, D. Sagaki, A. Schilling, and M. Shimozono,
Quantum Lakshmibai-Seshadri paths and root operators,
preprint
2013,
arXiv:1308.3529.

%\bibitem[LP1]{LP1}
%C. Lenart and A. Postnikov,
%Affine Weyl groups in K-theory and representation theory,
%Int. Math. Res. Not., 2007(2007), 1-65.
%
%\bibitem[LP2]{LP2}
%C. Lenart and A. Postnikov,
%A combinatorial model for crystals of Kac-Moody algebras,
%Trans. Amer. Math. Soc., 360 (2008), 4349-4381.

\bibitem[LS]{LS10}
T. Lam and M. Shimozono, 
Quantum cohomology of $G/P$ and homology of 
affine Grassmannian, 
{\it Acta Math.} {\bf 204} (2010), 49--90.

\bibitem[M]{M}
I. G. Macdonald,
``Affine Hecke Algebras and Orthogonal Polynomials'',
Cambridge Tracts in Mathematics Vol. 157,
Cambridge University Press, Cambridge, UK, 2003.


\bibitem[M1]{M1}
I. G. Macdonald,
A New Class of Symmetric Functions,
Publ. I.R.M.A., 
Strasbourg, Actes 20-e Seminaire Lotharingen, 1988, 131--171.

%\bibitem[N]{Nakajima}
%H. Nakajima, 
%Extremal weight modules of quantum affine algebras,
%Advanced Studies in Pure Math., 40 (2004), 343-369.

\bibitem[NS1]{NS}
S. Naito and D. Sagaki, 
Crystal of Lakshmibai-Seshadri paths associated to an integral
weight of level zero for an affine Lie algebra, 
{\it Int. Math. Res. Not.} {\bf 14} (2005), 815--840.

\bibitem[NS2]{NS-Deg}
S. Naito and D. Sagaki, 
Lakshmibai-Seshadri paths of a level-zero weight shape
and one-dimensional sums associated to
level-zero fundamental representations, 
{\it Compos. Math.} {\bf 144} (2008), 1525--1556. 

\bibitem[NS3]{NS-D}
S. Naito and D. Sagaki, Demazure submodules of level-zero extremal weight modules 
and specializations of Macdonald polynomials, Math. Z.,
published online:
DOI 10.1007/s00209-016-1628-7.

\bibitem[OS]{OS}
D. Orr and M. Shimozono,
Specialization of nonsymmetric Macdonald-Koornwinder polynomials,
preprint
2013,
arXiv:1310.0279.

\bibitem[Pa]{papi}
P. Papi,
A characterization of a special ordering in a root system,
{\it Proc. Amer. Math. Soc.} {\bf 120} (1994), 
%no. 3, 
661--665.

\bibitem[Pe]{Pet97}
D. Peterson, Quantum Cohomology of $G/P$, Lecture Notes, 
Cambridge, MA, Spring: Massachusetts Institute of Technology, 1997.

\bibitem[Po]{Po}
A. Postnikov,
Quantum Bruhat graph and Schubert polynomials,
{\it Proc. Amer. Math. Soc.} {\bf 133} (2005), 
%no. 3, 
699--709. 

\bibitem[RY]{RY}
A. Ram and M. Yip,
A combinatorial formula for Macdonald polynomials,
{\it Adv. Math.} {\bf 226} (2011), 
%no. 1, 
309--331.




\end{thebibliography}
\end{document}